\documentclass[a4paper, 12pt]{amsart}

\usepackage{amsmath,amssymb,enumerate,comment,amsthm,mathrsfs,bbm,mathtools,enumitem}
 \usepackage{comment,amsfonts,amscd}
 \usepackage[bbgreekl]{mathbbol}

 \usepackage{tikz}
 \usepackage[all]{xy}
\usepackage{amscd}
\usepackage{array}
\newcommand{\ip}[1]{\mathopen{\langle}#1\mathclose{\rangle}}
\newtheorem{Thm}{Theorem}[section]
\newtheorem{Prop}[Thm]{Proposition}

\newtheorem{Lem}[Thm]{Lemma}

\newtheorem{Cor}[Thm]{Corollary}

\theoremstyle{definition}

\newtheorem{Rem}[Thm]{Remark}

\newtheorem{Def}[Thm]{Definition}

\newtheorem{no}[Thm]{Notations}

\newtheorem{ass}[Thm]{Assumption}

\newtheorem{Exm}[Thm]{Example}

\newcommand{\Cs}{C$^\ast$}

\newcommand{\id}{\mbox{\rm id}}
\newcommand{\rg}{\mathop{{\mathrm C}_{\mathrm r}^\ast}}

\newcommand{\rca}[1]{\mathop{\rtimes _{#1, {\mathrm r}}}}
\newcommand{\Cn}[1]{\mathcal{O}_{#1}}

\newcommand{\IB}{\mathbb B}
\newcommand{\IC}{\mathbb C}

\newcommand{\IK}{\mathbb K}

\newcommand{\IN}{\mathbb N}
\newcommand{\IR}{\mathbb R}
\newcommand{\IT}{\mathbb T}

\newcommand{\IZ}{\mathbb Z}

\newcommand{\cF}{\mathcal F}
\newcommand{\cH}{\mathcal H}

\newcommand{\cG}{\mathcal G}

\newcommand{\cS}{\mathcal S}
\newcommand{\cU}{\mathcal U}

\newcommand{\fn}{\mathfrak n}
\newcommand{\fA}{\mathfrak A}

\newcommand{\fB}{\mathfrak B}

\newcommand{\cE}{\mathcal E}
\newcommand{\cZ}{\mathcal Z}

\newcommand{\cM}{\mathcal M}
\newcommand{\ve}{\varepsilon}

\newcommand{\acts}{\curvearrowright}

\newcommand{\ad}{\mathrm{Ad}}

\DeclareMathOperator{\cl}{\mathop{cl}}
\DeclareMathOperator{\SU}{\mathop{SU}}

\DeclareMathOperator{\csp}{\mathop{\overline{span}}}
\DeclareMathOperator{\spa}{\mathop{span}}

\DeclareMathOperator{\supp}{supp}

\DeclareMathOperator{\SL}{SL}
\DeclareMathOperator{\Ped}{Ped}

\DeclareMathOperator{\Cso}{{\rm C}^\ast}

\newcommand{\ONB}{\mathrm{ONB}}

\makeatletter
\@addtoreset{equation}{section}
\makeatother

\title[]{Stabilization theorem and symmetric structure of Cuntz--Pimsner algebras}
\author{Miho Mukohara}
\author{Yuhei Suzuki}

\address{Miho Mukohara, Department of Mathematics, Kyushu University, Japan}
\email{mukohara@math.kyushu-u.ac.jp}
\address{Yuhei Suzuki, Department of Mathematics, Faculty of Science, Hokkaido University,
Kita 10, Nishi 8, Kita-Ku, Sapporo, Hokkaido, 060-0810, Japan}
\email{yuhei@math.sci.hokudai.ac.jp}
\begin{document}
\maketitle

\begin{abstract}

We establish a crossed product decomposition theorem for stabilized Cuntz--Pimsner algebras.
This extends Cuntz's classical decomposition for Cuntz algebras
and reveals an implicit symmetric structure within these algebras.
Exploiting this structure, we characterize their simplicity and classify their ideals, tracial weights, and KMS weights for generalized quasi-free flows,
revisiting and refining seminal results by Kitamura, Schweizer, and Laca--Neshveyev.

We also give a short, elementary solution to the reduced Hao--Ng isomorphism problem for locally compact groups in full generality.
Bypassing non-self-adjoint operator algebra techniques used in prior work, our proof relies solely on \Cs-algebra theory.
In contrast, we disprove the full Hao--Ng conjecture by constructing counterexamples.

Combining these results, we investigate quasi-free actions on Cuntz algebras.
Notably, we affirmatively answer a recent question posed by Izumi on isometric shift-absorption for compact groups.
\end{abstract}

\tableofcontents

\section{Introduction}
Pimsner \cite{Pim} established an extremely powerful method for constructing universal \Cs-algebras from a \Cs-correspondence.
The advantages of the Pimsner construction lie in its inherent universality and the tractability of its KK-theoretic invariants.
Moreover, the Pimsner construction inherits various analytic properties from the coefficient \Cs-algebras (see \cite[Section 4.6]{BO}), while frequently yielding simple \Cs-algebras (see, e.g., \cite{Kum}, \cite{SuzP}).
Recently, this construction has emerged as a powerful tool even for well-studied \Cs-algebras such as Kirchberg algebras and AF-algebras.
(For background on the classification theory of simple \Cs-algebras,
we refer the reader to the comprehensive introduction in \cite{CGS}.)
 This is notably evident in the construction of group actions---especially amenable actions of non-amenable groups---as well as outer actions of unitary tensor categories (see, e.g., \cite{Kit}, \cite{Mey}, \cite{OS}, \cite{Suzsf}).

In this paper, we show that Kasparov's stabilization theorem for Hilbert \Cs-modules \cite{Kas} is highly effective for investigating the structure of Pimsner algebras.
Specifically, under the standard Assumption \ref{ass:setting} below, we prove that the stabilization of a Cuntz--Pimsner algebra is always decomposable as a crossed product of a self-similar automorphism, in the same spirit as Cuntz's crossed product decomposition of the stabilized Cuntz algebras.
Furthermore, we utilize this main result to examine the detailed structure of Cuntz--Pimsner algebras.
We provide characterizations of their simplicity, a classification of ideals and tracial weights, and a classification theorem for KMS weights for (generalized) quasi-free flows.
These results unify and extend existing literature, including the seminal results of Kitamura \cite{Kit}, Schweizer \cite{Sch}, and Laca--Neshveyev \cite{LN}.

We also provide a short, self-contained solution to the reduced Hao--Ng isomorphism problem for locally compact Hausdorff groups in full generality.
Although a solution for the non-degenerate case was recently announced in \cite{DT},
earlier versions of the preprint\footnote{A correction is proposed in the latest revision (v3).} contained a critical gap.
Unlike previous approaches developed in \cite{KR}, \cite{DT}, which rely heavily on the machinery of non-self-adjoint operator algebras and their \Cs- and ${\rm W}^\ast$-envelopes,
our proof builds solely on basic properties of \Cs-algebras.
Our proof also recovers the full isomorphism-type result formulated and proved by Katsoulis and Ramsey for hyperrigid \Cs-correspondences (Theorem 3.10, \cite{KR}) in the regular case.
We further show that their reformulation is essential by constructing counterexamples to the original conjecture.
Our counterexamples are regular and hence hyperrigid, thereby answering the questions posed below Theorem 3.10 and at the end of Section 4 in \cite{KR}
and disproving one half of the refined conjecture stated in Section 7 of \cite{KR19} even in this specific case (see also the concluding remarks in \cite{BKQR}, \cite{DT}).
The proofs of these results are independent of the rest of the paper;
both are self-contained and less than a page long.

As an application of these results, we investigate quasi-free actions on Cuntz algebras.
This approach provides an elementary and conceptual explanation at the level of \Cs-correspondences for the well-known result of Doplicher--Roberts \cite{MR901232} (and its far-reaching generalization \cite{MR1432596}) that the fixed point algebra of a quasi-free compact group action is stably isomorphic to a graph algebra. In particular, we resolve a question---recently posed by Izumi \cite{Izuqp}---concerning when a quasi-free action of a compact group is isometrically shift-absorbing.

\begin{ass}\label{ass:setting}
Throughout this article, except in Sections \ref{section:HN} and \ref{section:HN2},
for a \Cs-correspondence $\cE$ over $A$, we suppose the following standard assumptions.
The coefficient \Cs-algebra $A$ is $\sigma$-unital,
and $\cE$ is
 \begin{itemize}
\item faithful, i.e., the left action $A \rightarrow \IB(\cE)$ is injective,
\item non-degenerate, i.e., $A\cdot \cE$ is dense in $\cE$,
\item full, i.e., the range of the $A$-valued inner product is total in $A$,
\item countably generated.
\end{itemize}
\end{ass}
For basic facts and notation regarding Pimsner algebras and \Cs-correspondences, we refer the reader to Section 4.6 of \cite{BO}
in the faithful case, and \cite{Kat04} for the general case.
Unless otherwise specified, we follow the (standard) notations used therein.
\begin{no}Here we give a list of notations used throughout the article.
\begin{itemize}
\item We include $0$ in $\IN$.
\item Put $\IK:=\IK(\ell^2(\IN))$ for short.
\item The multiplier algebra of a \Cs-algebra $A$ is denoted by $\cM(A)$.
\item For a right Hilbert \Cs-module $\cE$, we write $\IK(\cE)$ and $\IB(\cE)$
for the \Cs-algebras of compact operators and adjointable operators, respectively. 
\item We write $\cU(\cE)$ for the unitary group of $\IB(\cE)$.
\item For a subset $S$ of a normed space $X$, we set $(S)_1:=\{s\in S: \|s\|=1\}$.
\item For two elements $x, y\in X$ and $\ve>0$, we write $x\approx_\ve y$ when $\|x-y\| <\ve$.
\item Similarly, for $S \subset X$, $x\in X$, and $\ve>0$, we write $x\in_\ve S$ if there is an element $s\in S$ with $x\approx_\ve s$.
\item For a subset $S$ of a topological space, we denote by $\cl(S)$ the closure of $S$.
\item For two \Cs-algebras $A, B$, we denote by $A\otimes B$ their minimal tensor product.
\item For a tensor product of \Cs-correspondences, we primarily use the interior tensor product.
Thus, the symbol $\otimes$ is reserved for the interior tensor product.
When we wish to emphasize the coefficient \Cs-algebra $A$, we write $\otimes_A$ as is standard.
Although the exterior tensor product is also denoted by the same symbol, this occurs only at a few points in this article and will be accompanied by an explicit note.
\end{itemize}
\end{no}

\section{Stabilizations and crossed product decompositions of Cuntz--Pimsner algebras}\label{section:st}

Here we record consequences of the stabilization theorem to the Cuntz--Pimsner algebras.
In particular, we show that the Cuntz--Pimsner algebras are stably isomorphic to the crossed product of an automorphism
on an associated \Cs-algebra.
We will use these results to study the simplicity, ideals, traces, and KMS weights on Cuntz--Pimsner algebras in later sections.
Notations introduced in this section will be used therein.

For a $\ast$-homomorphism $\rho \colon A \rightarrow \cM(A)$,
we write ${}_\rho A$ for the \Cs-correspondence over $A$
where the right \Cs-module structure is standard, and the left action is given via $\rho$.

The next proposition is elementary but fundamental for us.
\begin{Prop}\label{Prop:st}
Let $\mathcal{E}$ be a \Cs-correspondence over $A$.
Then the exterior tensor product $\mathcal{E}\otimes \IK$ is isomorphic to
${}_\rho (A\otimes \IK)$ for some $\ast$-homomorphism $\rho\colon A\otimes \IK \rightarrow \cM(A\otimes \IK)$.
\end{Prop}
\begin{proof}
We regard the complex conjugate space $\overline{\ell^2}$ as a \Cs-($\mathbb{C}$, $\IK$)-correspondence
given by the obvious left $\IC$-action and the right $\IK$-action together with the inner product
\[\ip{\bar{\xi}, \bar{\eta}}=e_{\xi, \eta} \quad {\rm~for~}\xi, \eta \in \ell^2.\]
We also regard $\ell^2$ as a \Cs-correspondence over $\IC$ in the obvious way.
Then, as right Hilbert \Cs-$\IK$-modules, we have an isomorphism
\[\ell^2\otimes \overline{\ell^2} \cong \IK\]
given by sending $\xi \otimes \bar{\eta}$ to $e_{\xi, \eta}$ for $\xi, \eta\in \ell^2$.
Then, as right Hilbert \Cs-$(A\otimes \IK)$-modules, one has
\[\cE \otimes \IK \cong \cE\otimes \ell^2\otimes \overline{\ell^2} \cong \cE\otimes \ell^2\otimes \ell^2\otimes \overline{\ell^2} \cong A\otimes \ell^2 \otimes \ell^2\otimes \overline{\ell^2} \cong A\otimes \IK.\]
Here the third isomorphism is a consequence of the stabilization theorem.
Indeed, by the stabilization theorem \cite[Theorem 2]{Kas} (with $G=1$),
$\cE\otimes \ell^2$ is isomorphic to a Hilbert \Cs-submodule of $A\otimes \ell^2$ of the form $p(A\otimes \ell^2)$ for some projection $p\in \cM(A\otimes \IK)=\IB(A\otimes \ell^2)$.
Since $\ell^2 \cong \ell^2\otimes \ell^2$, the projection $p$ satisfies $\bigoplus_{n\in \IN} p \sim p$, where the infinite direct sum is taken in the strict topology.
Moreover, as $\cE$ is full, the projection $p$ is full.
Hence, by Lemma 2.5 of \cite{Br}, $p$ is Murray--von Neumann equivalent to the unit in $\cM(A\otimes \IK)$. This completes the proof of the isomorphism.
\end{proof}

We recall the definition of a $G$-\Cs-correspondence over a $G$-\Cs-algebra $(A, \alpha)$.
A \Cs-correspondence $\cE$ over $A$ equipped with a continuous $G$-action $\upsilon \colon G \acts \cE$ by isometric linear transformations is said to be a $G$-\Cs-correspondence over $(A, \alpha)$
if $\upsilon$ satisfies the compatibility conditions
\[\ip{\upsilon_s(\xi), \upsilon_s(\eta)}_A = \alpha_s(\ip{\xi, \eta}_A),\quad \upsilon_s(\xi a)=\upsilon_s(\xi)\alpha_s(a),\quad \upsilon_s(a\xi)=\alpha_s(a)\upsilon_s(\xi)\]
for all $s\in G$, $\xi, \eta\in \cE$, and $a\in A$.
By the universal property of the Cuntz--Pimsner algebra,
the action $\upsilon$ extends to a \Cs-dynamical system $\alpha_\upsilon \colon G \acts \Cn{\cE}$.
We refer to the action $\alpha_\upsilon$ as the generalized quasi-free action induced by $\upsilon$.
When $\cE$ is of the form ${}_\rho A$ for some $G$-equivariant $\ast$-homomorphism $\rho \colon A\rightarrow \cM(A)$,
we write $\alpha_\upsilon$ by $\alpha$ for short.
When $\alpha$ is the trivial action, we refer to $\alpha_\upsilon$ as a quasi-free action.

\begin{Rem}[Equivariant version of Proposition \ref{Prop:st}]\label{Rem:st}
We apply the following fact in the case $G=\IR$ to study KMS weights on Cuntz--Pimsner algebras.
In Section \ref{section:HN}, we also employ the general case to give a short and alternative solution to the Hao--Ng isomorphism problem.

Let $G$ be a locally compact second-countable group.
By Proposition \ref{Prop:st} (the Kasparov stabilization theorem \cite{Kas}) and Lemma 2.3 of \cite{MP},
we obtain the following equivariant version
by replacing $\ell^2$ with $L^2(G)^{\infty}:=\ell^2\otimes L^2(G)$ in the proof:

\emph{Let $(\mathcal{E}, \upsilon)$ be a full, countably generated $G$-\Cs-correspondence over a $G$-\Cs-algebra $A$.
We equip $\IK(L^2(G)^{\infty})$ with the left regular $G$-action $\lambda^{\infty}$.
Then the exterior tensor product $(\mathcal{E}\otimes \IK(L^2(G)^{\infty}), \upsilon \otimes \lambda^{\infty})$ is $G$-equivariantly isomorphic to
${}_\rho (A\otimes \IK(L^2(G)^{\infty}))$ for some $G$-equivariant $\ast$-homomorphism $\rho\colon A\otimes \IK(L^2(G)^{\infty}) \rightarrow \cM(A\otimes \IK(L^2(G)^{\infty}))$.}
 
 Note that $\alpha_{\upsilon \otimes \lambda^{\infty}}\colon G \acts \Cn{\cE\otimes \IK(L^2(G)^{\infty})}$ is exterior equivalent
 to the plain stabilization $\alpha_\upsilon \otimes 1_{\IK} \colon G \acts\Cn{\cE}\otimes \IK$ via the unitary representation
 \[1_{\Cn{\cE}} \otimes \lambda^{\infty} \colon G \rightarrow \cU(L^2(G)^\infty)\subset \cM(\Cn{\cE\otimes \IK(L^2(G)^{\infty})}).\]
 Hence, there exists a canonical isomorphism between $\Cn{\cE\otimes \IK(L^2(G)^{\infty})}\rca{\alpha_{\upsilon\otimes \lambda^{\infty}}} G$ and $(\Cn{\cE}\rca{\alpha_\upsilon}G) \otimes \IK$.
\end{Rem}

Let $\rho\colon A \rightarrow \cM(A)$ be a faithful non-degenerate $\ast$-homomorphism.
We denote by the same symbol $\rho$ for the strictly continuous extension
 $\cM(A) \rightarrow \cM(A)$.
Consider the inductive system $(\cM(A), \rho)_{n\in \IZ}$ indexed by the directed set $\IZ$.
We denote by $\fA$ its inductive limit.
Let $\iota_n \colon \cM(A) \rightarrow \fA$ be the $n$-th canonical map.
Note that $\iota_{n+1}\circ \rho=\iota_n$ for $n\in \IZ$.

For each $n\in \IZ$, set $B_n:=\iota_n(A)$.
Unless otherwise specified, we identify $A$ with $B_0$ via $\iota_0$.
For $n\in \IZ$ and $m\in \IN$, one has
$B_n=\iota_{n+m}(\rho^m(A))$. From this, for $n, m \in \IZ$, by Cohen's factorization theorem (see Theorem 4.6.4 in \cite{BO}), one has
\begin{equation}\label{equation:B}
B_n \cdot B_m = B_{\max\{n, m\}}.
\end{equation}
(Throughout the paper, we only need the density of the obvious inclusion, which is clear. However we freely use Cohen's factorization theorem as it simplifies notations.)
Next, for $n, m \in \IZ \cup \{\pm \infty\}$ with $n\leq m$, set
\[B_{[n, m]}:=\csp\{B_k:k\in [n, m]\cap \IZ\} \subset \fA.\]
We write $\fB:=B_{[-\infty, \infty]}$ for short.
By equation (\ref{equation:B}), each $B_{[n, m]}$ is a \Cs-subalgebra of $\fA$.
When $n\in \IZ$ and $k\in \IN$, one has
$B_{[n, n+k]}=\iota_{n+k}(A+\rho(A)+\cdots +\rho^k(A))$.
For $n\in \IZ$, by equation (\ref{equation:B}), the hereditary \Cs-subalgebra of $\fB $ generated by $B_n$ is equal to $B_{[n, \infty]} \lhd \fB$.
In particular $B_n$ is non-degenerate in $B_{[n, \infty]}$.
Note that when $\rho(A)\subset A$, one has $B_n \subset B_{n+1}$ for all $n\in \IZ$
and hence in this case $A$ is non-degenerate in $\fB$.
By the universality of $\fA$, one has an automorphism $\sigma$ on $\fA$
satisfying $\sigma \circ \iota_n =\iota_{n+1}$ for $n\in \IZ$.
It is easy to see that $\sigma(B_n)=B_{n+1}$ for $n\in \IZ$.
Hence $\sigma$ restricts to an automorphism on $\fB $,
which we denote by the same symbol $\sigma$.
Note that $\iota_n \circ \rho=\sigma^{-1}\circ \iota_n$.

The next lemma is useful to determine the position of an element in $\fB$.

\begin{Lem}\label{Lem:mult}
For $a\in \cM(A)$, if $\rho(a)\in A$, one has $a\in A$.
\end{Lem}
\begin{proof}
Take an approximate unit $(e_n)_{n\in \IN}$ of $A$.
Assume that $a\in \cM(A)\cap \rho^{-1}(A)$.
Then one has
\[\rho(a)=\lim_{n \rightarrow \infty} \rho(a e_n)\]
in norm. Since $\rho$ is isometric, this implies
$a=\lim_{n \rightarrow \infty} ae_n \in A$.
\end{proof}

For $\rho \colon A \rightarrow \cM(A)$,
we denote by $\Cn{\rho}$ the Cuntz--Pimsner algebra of the \Cs-correspondence ${}_\rho A$ over $A$.
The following theorem and Proposition \ref{Prop:st} show that,
under Assumption \ref{ass:setting}, the Cuntz--Pimsner algebra $\Cn{\cE}$ is always stably isomorphic to the reduced crossed product of an associated self-similar automorphism.
This is a natural generalization of the canonical crossed product decomposition of
the stabilized Cuntz algebras $\Cn{n}\otimes \IK$ for $2\leq n <\infty$ established in \cite{Cu77}.
\begin{Thm}\label{Thm:decomp}
The Cuntz--Pimsner algebra $\Cn{\rho}$ is isomorphic to
the full hereditary \Cs-subalgebra $H$ of $\fB \rca{\sigma} \mathbb{Z}$ generated by $A$.
\end{Thm}
\begin{proof}
Let $u\in \cM(\fB \rca{\sigma} \mathbb{Z})$ denote the canonical implementing unitary element of $\sigma$.
Direct calculations show that the family
$(u a)_{a\in A}$ and the canonical copy
$A \subset \fB \rca{\sigma} \IZ$ satisfy the Cuntz--Pimsner relations of ${}_\rho A$.
(Indeed, if $\rho(a)\in A$ for $a\in A$, write $\rho(a)=b c^\ast$; $b, c\in A$.
Then one has $(ub)(uc)^\ast= u \rho(a)u^\ast =u \iota_{-1}(a)u^\ast=a$. The other relations are clear.)
Hence one has a $\ast$-homomorphism $\Phi \colon \Cn{\rho} \rightarrow \fB \rca{\sigma} \mathbb{Z}$
given by $\Phi(S_a)=u a$ for $a\in {}_\rho A$.
Since $\rho$ is non-degenerate, with $(e_n)_n$ an approximate unit of $A$,
for $a\in A$, one has
\[\Phi(S_a)=\lim_n u \rho(e_n) a e_n= \lim_{n} e_n \Phi(S_a)e_n \in H.\]
Hence $\Phi(\Cn{\rho}) \subset H$.
It is also clear that $\Phi$ is gauge-equivariant.
Since $\Phi$ is injective on $A$, it is injective on $\Cn{\rho}$ as well by the gauge-invariant uniqueness theorem (\cite{BO}, Theorem 4.6.20).
It remains to show that $H\subset \Phi(\Cn{\rho})$.

Obviously we have $A \subset \Phi(\Cn{\rho})$.
For $(n, m)\in \IN^2\setminus\{(0, 0)\}$ and $a_1, \ldots, a_n, b_1, \ldots, b_m\in A$, observe that
\[\Phi(S_{a_1} \cdots S_{a_n} S_{b_m}^\ast \cdots S_{b_1}^\ast)= \iota_1(a_1)\iota_2(a_2)\cdots \iota_n(a_n) \iota_{n}(b_m^\ast) \iota_{n-1}(b_{m-1}^\ast) \cdots \iota_{n-m+1}(b_1^\ast) u^{n-m}.\]
(Here and below when $n$ or $m$ is $0$, we ignore the corresponding terms.)
By equation (\ref{equation:B}), one has
\[B_n u^{n-m}=\{\Phi(S_{a_1} \cdots S_{a_n} S_{b_m}^\ast \cdots S_{b_1}^\ast):
a_1, \ldots, a_n, b_1, \ldots, b_m\in A\}.\]
We also have $B_0=\Phi(A)$.
Thus, for $k\in \IN$, by applying the equality to the case $n=l+k, m=l$ with $l\in \IN$,
we obtain $B_{[k, \infty]} u^k \subset \Phi(\Cn{\rho})$.
For $k \in \IN$, one has 
\[A\cdot (\fB u^k) \cdot A= A\cdot \fB \cdot B_k \cdot u^k = B_{[k, \infty]} u^k.\]
This proves the desired inclusion $H\subset \Phi(\Cn{\rho})$.
\end{proof}

\section{Invariant ideals and simplicity of Cuntz--Pimsner algebras}\label{section:simple}
In this section, we classify gauge-invariant ideals and characterize the simplicity for Cuntz--Pimsner algebras.
We first study the case $\cE={}_\rho A$ for some $\rho \colon A \rightarrow \cM(A)$.
We characterize the minimality of the associated automorphism $\sigma$ on $\fB$, and also classify the $\sigma$-invariant ideals of $\fB$.
The general case is deduced to this case by Proposition \ref{Prop:st}.

The next lemma classifies $\sigma$-invariant ideals of $\fB$.
This revisits a result of Katsura \cite{KatId} on the classification of gauge-invariant ideals in the Cuntz--Pimsner algebras.
We note that the \Cs-algebra $D$ therein cannot be replaced
by the most natural candidate $A$ as a counter-example is given in Remark \ref{Rem:idclass} below.
\begin{Lem}\label{Lem:idclass}
There is a bijective correspondence between $\sigma$-invariant ideals $I \lhd \fB $
and ideals $J \lhd D:=A+\rho(A)$ satisfying 
\begin{equation}\label{equation:inv2}
D\cdot \rho(J)\subset J,\quad \rho^{-1}(J)\cap D \subset J.
\end{equation}
The correspondence is given by
\[I \mapsto I \cap D,\quad J \mapsto \csp\{\iota_n(J):n \in \IZ\}.\]
\end{Lem}
\begin{proof}
Throughout the proof, we identify $D$ with $B_{[0, 1]}$ via $\iota_1$.

First, let $J\lhd D$ be an ideal with condition (\ref{equation:inv2})
and put
\[I:=\csp\{\sigma^n(J):n \in \IZ\} \subset \fB ,\quad I_n:= \csp\{\sigma^k(J):0\leq k\leq n\} \subset B_{[0, n+1]}.\]
We show that $I$ is a $\sigma$-invariant ideal with $I\cap D =J$.

Note that for $n\geq 2$, one has
\[D \rho^n (J) = D \rho^{n-1}(D)\rho^n(J) =D \rho^{n-1}(D\rho(J))\subset D \rho^{n-1}(J).\]
Hence one has
\begin{equation}\label{equation:inv3}
D\rho^n(J)\subset J \quad {\rm~for~}n\in \IN.
\end{equation}
Clearly we also have
\begin{equation}\label{equation:inv5}
\rho^n(D) J \subset J \quad {\rm~ for~} n\in \IN.
\end{equation}
By relations (\ref{equation:inv3}) and (\ref{equation:inv5}),
one has 
\begin{equation}\label{equation:inv4}
\sigma^n(J)\cdot B_{[m, m+1]}\subset \sigma^{\max\{n, m\}}(J)\quad{\rm~ for~} n, m\in \IZ.
\end{equation}
This shows that $I\lhd \fB $ and $I_n \lhd B_{[0, n+1]}$.
We next show that $I\cap D= J$.
Relation (\ref{equation:inv4}) implies
\[I \cap D \subset D\cdot I \subset I_+:=\csp\{\sigma^n(J):n\geq 0\} \lhd B_{[0, \infty]}.\]
Hence it suffices to show that $I_+ \cap D =J$.
Let
\[\pi\colon B_{[0, \infty]}\rightarrow B_{[0, \infty]}/I_+,\quad \pi_n\colon B_{[0, n+1]}\rightarrow B_{[0, n+1]}/I_n,\quad q\colon D \rightarrow D/J\]
denote the quotient maps.
Then, for any $x\in B_{[0, N+1]}$; $N\in \IN$, one has
\[\|\pi(x)\|=\inf_{n\geq N} \|\pi_n(x)\|.\]
If $I_+ \cap D \neq J$, then for some $x\in D$, we have $\|\pi(x)\|<\|q(x)\|$.
Then, for a sufficiently large $n\in \IN$, one has $\|\pi_n(x)\|<\|q(x)\|$.
This shows that the ideal $D\cap I_n= \ker(\pi_n|_D) \lhd D$ is strictly larger than $J$.
Hence there are elements $x_0, \ldots, x_n\in J$ with
\[a:=\sum_{k=0}^n\sigma^k(x_k) \in D\setminus J.\]
We fix such a presentation of $a$ with the smallest possible $n\geq 1$.
Write $x_n=a_n + \rho(b_n)$; $a_n, b_n\in A$.
Then one has 
\[\iota_{n+1}(a_n)=a-\iota_n(b_n)-\sum_{k=0}^{n-1} \sigma^k(x_k) \in B_{[0, n]} \subset \iota_{n+1}(\rho(\cM(A))).\]
Hence $a_n\in A \cap \rho(\cM(A))$. By Lemma \ref{Lem:mult}, one has $a'_n \in A$
with $\rho(a'_n)=a_n$.
Put $c_n:= a'_n + b_n\in A$. Then $\rho(c_n) = x_n \in J$ and hence $c_n\in J$ by condition (\ref{equation:inv2}).
Define $x'_{n-1}:=x_{n-1}+c_n\in J$.
Then one has
\[a=\sum_{k=0}^{n-2} \sigma^k(x_k) +\sigma^{n-1}(x'_{n-1}).\]
This contradicts to the minimality of $n$.
Thus $I_+ \cap D = J$.

Next, let $I \lhd \fB$ be a $\sigma$-invariant ideal.
Put $J:=D \cap I$.
Then, as $I$ is $\sigma$-invariant,
one has
\[D\cdot \rho(J) \subset D \cap I=J,\quad \rho^{-1}(J)\cap D \subset I \cap D = J.\]
Thus $J$ satisfies relation (\ref{equation:inv2}).

We next show that $J$ generates $I$ as a $\sigma$-invariant ideal in $\fB$.
This completes the proof.

Denote by $I'$ the $\sigma$-invariant ideal of $\fB $ generated by $J$. Clearly $I' \subset I$ and 
\[I'=\csp\{\sigma^n(J): n\in \IZ\}\]
by relation (\ref{equation:inv2}).
To lead to a contradiction, assume that $I\neq I'$.
Then, since both $I'$ and $I$ are $\sigma$-invariant ideals,
there exists $n\in \IN$ with
$I'\cap B_{[0, n]} \neq I \cap B_{[0, n]}$.
Let $N\in \IN$ be the smallest integer with this property.
Note that $N>1$ as $I\cap D = J \subset I' \cap D$.
Pick
\[x=\sum_{i=0}^N \iota_i(a_i) \in (I\setminus I') \cap B_{[0, N]};\quad a_i\in A.\]
Then for any $\varepsilon>0$, one has $e\in (A_+)_1$ with
\[\sum_{i=1}^{N}\iota_{1}(e) \iota_i(a_i)\approx_\varepsilon \sum_{i=1}^{N}\iota_i(a_i).\]
Observe that $\iota_1(e)x \in B_{[1, N]}\cap I = \sigma(B_{[0, N-1]}\cap I)\subset I'$
 by the minimality of $N$.
Also
\[I\ni x-\iota_1(e)x\approx_\varepsilon \iota_0(a_0)-\iota_1(e \rho(a_0)) \in B_{[0, 1]}.\]
Since the canonical map $B_{[0, 1]}/(I \cap B_{[0, 1]}) \rightarrow B_{[0, n]}/(I \cap B_{[0, n]})$ is isometric,
this implies 
\[x-\iota_1(e)x \in_{2\varepsilon} B_{[0, 1]}\cap I\subset I'.\]
This proves $x\in_{2\varepsilon} I'$.
As $\varepsilon>0$ is arbitrary, we conclude $x\in I'$. This is a contradiction.
\end{proof}
\begin{Rem}\label{Rem:idclass}
When $\rho$ is proper, one has $D=A$ and hence the intersection with $A$ already separates $\sigma$-invariant ideals of $\fB$.
Unfortunately this is not true in general.
Here we present a counter-example.

Put $A:=c_0(\IN)$. Put $S:=2\IN +1$.
Then we define a $\ast$-homomorphism $\rho \colon A \rightarrow \cM(A)=\ell^\infty(\IN)$ to be
\begin{equation*}
 \rho(\delta_n):=
 \begin{cases}
 \delta_0 + \chi_S & \text{if $n=0$,} \\
 \delta_{2n} & \text{otherwise.}
 \end{cases}
\end{equation*}
Then clearly $\rho$ is injective and non-degenerate.
One has
\[D:=A+\rho(A)= A+ \IC \chi_S \subset \ell^\infty(\IN).\]
Consider the two ideals
\[J:= c_0(\IN\setminus \{0\})+ \IC \chi_S,\quad J_0:=c_0(\IN\setminus \{0\}) \lhd D.\]
Then both $J$ and $J_0$ satisfy condition (\ref{equation:inv2}) in Lemma \ref{Lem:idclass}.
Hence by Lemma \ref{Lem:idclass}, they generate distinct $\sigma$-invariant ideals in $\fB $, say $I$, $I_0$, respectively,
which satisfy $I\cap D= J$, $I_0 \cap D=J_0$.
At the same time, we have
\[I \cap A = J\cap A =J_0=I_0\cap A.\]
Thus the two ideals are not separated on $A$.
\end{Rem}

While the minimality of $\sigma$ is already characterized by Lemma \ref{Lem:idclass}, we provide a more straightforward characterization.
\begin{Lem}\label{Lem:minimal}
The automorphism $\sigma$ on $\fB$ is minimal if and only if there is no proper ideal $J\lhd A$ satisfying
\begin{equation}\label{equation:inv}
A\rho(J)\subset J,\quad \rho^{-1}(J) \subset J.
\end{equation}
\end{Lem}
\begin{proof}
Assume that there is no proper ideal $J \lhd A$ with condition (\ref{equation:inv}).
Let $I \lhd \fB$ be a nonzero $\sigma $-invariant ideal.
We show that $J:=I\cap A$ is a nonzero ideal satisfying condition (\ref{equation:inv}).
Since $A$ generates $\fB $ as a $\sigma$-invariant ideal, this proves the minimality of $\sigma$.

As $I$ is $\sigma$-invariant,
one has
\[A\rho(J) \subset A \cap I=J,\quad \rho^{-1}(J) \subset I \cap A = J.\]
Hence $J$ indeed satisfies condition (\ref{equation:inv}).
By Lemma \ref{Lem:idclass}, one has
$I \cap B _{[0, 1]}\neq 0$.
Since $B_{1}\lhd B _{[0, 1]}$ is essential,
we have
$I \cap B_{1}\neq 0$.
By applying $\sigma^{-1}$ further, we conclude
$J=I\cap A \neq 0$.

Conversely, assume that there is a proper ideal $J \lhd A$ with condition (\ref{equation:inv}).
Then
\[L:=J+\rho(J) \subsetneq A + \rho(A)\] as $A\cdot L \subset J \subsetneq A$.
By condition (\ref{equation:inv}), $L$ is an ideal of $D$.
Clearly $L$ satisfies the first condition in (\ref{equation:inv2}) in Lemma \ref{Lem:idclass}.
To show the second condition, pick $a, b \in A$ with $\rho(a+\rho(b))\in L$.
We need to show that $a+\rho(b)\in L$.
Pick $x, y\in J$ with $\rho(a+\rho(b))= x+\rho(y)$.
This equality implies $\rho(b)\in A$, hence by replacing $a$ with $a+\rho(b)\in A$, we may assume that $b=0$.
We have $\rho(a-y)=x\in J$, which implies $a-y\in J$ by condition (\ref{equation:inv}). Hence $a\in J$ as desired.
By Lemma \ref{Lem:idclass}, the ideal $L$ generates a proper $\sigma$-invariant ideal in $\fB$.
Hence $\sigma$ is not minimal.
\end{proof}
When the coefficient \Cs-algebra is not stably projectionless, we conclude the following complete characterization of the simplicity of the Cuntz--Pimsner algebras.
The assumption is used to control the multiplier algebra of an inductive limit \Cs-algebra, which is very hard in the stably projectionless case.
Note that in the unital case, this recovers the claim of \cite{Sch}, whose important technical parts are unfortunately unpublished and unavailable in public.

We introduce a few related definitions.
\begin{Def}[cf.~\cite{Sch}, \cite{Kit}]\label{Def:IdC*corr}
Let $\cE$ be a \Cs-correspondence over $A$. Let $\rho$ denote the left $A$-action of $\cE$.
We say that $I \lhd A$ is $\cE$-invariant
if it satisfies
\[\ip{\cE, \rho(I)\cE}\subset I,\quad \rho^{-1}(\IK(\cE I))\subset I.\]
We say that $\cE$ is minimal if there is no proper $\cE$-invariant ideal of $A$.
\end{Def}
Note that $\cE$ is minimal if and only if the associated $\ast$-homomorphism
$\rho \colon A \otimes \IK \rightarrow \cM(A\otimes \IK)$ in Proposition \ref{Prop:st} 
satisfies condition (\ref{equation:inv}).
\begin{Thm}\label{Thm:simple}
Assume that $A$ is not stably projectionless. Let $\mathcal{E}$ be a \Cs-correspondence over $A$.
Then $\Cn{\mathcal{E}}$ is simple if and only if $\mathcal{E}$ satisfies the following conditions.
\begin{enumerate}\upshape
\item $\cE$ is minimal.
\item There are no $n\geq 1$ and a nonzero projection $p\in A\otimes \IK$ with
\[p(\mathcal{E}^{\otimes n} \otimes \IK)p \cong p(A\otimes \IK)p\]
as \Cs-correspondences over $p(A\otimes \IK)p$.
\end{enumerate}
\end{Thm}
\begin{proof}
Clearly all properties are preserved under taking the stabilization.
Hence by taking the stabilization, it suffices to show the claim when
$\cE={}_\rho A$ for some $\rho\colon A \rightarrow \cM(A)$.
Then by Lemma \ref{Lem:minimal}, the minimality of $\cE$ is equivalent to the minimality of $\sigma$ on $\fB$.
Thus the minimality of $\cE$ is necessary for the simplicity of $\Cn{\rho}$.
Hence to complete the proof, it suffices to show that,
for a minimal ${}_\rho A$, the second condition is equivalent to the proper outerness of $\sigma^n$ for all $n \geq 1$.
(For details on the proper outerness of an automorphism, we refer the reader to Section \ref{section:prop outer}.)

We first assume that $\cE$ fails condition (2) for an $n\geq 1$, and show that $\sigma^n$ is not properly outer.
By assumption, one has a nonzero partial isometry element $v \in A$
satisfying $v a = \rho^n(a) v$ for all $a\in pAp$ where $p=v^\ast v$, and $vv^\ast = \rho^n(p)$.
This implies $\rho^n(p)\in ApA$ and $\rho^n(p)A\rho^n(p)=vA v^\ast =\rho^n(pAp)$.
For $k \geq 2$, by using $\rho^{n(k-1)}(v)\rho^{n(k-2)}(v) \cdots v$ instead of $v$,
we obtain $\rho^{nk}(p)\in ApA \subset A$ and $\rho^{nk}(p)A\rho^{nk}(p)=\rho^{nk}(pAp)$.
This yields
\[\rho^k(p)\in A,\quad \rho^k(p)A\rho^k(p)=\rho^k(pAp)\quad {\rm~for~ all~} k\in \IN.\]
We conclude that $p\fB p=pB_{[0, \infty]}p=pAp$.
This yields $\sigma^n|_{p\fB p}=\ad(v)|_{p\fB p}$, hence $\sigma^n$ is not properly outer.

Conversely, assume that for an $n\geq 1$, $\sigma^n$ is not properly outer.
Then $\sigma^n$ is inner by the minimality of $\sigma$ and
Proposition \ref{Prop:properly outer}.
Fix a nonzero projection $p\in A$. Then, since $\sigma^n$ is inner,
one has $\iota_n(p) \sim p$ in $B_{[0, \infty]}$.
Choose a sufficiently large $N\geq n$ with $\iota_n(p) \sim p$ in $B_{[0, N]}$.
Since $B_{[n, N]} \lhd B_{[0, N]}$, this implies $p\in B_{[n, N]}$.
Hence we conclude
 $\rho^n(p)\in A$ (cf.~Lemma \ref{Lem:mult}).
Since this argument works for $nk$; $k\in \IN$, instead of $n$,
we conclude $\rho^{nk}(p)\in A$ for all $k\in \IN$.
Hence $\rho^k(p)\in A$ for all $k\in \IN$ by Lemma \ref{Lem:mult}.
Thus, by replacing $p$ with $\rho^{N-n}(p)$,
we may further assume that $p \sim \rho^n(p)$ in $A$.

We next claim that, for a sufficiently large $M \in \IN$, with $q:=\rho^M(p)\in A$, one has $\rho^{kn}(qAq) = \rho^{kn}(q)A\rho^{kn}(q)$ for $k\in \IN$.
Note that the equality implies $\rho^k(qAq)= \rho^k(q)A\rho^k(q)$ for all $k \in \IN$,
by the following commutative diagram and the injectivity of $\rho$:
\begin{equation*}
\xymatrix{
q A q \ar[r]^-{\rho^k}\ar[d]_-{\rho^{nk}}& \rho^k(q)A\rho^k(q) \ar[dl]^-{\rho^{(n-1)k}}\ar@{}@<-1.5ex>[dl]|{\circlearrowright}\\
\rho^{nk}(q)A\rho^{nk}(q)
}
\end{equation*}
Put 
\[H_k:=p B_k p =\iota_k(\rho^k(p)A\rho^k(p)) \quad{\rm~ for~} k\in \IN.\]
As $\rho^k(p)\in A$ for $k\in \IN$,
one has $H_k=\iota_{k+1}(\rho^{k+1}(p)\rho(A)\rho^{k+1}(p)) \subset H_{k+1}$.
Hence the sequence $(H_k)_k$ is increasing and
their union is dense in $p\fB p$.
Pick a unitary element $u\in \cM(\fB)$ with $\ad(u)=\sigma^{-n}$.
Since $\rho^n(p)\sim p$ in $A$,
one has
\[up\in \rho^n(p)\fB p=\cl \left(\bigcup_{m\in \IN} \rho^n(p)B_m p\right),\]
where the union is increasing.
Hence for a sufficiently large $M$, one has $v\in (B_Mp)_1$ with $\|up-v\|<1/4$.
For $L\geq M$, as $pB_{L+n}p= pu^\ast B_L up$, one has $(pB_{L+n}p)_1 \subset_{1/2} pB_{L}p$.
Since $pB_{L}p$ is a closed subspace of $pB_{L+n}p$, we conclude $p B_L p = p B_{L+n} p$ for $L \geq M$.
This yields $pB_M p=pB_{M+kn}p$ for $k\in \IN$, whence
\[\iota_{M+kn}(\rho^{M+kn}(p)\rho^{kn}(A)\rho^{M+kn}(p))=pB_M p=p B_{M+kn} p=\iota_{M+kn}(\rho^{M+kn}(p)A \rho^{M+kn}(p)).\]
Hence the present $M$ satisfies the desired condition.

Now we have $q\fB q=qAq$.
This together with the relation $\rho^n(q)\sim q$ in $A$ implies $\rho^n(q)\fB q \subset A$.
In particular $uq \in A$, and this element gives rise to the isomorphism
\[q({}_{\rho^n} A)q \cong qAq.\]

\end{proof}
\begin{Rem}
Suppose that $A\otimes \IK$ contains a full projection, $\cE$ is minimal, and $\Cn{\cE}$ is not simple.
Then, by applying the latter half of the proof of Theorem \ref{Thm:simple}
to a full projection,
we conclude the periodicity of $\cE$: $\cE^{\otimes n} \cong A$ for some $n\geq 1$.
Note that when $A$ is unital, this in particular recovers the claim of Schweizer \cite{Sch}, whose proof is unpublished and not available in public.
Indeed one can find $n\geq 1$ and a full projection $p$ with $p (\cE^{\otimes n} \otimes \IK)p \cong p(A\otimes \IK)p$ as in the proof of Theorem \ref{Thm:simple}.
By the fullness of $p$, one can choose a family $(p_i)_{i\in I}$ of mutually orthogonal projections in $A\otimes \IK$
with $p_i \precsim p$ for all $i\in I$ and $\sum_{i\in I} p_i = 1_{A\otimes \IK}$ in the strict topology.
These conditions imply that $\cE^{\otimes n} \otimes \IK \cong A\otimes \IK$,
from which it follows that $\cE^{\otimes n} \cong A$.
\end{Rem}
\begin{Rem}
In the non-unital case, unlike the claim of \cite{Sch} in the unital case,
one cannot conclude the periodicity of $\cE$, that is, the isomorphism $\cE^{\otimes n} \cong A$ for some $n \geq 1$,
from the non-simplicity of $\Cn{\cE}$ and the minimality of $\cE$.
Here we give an example.
Set $A:= c_0(\IN)\otimes \IK$.
Pick isometry elements $v_1, v_2\in \cM(\IK)$ with $v_1v_1^\ast + v_2 v_2^\ast=1$.
Then define a $\ast$-homomorphism $\rho \colon A\rightarrow A$ to be
\[\rho(\delta_0 \otimes x):= \delta_0 \otimes v_1 x v_1^\ast,\quad \rho(\delta_1\otimes x):=\delta_0\otimes v_2 x v_2^\ast,\quad \rho(\delta_n\otimes x):=\delta_{n-1}\otimes x\]
for $x\in \IK$ and $n\geq 2$. Clearly $\rho$ is injective, non-degenerate, and ${}_\rho A$ is minimal.
On the one hand, one has $\fB \cong \IK$.
Hence $\sigma$ is inner, and $\Cn{\rho}$ is isomorphic to $\IK\otimes C(\IT)$.
On the other hand, as $\rho$ is not surjective, one has
\[({}_\rho A)^{\otimes n} \not\cong A\]
for all $n \geq 1$.
Yet, with $p=\delta_0 \otimes e$ where $e\in \IK$ is any nonzero projection, obviously one has
$p({}_{\rho} A)p \cong p A p$.
This example also shows that not all nonzero projections $q$ satisfy the condition $q({}_{\rho} A)q\cong qA q$.
Indeed put $q=\delta_n\otimes e$ where $n\geq 1$ and $e\in \IK$ is a nonzero projection.
Then for $0\leq k<n$ and $l\geq 1$, one has
\[\rho^k(q)({}_{\rho} A)^{\otimes l}\rho^k(q)=0 \not\cong \rho^k(q)A\rho^k(q).\]
\end{Rem}
\begin{Rem}
For a non-proper minimal \Cs-correspondence, the second condition of Theorem \ref{Thm:simple} is redundant.
Hence this recovers a result of Kitamura \cite{Kit} in the non-stably-projectionless case.
For the general case, see also Remark \ref{Rem:simple} for a proof based on an old result \cite{OP}.
\end{Rem}

\section{Simplicity of Cuntz--Pimsner algebras: stably projectionless case}\label{section:simplepjl}
Here we give a complement of Theorem \ref{Thm:simple} in the stably projectionless case.
\begin{Thm}\label{Thm:simpleproless}
Let $A$ be a stably projectionless \Cs-algebra.
Let $\cE$ be a \Cs-correspondence over $A$.
Then $\Cn{\cE}$ is simple if and only
if $\cE$ is minimal and one of the following conditions holds true:
\begin{enumerate}\upshape
\item $A\otimes \Cn{\infty}$ is still $($stably$)$ projectionless,
\item $A \otimes \Cn{\infty}$ contains a nonzero projection,
and the exterior tensor product $\cE\otimes \Cn{\infty}$ satisfies the second condition in Theorem \ref{Thm:simple}.
\end{enumerate}
The Cuntz algebra $\Cn{\infty}$ can be replaced by any other purely infinite simple \Cs-algebra.
\end{Thm}
\begin{Rem}
Motivating examples with (1) is arising from quasi-free flows on the Cuntz--Pimsner algebras.
Indeed, for any (unitary) flow $u \colon \IR \acts \cE$ on a \Cs-correspondence $\cE$ over $A$,
the crossed product \Cs-algebra $\Cn{\cE} \rca{{\alpha_\upsilon}}\IR$ of the quasi-free flow ${\alpha_\upsilon}$ is isomorphic
to the Cuntz--Pimsner algebra of the Hao--Ng \Cs-correspondence $\cE\rca{u}\IR$.
The coefficient \Cs-algebra is $A \rca{1} \IR \cong A \otimes C_0(\IR)$,
which clearly satisfies condition (1).
\end{Rem}
\begin{proof}
By passing to the stabilization, we may assume that $\cE={}_\rho A$ for some $\rho \colon A \rightarrow \cM(A)$.
In case (2), the statement is an immediate consequence of Theorem \ref{Thm:simple}.
Hence we may assume that $A\otimes \Cn{\infty}$ is stably projectionless.
If $\rho$ is non-proper, the sequence $B_{[n, \infty]} \lhd \fB$; $n\in \IZ$,
is strictly decreasing.
Hence $\sigma\colon \IZ \acts \fB$ is outer.
By Proposition \ref{Prop:properly outer} or by \cite{OP}, we conclude the simplicity of $\Cn{\cE}$.

Next we consider the case that $\rho$ is proper.
In this case $\fB \otimes \Cn{\infty}$ is also stably projectionless.
Since the tensor product of a simple \Cs-algebra with $\Cn{\infty}$ contains nonzero projections,
$\fB $ has no nonzero simple \Cs-subalgebra.
This together with the minimality of $\sigma$
forces that the action $\sigma \colon \IZ \acts \fB $ is outer.
By Theorem \ref{Thm:decomp} and Proposition \ref{Prop:properly outer}, we conclude the simplicity of $\Cn{\cE}$.
\end{proof}
\begin{Rem}\label{Rem:simple}
The above proof in the non-proper case does not use the assumption (1).
Hence this recovers Theorem 2.3 of \cite{Kit} from \cite{OP}.
\end{Rem}

\section{Proper outerness of minimal $\IZ$-actions}\label{section:prop outer}
We discuss the proper outerness of minimal $\IZ$-actions without assuming the simplicity of the underlying \Cs-algebras.

We first recall the definition of proper outerness.
\begin{Def}
An automorphism $\alpha$ of $A$ is said to be properly outer
if it satisfies the following condition:
For any nonzero hereditary \Cs-subalgebra $H\subset A$ and any $a\in A$,
one has
\[\inf\{\|ha\alpha(h^\ast)\| : h\in H,\|h\|=1\}=0.\]
Note that, by spectral theory, this condition implies that $\alpha$ is outer.

We say that an action $\alpha \colon \Gamma \acts A$ is properly outer if $\alpha_s$ is properly outer for all $s\in \Gamma \setminus \{e\}$.
\end{Def}
The next statement follows from \cite[Theorem 2.15]{MR3756105} and some well-known facts.

\begin{Prop}\label{Prop:properly outer}
 For an action $\alpha\colon \IZ\acts A$ on a $\sigma$-unital \Cs-algebra $A$,
 the following conditions are equivalent.
 \begin{enumerate}\upshape
 \item It is outer and minimal.\label{item:outer}
 \item The crossed product \Cs-algebra $A\rca{\alpha}\IZ$ is simple.\label{item:simple}
 \item It is properly outer and minimal.\label{item:prouter}
 \end{enumerate}
\end{Prop}

\begin{proof}
The equivalence between (\ref{item:outer}) and (\ref{item:simple}) follows from \cite[Theorem 2.15]{MR3756105}.

We assume (\ref{item:simple}). Let $n\in \IZ \setminus \{0\}$.
Since the subgroup conditional expectation $E\colon A\rca{\alpha} \IZ\to A\rca{\alpha} n\IZ= A\rca{\alpha^{n}} \IZ$ has finite Watatani index,
it follows from Theorem 3.3 of \cite{Izu} that $A\rca{\alpha} n\IZ$ is a finite direct sum of simple \Cs-algebras.
We consider the strictly continuous extension $\widehat{\alpha^{n}}\colon \IT\acts \cM(A\rca{\alpha^{n}}\IZ)$ of the dual action.
Since the center $\mathrm{Z}(\cM(A\rca{\alpha^{n}}\IZ))$ is finite-dimensional,
the restricted action $\IT\acts \mathrm{Z}(\cM(A\rca{\alpha^{n}}\IZ))$ is trivial.
Thus $\widehat{\alpha^{n}}$ leaves every ideal of $A\rca{\alpha^{n}}\IZ$ invariant.
This implies that the strong Connes spectrum $\tilde{\IT}(\alpha^{n})$ is equal to $\IT$ (see Section 1 of \cite{Kis}).
Now, the proper outerness of $\alpha^{n}$ follows from \cite[Lemma 1.1]{Kis}.
This proves the implication $(2)\Rightarrow(3)$.

The remaining implication $(3)\Rightarrow (1)$ is trivial.
\end{proof}

\section{A short solution to the reduced Hao--Ng isomorphism problem}\label{section:HN}
The reduced Hao--Ng isomorphism problem asks whether the reduced crossed product of a Cuntz--Pimsner algebra
is naturally isomorphic to the Cuntz--Pimsner algebra of the reduced crossed product \Cs-correspondence.
This isomorphism is particularly useful for non-discrete groups,
because the structure of Cuntz--Pimsner algebras is much easier to study than that of crossed products by continuous groups.
Indeed, the general structure theory available for the latter remains very limited.

In this section, we give a short, alternative, and self-contained solution to the \emph{reduced Hao--Ng isomorphism problem} in full generality.
Throughout this section, we do not impose any of the conditions in Assumption \ref{ass:setting}\footnote{Thus, our \Cs-correspondence may fail to be non-degenerate, full, faithful, or countably generated, and the coeffcient \Cs-algebra may fail to be $\sigma$-unital.}.
To the best of our knowledge, prior to the preprint \cite{DT},
a positive solution was known only in specific cases such as
amenable groups \cite{HN}, discrete exact groups \cite{BKQR}, discrete groups \cite{Kat17}, and hyperrigid \Cs-correspondences \cite{KR}.
Furthermore, all of these results,
including \cite{DT}, assume non-degeneracy.

Our motivation is threefold:
\begin{itemize}
\item The strategies in \cite{KR}, \cite{DT}, etc. rely heavily on the machinery of non-self-adjoint operator algebras and their \Cs-envelopes.
Since this machinery may be less familiar to general \Cs-algebraists, it is desirable to provide an alternative proof using only standard facts about \Cs-algebras.
\item We wish to clarify that the reduced Hao--Ng isomorphism for \Cs-dynamical systems is not a non-self-adjoint operator algebra phenomenon, contrary to the prevailing view in the literature (see, e.g., \cite{Kat17}, \cite{KR19}, \cite{KR}, \cite{DT}). As additional benefits, we are able to remove the non-degeneracy assumption, and our proof is less than a page long.
\item The first two versions of \cite{DT}, which claim a solution to the reduced Hao--Ng isomorphism problem in the non-degenerate case, contain a serious gap.
Although a correction\footnote{Since we do not specialize in non-self-adjoint operator algebras,
we refrain from evaluating the correctness of this revision.} has been proposed in the latest major revision (v3), we believe it is preferable to provide a self-contained proof in the present paper to prevent confusion in the literature, as this result plays a central role in the last two sections.
\end{itemize}
Indeed, surprisingly, our proof relies only on the following three well-known facts:
\begin{enumerate}
\item The reduced crossed product functor preserves the injectivity of morphisms.
\item The reduced crossed product of an essential $G$-invariant ideal is an essential ideal of the reduced crossed product.
\item Katsura's gauge-invariant uniqueness theorem \cite{Kat04}, Theorem 6.4.
\end{enumerate}
Note that by conditions (1) and (2), for any $G$-\Cs-algebra $(A, \alpha)$,
as $A\rca{\alpha} G \lhd \cM(A)_{\rm c} \rca{\alpha} G$ is essential,
one has the inclusion $\cM(A)_{\rm c} \rca{\alpha} G \subset \cM(A\rca{\alpha} G)$
where $\cM(A)_{\rm c}$ denotes the $G$-continuous part of $\cM(A)$.
Hence, our proof applies to any crossed product functor satisfying conditions (1) and (2).
We remark that the full crossed product functor fails to satisfy both conditions.

Throughout this section, let $G$ be a locally compact (Hausdorff) group.
We recall that for a $G$-\Cs-correspondence $(\cE, \upsilon)$ over a $G$-\Cs-algebra $(A, \alpha)$,
the reduced Hao--Ng \Cs-correspondence $\cE\rca{\upsilon} G$ is the \Cs-correspondence over $A\rca{\alpha} G$
defined as the norm closure of $C_c(G, \cE) \subset \Cn{\cE}\rca{\alpha_\upsilon} G$,
where the bimodule structure and inner product are the restrictions of the canonical operations on $\Cn{\cE}\rca{\alpha_\upsilon} G$.
Alternatively, one can choose the Toeplitz--Pimsner algebra $\mathcal{T}_\cE\rca{\alpha_\upsilon} G$ in place of $\Cn{\cE}\rca{\alpha_\upsilon} G$,
or one can define $\cE\rca{\upsilon} G$ as the completion of the pre-\Cs-correspondence $C_c(G, \cE)$ over $C_c(G, A) \subset A\rca{\alpha} G$
where the bimodule structure and inner product are defined via convolution.

The reduced Hao--Ng isomorphism problem (see, e.g., \cite{BKQR}, \cite{KR}, \cite{DT}) asks whether the inclusion $\cE\rca{\upsilon} G \subset \Cn{\cE}\rca{\alpha_\upsilon}G$ extends to
a \Cs-algebra isomorphism
\[\Cn{\cE\rca{\upsilon}G} \cong \Cn{\cE}\rca{\alpha_\upsilon} G.\]

Note that both \Cs-algebras $\IK(\cE\rca{\upsilon} G)$ and $\IK(\cE)\rca{\alpha_\upsilon} G$ are canonically identified
with the same \Cs-subalgebra of $\mathcal{T}_\cE\rca{\alpha_\upsilon} G$.
This yields a canonical identification $\IK(\cE\rca{\upsilon} G)=\IK(\cE)\rca{\alpha_\upsilon} G$, under which the two dense copies of $C_c(G, \cE) \cdot C_c(G, \cE)^\ast$ are identified.
Under this identification, the left action of $A\rca{\alpha} G$ on $\cE\rca{\upsilon} G$ is the $\ast$-homomorphism
$A\rca{\alpha} G \rightarrow \IB(\cE)_{\rm c} \rca{\alpha_\upsilon} G \subset \IB(\cE \rca{\upsilon} G)$
induced by the original left action $A\rightarrow \IB(\cE)_{\rm c}$.
Hence, $\cE\rca{\upsilon} G$ is faithful if and only if $\cE$ is faithful.

\begin{Thm}\label{Thm:HN}
For any $G$-\Cs-correspondence over a $G$-\Cs-algebra,
the reduced Hao--Ng isomorphism holds.
\end{Thm}

\begin{proof}
We first observe that it suffices to consider the case where $\cE$ is faithful.
Indeed, let $(\cE, \upsilon)$ be any $G$-\Cs-correspondence over a $G$-\Cs-algebra $(A, \alpha)$.
Let $\rho \colon A \rightarrow \IB(\cE)$ denote the left action, and write $I$ for its kernel.
Recall that the Katsura ideal $J_{\cE} \lhd A$ consists of all elements $a\in A$ satisfying $a I =\{0\}$ and $\rho(a)\in \IK(\cE)$.
Set $\cF:= \cE\oplus I$, which we regard as a $G$-\Cs-correspondence in the natural way.
Clearly, $\cF$ is faithful, the left action of $J_{\cE}$ on the second direct summand $I$ of $\cF$ is zero, and $J_{\cE}\subset J_{\cF}=J_{\cE}+I$.
Observe that \[\cF \rca{\upsilon \oplus \alpha} G \cong (\cE \rca{\upsilon} G)\oplus (I\rca{\alpha}G).\]
Since $\cF$ is faithful, so is $\cF \rca{\upsilon \oplus \alpha} G$.
The kernel of the left action of $A\rca{\alpha} G$ on $\cE\rca{\upsilon} G$ contains $I\rca{\alpha}G$\footnote{The kernel coincides with $I\rca{\alpha}G$
when $G$ is exact, but we only use this obvious inclusion.}, 
and hence $J_{\cE \rca{\upsilon} G} \cdot (I \rca{\alpha} G)=\{0\}$ by definition.
This yields $J_{\cE\rca{\upsilon} G} \subset J_{\cF\rca{\upsilon\oplus \alpha} G}$.
By Katsura's gauge-invariant uniqueness theorem (\cite{Kat04}, Theorem 6.4), we obtain the canonical inclusions
\[\Cn{\cE} \subset \Cn{\cF}, \quad \Cn{\cE\rca{\upsilon} G} \subset \Cn{\cF\rca{\upsilon\oplus \alpha} G}.\]
Thus, the reduced Hao--Ng isomorphism for $\cF$
restricts to that for $\cE$.

Now, let $(\cE, \upsilon)$ be a faithful $G$-\Cs-correspondence.
Note that $J_\cE=\rho^{-1}(\IK(\cE))\lhd A$.
Hence, the left multiplication action of $A+\IK(\cE) \subset \Cn{\cE}$ on $\cE \subset \Cn{\cE}$ is faithful.
By the inclusion $\IB(\cE)_{\rm c} \rca{\alpha_\upsilon} G \subset \IB(\cE\rca{\upsilon} G)$,
the left multiplication action $(A+\IK(\cE))\rca{\alpha_\upsilon} G \rightarrow \IB(\cE\rca{\upsilon}G)$ is faithful.
Thus, for any $a\in A \rca{\alpha} G$ and $x\in \IK(\cE)\rca{\alpha_\upsilon} G =\IK(\cE\rca{\upsilon}G)$,
the condition $(a-x)\cdot (\cE\rca{\upsilon}G)=0$ implies $a=x$ in $\Cn{\cE}\rca{\alpha_\upsilon}G$.
Therefore, the pair of inclusion maps $A \rca{\alpha}G, \cE\rca{\upsilon} G \subset \Cn{\cE}\rca{\alpha_\upsilon} G$
defines an injective covariant representation of $\cE\rca{\upsilon} G$.
By the universal property of Cuntz--Pimsner algebras, this extends uniquely to a $\ast$-homomorphism 
\[\Cn{\cE\rca{\upsilon} G} \rightarrow \Cn{\cE}\rca{\alpha_\upsilon} G.\]
 By Katsura's gauge-invariant uniqueness theorem, this map is an isomorphism.
\end{proof}
\section{Counterexamples to the full Hao--Ng isomorphism problem}\label{section:HN2}
In this section, we record a short proof of the full isomorphism-type result established by Katsoulis and Ramsey for hyperrigid \Cs-correspondences (\cite{KR}, Theorem 3.10) in the regular case.
Recall that a \Cs-correspondence is said to be regular if it is faithful and proper.
The proof of the isomorphism is essentially the same as that for the reduced case in Section \ref{section:HN}.
Note that, in contrast to prior work, we do not require the non-degeneracy assumption.

More significantly, we disprove the full Hao--Ng isomorphism conjecture by constructing (full, regular, and non-degenerate) counterexamples.
This resolves several open problems posed in the literature; see, e.g., \cite{BKQR}, \cite{KR19}, \cite{KR}, \cite{DT}.
The failure of the conjectured isomorphisms stems from the fact that the full crossed product functor
may send an injective morphism to a non-injective one.
The results in this section are not used elsewhere in this paper.

Throughout this section, the symbol ``$\rtimes$'' denotes the full crossed product.
Let $G$ be a locally compact (Hausdorff) group.

\begin{Def}[A modified full Hao--Ng \Cs-correspondence; see Section 7.1 in \cite{KR19}.]
We define $\cE\hat{\rtimes}_{\upsilon}G$ to be the norm closure of $C_c(G, \cE)$ in $\Cn{\cE}\rtimes_{\alpha_\upsilon} G$.
We write $A\rtimes ^\upsilon _{\alpha}G$\footnote{In \cite{KR}, this completion is denoted by $A\hat{\rtimes}_\alpha G$. However we use the present notation as it depends
not only on $(A, \alpha)$ but also on $(\cE, \upsilon)$.} for the norm closure of $C_c(G, A)$ in $\Cn{\cE}\rtimes_{\alpha_\upsilon} G$, which is a \Cs-algebra.
Then, as in the reduced case, $\cE\hat{\rtimes}_{\upsilon} G$ forms a \Cs-correspondence over $A\rtimes ^\upsilon _{\alpha}G$.
\end{Def}
For a regular \Cs-correspondence $\cE$, clearly $\cE\hat{\rtimes}_{\upsilon}G$ is faithful.
\begin{Rem}
\begin{enumerate}
\item The full Hao--Ng \Cs-correspondence $\cE\rtimes_{\alpha_\upsilon} G$ \cite{HN}
is defined over the full crossed product $A\rtimes_\alpha G$.
Analogously to the reduced case, it can be defined as the norm closure of $C_c(G, \cE)$ in $\mathcal{T}_{\cE} \rtimes_{\alpha_\upsilon} G$ (see Remark 7.8 in \cite{KR19}).
This definition is used even for the full Hao--Ng isomorphism problem for non-amenable groups in the literature (see e.g., \cite{BKQR}, \cite{KR19}, \cite{KR}).

\item The \Cs-correspondence $\cE\hat{\rtimes}_{\upsilon} G$ coincides with $\cE\rtimes_{\alpha_\upsilon} G$
if and only if the canonical map $A\rtimes_\alpha G \rightarrow \Cn{\cE}\rtimes_{\alpha_\upsilon} G$ is injective.
This holds, for instance, when the Katsura ideal is zero.
In general, however, this injectivity fails, as shown in Theorem \ref{Thm:counterex} below.
\end{enumerate}
\end{Rem}

Applying the same strategy as in Section \ref{section:HN},
we revisit Theorem 3.10 of \cite{KR} in the regular case.
Since the proof is identical to that of Theorem \ref{Thm:HN2}, we omit it.
Note that in the proof, we use the assumptions only to deduce the faithfulness of the left action of $(A+\IK(\cE)) \rtimes_\alpha^\upsilon G$ on $\cE\hat{\rtimes}_\upsilon G$.
\begin{Thm}\label{Thm:HN2}
For any regular $G$-\Cs-correspondence $(\cE, \upsilon)$ over a $G$-\Cs-algebra,
the inclusion map $\cE\hat{\rtimes}_{\upsilon}G \rightarrow \Cn{\cE}\rtimes_{\alpha_\upsilon} G$
extends to a \Cs-algebra isomorphism
\[\Cn{\cE\hat{\rtimes}_{\upsilon}G}\cong \Cn{\cE}\rtimes_{\alpha_\upsilon} G.\]
\end{Thm}
Now, we show that the (original) full Hao--Ng isomorphism fails, disproving both the original conjecture \cite{HN} (see also \cite{BKQR}, \cite{KR}) and one half of its refined version in \cite{KR19}.
The following counterexamples show that this failure occurs even for discrete exact group actions on regular (hence hyperrigid), full, and non-degenerate \Cs-correspondences. Our construction is inspired by \cite{Suzeq}.
For basic facts about the amenability of \Cs-dynamical systems, we refer the reader to \cite{OS}.
Note that Katsura's refined Cuntz--Pimsner algebra $\Cn{\cE}$ satisfies
$A\subset \Cn{\cE}$ even for a non-faithful \Cs-correspondence $\cE$ over $A$ (see \cite{Kat04}, Proposition 4.11). 
\begin{Thm}\label{Thm:counterex}
For any non-amenable exact discrete group $\Gamma$,
there exist a $\Gamma$-\Cs-algebra $(A, \alpha)$ and a $\Gamma$-equivariant $\ast$-homomorphism $\rho \colon A \rightarrow A$
such that the canonical map $A\rtimes_\alpha \Gamma \rightarrow \Cn{\rho} \rtimes_\alpha \Gamma$ is not injective.
More precisely, one can find such a $\rho$ so that
$A\rtimes_\alpha\Gamma \neq A\rca{\alpha}\Gamma$, while the action $\Gamma \acts \Cn{\rho}$ is amenable.
\end{Thm}
\begin{proof}
Choose an amenable action $\Gamma \acts B$ on a unital \Cs-algebra $B$.
Let $\widetilde{\Gamma}< \prod_{\IN} \Gamma$
be the subgroup generated by $\bigoplus_{\IN} \Gamma$ and the constant sequences.
Set
\[A:=\left(\bigotimes_{\IN} B\right) \rca{\beta} \widetilde{\Gamma},\]
where $\beta \colon \widetilde{\Gamma} \acts \bigotimes_{\IN} B$ denotes the restriction of the tensor product action.
We identify $\Gamma$ with the subgroup of $\widetilde{\Gamma}$ consisting of all constant sequences.
Denote by $\alpha \colon \Gamma \acts A$ the restriction of the conjugation action.
Since the action arises from a unitary representation in $A$,
one has
\[A\rtimes_\alpha \Gamma \neq A\rca{\alpha} \Gamma.\]

We next define a $\Gamma$-equivariant $\ast$-homomorphism $\rho \colon A \rightarrow A$ as follows.
Let $\theta\colon \bigotimes_{\IN} B \rightarrow \bigotimes_{\IN} B$ be the one-sided shift:
$\theta(x):=1_{B}\otimes x$.
Let $\sigma\colon \widetilde{\Gamma} \rightarrow \widetilde{\Gamma}$
denote the (well-defined) group homomorphism given by the index shift
\[\sigma(s_0, s_1, \ldots):=(1_\Gamma, s_0, s_1, \ldots).\]
Note that both $\theta$ and $\sigma$ are injective and that
$\theta$ is $\beta$-($\beta \circ \sigma$)-equivariant.
Hence, we obtain a $\ast$-endomorphism $\rho \colon A\rightarrow A$ satisfying
\[\rho(x s):=\theta(x) \sigma(s) \quad \text{for~} x \in \bigotimes_{\IN} B, s\in \widetilde{\Gamma}.\]
Clearly, $\rho$ is $\Gamma$-equivariant.

Observe that the $0$-th copy of $B$ in $\bigotimes_{\IN} B \subset A$ commutes with $\rho(A)$.
Since $\Gamma \acts B$ is amenable, this shows that the inductive limit action
\[\Gamma \acts \fB:=\varinjlim_{n\in \IN} (A, \rho)\]
is amenable, since there exists a unital $\Gamma$-equivariant embedding of $B$ into the central sequence algebra of $\fB$ (cf.~ Section 4 in \cite{OS}).
Since $\fB$ is $\Gamma$-equivariantly isomorphic to the gauge fixed point algebra $\Cn{\rho}^{\IT}$, the induced action $\Gamma \acts \Cn{\rho}^{\IT}$ is amenable.
Hence, so is $\Gamma \acts \Cn{\rho}$ by Theorem 5.1 of \cite{OS}. 
\end{proof}
\begin{Rem}
Here, we record consequences of the construction in Theorem \ref{Thm:counterex}
\begin{enumerate}
\item If $A_0$ is exact and the full group \Cs-algebra $\Cso(\Gamma)$ is non-exact (e.g., when $\Gamma$ contains a non-abelian free group),
then $\Cn{{}_\rho A\hat{\rtimes}_\alpha \Gamma} \cong \Cn{\rho}\rtimes_\alpha \Gamma= \Cn{\rho}\rca{\alpha} \Gamma$
is exact, while $\Cn{{}_\rho A \rtimes_\alpha \Gamma}$ is non-exact since it contains $\Cso(\Gamma)$.
Thus, even abstractly, the expected isomorphism $\Cn{\cE\rtimes_\upsilon \Gamma} \cong \Cn{\cE}\rtimes_{\alpha_\upsilon}\Gamma$ fails.
\item The original full isomorphism for non-amenable groups is known to hold only for proper graph \Cs-correspondences \cite{KR}.
Our counterexamples demonstrate that the conjectured isomorphism can fail drastically even in the setting of a single $\ast$-endomorphism on an inner \Cs-dynamical system,
highlighting the difficulty of obtaining a reasonable characterization of its validity.
\item
By considering the induced \Cs-dynamical system,
the statement of Theorem \ref{Thm:counterex}
extends to locally compact groups containing a non-amenable exact discrete subgroup.
For instance, this applies to non-amenable almost connected groups, by Theorem 5.5 of \cite{Ric}.

\item
In \cite{OS}, it is shown that for a $G$-\Cs-correspondence $\cE$ over an amenable $G$-\Cs-dynamical system $(A, \alpha)$,
the generalized quasi-free action $G\acts \Cn{\cE}$ is amenable.
Theorem \ref{Thm:counterex} shows that the converse fails.
\end{enumerate}
\end{Rem}

\section{Traces and KMS weights on Cuntz--Pimsner algebras}
In this section, we study traces and KMS weights on the Cuntz--Pimsner algebras.
We first study tracial weights. Then the study of KMS weights is deduced to this study by using the Takesaki--Takai duality (see \cite{KK}, \cite{Thom}) and the Hao--Ng isomorphism \cite{HN} (see also Theorem \ref{Thm:HN}).
These results recover the main result of Laca--Neshveyev \cite{LN}.
An advantage of our approach is that, unlike the strategy in \cite{LN}, weights can be directly constructed without passing to the Toeplitz extension.
We also give a new sufficient condition for a \Cs-correspondence to have automatic gauge-invariance of traces and KMS weights on the Cuntz--Pimsner algebra.
The condition, which we call tracial proper outerness, involves W$^\ast$-completions. The definition and our proofs are inspired by \cite{Ur} and references therein.

We note that a strong Morita equivalence gives rise to an affine isomorphism betweeen the trace spaces,
and an $\IR$-equivariant strong Morita equivalence gives rise to affine isomorphisms between the KMS spaces.
Hence by Proposition \ref{Prop:st} and Remark \ref{Rem:st}, by passing to the stabilization, the essential part of our proof is deduced to the case that $\mathcal{E}={}_\rho A$ for some $\rho \colon A\rightarrow \cM(A)$.
This together with Theorem \ref{Thm:decomp} makes many notations and calculations rather simple.
The statements for general \Cs-correspondences will be recorded at the end of this section.

Let $A$ be a \Cs-algebra.
Let $\rho \colon A\rightarrow \cM(A)$ be a faithful non-degenerate $\ast$-homomorphism.
We adapt the notations in Section \ref{section:st}.
The \Cs-dynamical system $(\fB, \sigma)$ introduced therein plays a fundamental role.

For basic facts on weights on \Cs-algebras, we refer the reader to the books \cite{Pedbook} and \cite{Thom}.
Throughout the article, all weights on \Cs-algebras, except on von Neumann algebras and multiplier algebras, are supposed to be densely defined and lower semi-continuous.
Following \cite{Thom}, we include the flow invariance in the axioms of a KMS weight.
\begin{no}

Here we fix some notations. We note that some of them are different from those in \cite{Thom} for our convenience.
\begin{itemize}
\item For a \Cs-algebra $A$, let ${\rm Ped}(A)$ denote its Pedersen ideal.
\item
The GNS representation of a weight $\varphi$ is denoted by $\pi_\varphi$.
\item For a weight $\varphi$, one has a unique normal, semi-finite, faithful weight $\bar{\varphi}$ on $\pi_\varphi(A)''$
with $\bar{\varphi}\circ \pi_\varphi= \varphi$. We write $\varphi^{\cM}$ for the composite $\bar{\varphi}\circ \pi_\varphi|_{\cM(A)_+} \colon \cM(A)_+ \rightarrow [0, \infty]$.
Note that $\varphi^{\cM}$ is not densely defined unless $\varphi$ is bounded, but satisfies the other axioms of a weight.
Note also that, by the normality of $\bar{\varphi}$, with $(e_n)_n$ an approximate unit of $A$,
one has
\[\varphi^\cM(a)=\lim_{n\rightarrow \infty} \varphi(a^{\frac{1}{2}} e_n a^{\frac{1}{2}}) \quad {\rm~for~all~}a\in \cM(A)_+.\]
\item Let ${\rm T}(A)$ denote the space of tracial weights on $A$.
\item We often identify $\tau \in {\rm T}(A)$ with a tracial positive linear functional on $\Ped(A)$ with the same symbol; see \cite{Pedbook} for details.
\item For a $\ast$-homomorphism $\rho\colon A \rightarrow \cM(A)$, we set
\[{\rm T}_\rho(A):=\{\tau \in {\rm T}(A): \tau^{\cM} \circ \rho \leq \tau^\cM, \quad \tau \circ \rho \equiv \tau {\rm~on~} A_+\cap \rho^{-1}(A)\}.\]
Note that when $\rho(A) \subset A$, one has
\[{\rm T}_\rho(A)=\{\tau \in {\rm T}(A): \tau \circ \rho = \tau\}.\]
\item For a flow $\alpha \colon \IR \acts A$ and $\beta \in \IR$,
let ${\rm KMS}^\beta(A, \alpha)$ denote the space of all $\beta$-KMS weights on $(A, \alpha)$.
We set
\[{\rm T}^\beta(A, \alpha):=\{\tau \in {\rm T}(A): \tau\circ \alpha_t=e^{-\beta t} \tau~{\rm~for~all~}t\in \IR\}.\]
We also denote ${\rm KMS}^0(A, \alpha)={\rm T}^0(A, \alpha)$, the space of all $\alpha$-invariant tracial weights,
by ${\rm T}_\alpha(A)$.
Note that when $\alpha$ is trivial, we have
\[{\rm KMS}^\beta(A, \alpha)={\rm T}(A)\]
for all $\beta \in \IR$.

\item When $\rho$ is $\alpha$-equivariant, we set
\[{\rm KMS}_\rho^\beta(A, \alpha):=\{\varphi \in {\rm KMS}^\beta(A, \alpha): \varphi ^{\cM} \circ \rho \leq \varphi ^{\cM}, \quad \varphi\circ \rho \equiv \varphi {\rm~on~} A_+\cap \rho^{-1}(A)\},\]
\[{\rm T}_\rho^\beta(A, \alpha):= {\rm T}_\rho(A)\cap {\rm T}^\beta(A, \alpha).\]
\end{itemize}
Note that all these sets of weights form a cone.
\end{no}
\begin{Rem}
Even when $\rho\colon A \rightarrow \cM(A)$ is non-proper,
the space ${\rm T}_\rho(A)$ can be large. Here we illustrate an example.

Let $X$ be a locally compact space with a Borel regular measure $\mu$.
Let $\tau \in {\rm T}(A)$.
Assume that one has a trace-preserving, injective, non-degenerate $\ast$-homomorphism
$\theta \colon (C_0(X)\otimes A, \mu \otimes \tau) \rightarrow (\cM(A), \tau^\cM)$.
(For instance, when $X$ is a compact metrizable space, $\mu$ is a faithful probability measure on $X$,
$A$ is any $\cZ$-stable \Cs-algebra, there exists an embedding
$\theta \colon C(X) \otimes A \rightarrow A$ with $\tau \circ \theta=\mu \otimes \tau$ for all $\tau \in {\rm T}(A)$
by the existence part of the classification theorem; applying Theorem B in \cite{CGS} to $C(X)$ and $\cZ$ in place of $A$ and $B$ therein.)
Let $(\varrho_x)_{x\in X}$ be a point-strict continuous family of $\ast$-endomorphisms $\varrho_x \colon A \rightarrow \cM(A)$.
Assume that
\begin{equation}\label{equation:harmonic}
\int_X \tau^\cM \circ \varrho_x {\rm d}\mu(x) =\tau.
\end{equation}
Define a $\ast$-homomorphism $\varrho \colon A\rightarrow \cM(C_0(X)\otimes A)$
 by sending $a\in A$ to the function $(X\ni x \mapsto \varrho_x(a))$.
Set $\rho:=\theta \circ \varrho \colon A \rightarrow \cM(A)$.
Then one has $\tau^\cM \circ \rho = \tau^\cM$ by the choice of $(\varrho_x)_{x\in X}$ and $\theta$.
If we assume in addition that $\varrho(A) \cap (C_0(X)\otimes A) =0$,
then one has $\rho(A) \cap A =0$.
Hence, even after relaxing the equality (\ref{equation:harmonic}) to the inequality
\begin{equation}\label{equation:harmonic2}
\int_X \tau^\cM \circ \varrho_x {\rm d}\mu(x) \leq \tau,
\end{equation}
 we still have $\tau\in {\rm T}_\rho(A)$.
The case $X$ is discrete, $\mu$ is the counting measure on $X$, and that $\theta\colon c_0(X)\otimes A\otimes \IK \rightarrow A\otimes \IK$
is the embedding induced from the diagonal embedding $c_0(X)\otimes \IK \rightarrow \IK(\ell^2(X))\otimes \IK \cong \IK$,
 is studied in \cite{SuzP} to construct amenable actions on finite simple \Cs-algebras.
The restriction of a trace-scaling flow on a $\cZ$-stable \Cs-algebra to a suitable interval also fulfills these conditions.
Our present results also apply to the continuous case, which illustrates an advantage of the present approach.
\end{Rem}

We use the next lemma to construct tracial weights.
When $(B_\lambda)_{\lambda \in \Lambda}$ is an increasing net of hereditary \Cs-subalgebras with a dense union,
the unions $\bigcup_{\lambda \in \Lambda} B_\lambda$ and $\bigcup_{\lambda \in \Lambda} \Ped(B_\lambda)$ satisfy
the assumption below.
\begin{Lem}\label{Lem:traceext}
Let $B\subset A$ be a dense $\ast$-subalgebra
satisfying $\cl(eAe) \subset B$ for all $e\in B_+$.
Then any tracial positive linear functional $\tau_0 \colon B \rightarrow \IC$
extends to a tracial weight $\tau$ on $A$.
\end{Lem}
\begin{proof}
By assumption, $B$ is closed by (non-unital) continuous functional calculus.
Hence it contains an approximate unit $(e_n)_n$ of $A$.
Since $\cl(e_n A e_n) \subset B$ is a \Cs-algebra, $\tau_0|_{\cl(e_n A e_n)}$ is a bounded tracial weight.
Hence it extends to
a tracial weight $\tau_n$ on the ideal $I_n \lhd A$ generated by $e_n$.
The desired tracial weight $\tau$ is given by the formula
\[\tau(a):=\lim_{n\rightarrow \infty} \tau_n(a^{\frac{1}{2}} e_n a^{\frac{1}{2}}) \quad {\rm~for~}a\in A_+.\]
\end{proof}
The next lemma is useful to classify tracial weights.
\begin{Lem}\label{Lem:traceunique}
Let $B\subset \Ped(A)$ be a dense $\ast$-subalgebra.
Then for $\tau, \omega\in {\rm T}(A)$, the equality $\tau|_B=\omega|_B$ implies $\tau = \omega$.
\end{Lem}
\begin{proof}
By assumption, one can find a not-necessary-increasing approximate unit $(e_n)_n$ in $(B_+)_1$.
Since $e_n\in \Ped(A)$,
both $\tau$ and $\omega$ are bounded on $\cl(e_n A e_n)$.
Hence these two weights coincide on each $\cl(e_n A e_n)$.
Then, as both $\tau, \omega$ are tracial and lower semi-continuous, for $a\in A_+$, one has
\[\tau(a)=\lim_{n\rightarrow \infty} \tau(e_n a e_n) = \lim_{n\rightarrow \infty} \omega(e_n a e_n) = \omega(a).\]

\end{proof}

The next lemma is a standard application of continuous functional calculus. For completeness, we include a proof.
\begin{Lem}\label{Lem:positive}
Let $B_1, \ldots, B_n$ be \Cs-subalgebras of $A$ with $B_i \cdot B_j \subset B_{\max\{i, j\}}$ for $i, j$, and $B_1+ B_2 +\cdots +B_n=A$.
Then for any $a\in A_+$, one has $b_i \in B_i$, $i=1, \ldots, n$, with
\[a=\sum_{i=1}^n b_i, \quad \sum_{i=1}^k b_i \geq 0\quad {\rm~for~all~} 1\leq k\leq n.\]
\end{Lem}
\begin{proof}
We prove the statement by induction on $n$.
The case $n=1$ is trivial.
For $n \geq 2$, assume that the claim holds true for $n-1$.
Let $B_1, \ldots, B_n$ be as in the statement.
Set $C:= B_1 + B_2 + \cdots + B_{n-1}$.
Then $C$ is a \Cs-subalgebra of $A$ with $C \cdot B_n \subset B_n$, $C+B_n= A$.
Let $a\in A_+$.
Pick $c'\in C$ and $b'_n\in B_n$ with $a=c' + b'_n$.
By considering the quotient map $A \rightarrow A/B_n = C/(B_n\cap C)$, one has
$c'-{\rm Re}(c')_+\in B_n \cap C$. Put $c:= {\rm Re}(c')_{+}\in C_+$, $b_n:= b'_n +c' -{\rm Re}(c')_{+} \in B_n$.
Then we have $a=c+ b_n$.
By applying the induction hypothesis to
$B_1, \ldots, B_{n-1} \subset C$ and $c\in C_+$,
one can find $b_i \in B_i$; $i=1, \ldots, n-1$,
with
\[c=\sum_{i=1}^{n-1} b_i, \quad \sum_{i=1}^k b_i \geq 0 \quad {\rm~for~} 1 \leq k \leq n-1.\]
The sequence $(b_i)_{i=1}^n$ satisfies the desired conditions.
\end{proof}
We first study tracial weights on $\fB$.
\begin{Lem}\label{Lem:tr}
There is an affine isomorphism from ${\rm T}(\fB )$ onto
the space of all sequences $(\tau_n)_{n\in \IZ}$ in ${\rm T}(A)$ satisfying
\begin{equation}\label{equation:compatible}
\tau_{n+1}^\cM \circ \rho \leq \tau_{n}, \quad \tau_{n+1}(\rho(a)) = \tau_{n}(a) {\rm~for~}a\in A_+\cap \rho^{-1}(A) \quad{\rm~for~all~}n\in \IZ.\end{equation}
The correspondence is given by
sending $\tau\in {\rm T}(\fB )$ to the sequence $(\tau\circ \iota_n)_{n\in \IZ}$.
\end{Lem}
\begin{proof}
We first show that the map is well-defined.
Let $\tau \in {\rm T}(\fB)$.
Then clearly the sequence $(\tau \circ \iota_n)_{n\in \IZ}$ satisfies the second condition in (\ref{equation:compatible}).
To show the first condition, pick $a\in A_+$ and an approximate unit $(e_k)_{k\in \IN}$ of $A$.
Then one has
\begin{eqnarray*}
\tau_n(a)&=&\tau(\iota_n(a))=\tau(\iota_{n+1}(\rho(a)))\\
&\geq& \sup_{k\in \IN} \tau(\iota_{n+1}(\rho(a)^{\frac{1}{2}} e_k \rho(a)^{\frac{1}{2}}))\\
&=& \tau^\cM_{n+1}(\rho(a)).
\end{eqnarray*}
This proves the claim. Clearly the map is affine.
The injectivity of the map follows from
Lemme \ref{Lem:traceunique}, applying to $B=\spa\{\iota_n(\Ped(A)):n \in \IZ\} \subset \Ped(\fB)$.

To show the surjectivity, let $(\tau_n)_{n\in \IZ}$ be a sequence in ${\rm T}(A)$ satisfying (\ref{equation:compatible}).
We fix an approximate unit $(e_n)_{n\in \IN}$ in $\Ped(A)$ with $e_n e_{n+1}=e_n$ for $n\in \IN$.
Note that as $e_n\in \Ped(A)_+$, for each $m, k\in \IN$, the weight $\tau_{m+k}$ is bounded on $\rho^k(e_n) A \rho^k(e_n)$
(whose norm is at most $\tau(\iota_m(e_{n+1})) < \infty$).
Hence $\tau_{m+k}$ has a continuous extension to $\cl(\rho^k(e_n) A \rho^k(e_n))$, which we denote by the same symbol $\tau_{m+k}$.
Let $k\in \IN$.
For $a\in \cl(\iota_{m}(e_n)B_{[m, m+ k]}\iota_{m}(e_n))$, write
\[a= \sum_{j=0}^k\iota_{m+j}(a_j);\quad a_j \in \cl(\rho^j(e_n)A \rho^j(e_n)) {\rm~ for~} 0 \leq j \leq n,\]
and set
\[\omega_{n, m, k}(a):=\sum_{j={0}}^k \tau_{m+j}(a_j).\]
By the second condition in (\ref{equation:compatible}), $\omega_{n, m, k}$ is well-defined.
By definition, $\omega_{n, m, k}$ is linear and self-adjoint, and satisfies $\omega_{n, m, k} \circ \iota_{m+l} \equiv \tau_{m+l}$ on
$\cl(\rho^l(e_n)A \rho^l(e_n))$ for $0 \leq l \leq k$.

We show that $\omega_{n, m, k}$ defines a bounded tracial weight on $\cl(\iota_{m}(e_n)B_{[m, m+ k]}\iota_{m}(e_n))$.
For $0\leq j \leq l \leq n$, $a\in \cl(\rho^j(e_n)A \rho^j(e_n))$, $b\in \cl(\rho^l(e_n)A \rho^l(e_n))$, one has
\[\omega_{n, m, k}(\iota_{m+j}(a)\iota_{m+l}(b))=
\tau_{m+l}(\rho^{l-j}(a)b) = \tau_{m+l}(b\rho^{l-j}(a))
=\omega_{n, m, k}(\iota_{m+l}(b)\iota_{m+j}(a)).\]
Hence $\omega_{n, m, k}$ is tracial.
We next show that $\omega_{n, m, k}$ is positive. Pick $0 \leq a\in \cl(\iota_m(e_n)B_{[m, m+ k]}\iota_m(e_n))$.
By Lemma \ref{Lem:positive}, one can choose $a_j\in \cl(\rho^j(e_n)A\rho^j(e_n))$; $j=0, 1, \ldots, k$, satisfying
\begin{equation}\label{equation:decompa}
a=\sum_{j=0}^k \iota_{m+j}(a_j), \quad \sum_{j=0}^l \iota_{m+j}(a_j) \geq 0
\end{equation}
for $0 \leq l \leq k$.
Note that the second condition implies $\sum_{j=0}^l\rho^{l-j}(a_j) \geq 0$ for $0 \leq l \leq k$.
Then, by the first condition in (\ref{equation:compatible}) and (\ref{equation:decompa}), one has
\begin{eqnarray*}
\omega_{n, m, k}(a)&=&\sum_{j=0}^k \tau_{m+j}(a_j)\\
&\geq&\tau_{m+1}^\cM (\rho(a_0))+\tau_{m+1}(a_1) + \sum_{j=2}^{k} \tau_{m+j}(a_j) \\
&=&\tau_{m+1}^\cM (\rho(a_0)+a_1) + \sum_{j=2}^{k} \tau_{m+j}(a_j) \\
&\geq& \tau_{m+2}^\cM( \rho^2(a_0)+\rho(a_1)+a_2)+ \sum_{j=3}^{k} \tau_{m+j}(a_j)\\
&\vdots&\\
&\geq& \tau^{\cM}_{m+k}(\sum_{j=0}^k \rho^{k-j}(a_j)) \geq 0.
\end{eqnarray*}
This proves the positivity of $\omega_{n, m, k}$.
Hence the formula indeed defines a well-defined bounded tracial weight on $\cl(\iota_m(e_n) B_{[m, m+k]} \iota_m(e_n))$.

Clearly $\omega_{n, m, k+1}$ extends $\omega_{n, m, k}$, and their norms are bounded by $\tau_m(e_{n+1}) <\infty$.
Hence one has a (bounded) tracial weight $\omega_{n, m}$ on $\cl(\iota_m(e_n) \fB \iota_m(e_n))$ which extends $\omega_{n, m, k}$'s.
Then, for each $m\in \IN$, by Lemma \ref{Lem:traceext},
one has $\omega_m\in {\rm T}(B_{[m, \infty]})$ which extends $\omega_{n, m}$'s.
It is not hard to check that $\omega_{m-1}$ extends $\omega_m$ for $m\in \IZ$.
Hence by Lemma \ref{Lem:traceext}, one has $\omega\in {\rm T}(\fB)$ which extends $\omega_m$'s.
This trace satisfies $\omega \circ \iota_n = \tau_n$ for all $n\in \IZ$.

\end{proof}
\begin{Lem}\label{Lem:invtr}
There is an affine isomorphism from ${\rm T}_\sigma(\fB )$ onto ${\rm T}_\rho(A)$.
The correspondence is given by
sending $\tau\in {\rm T}_\sigma(\fB )$ to the restriction $\tau|_A$.
\end{Lem}
\begin{proof}
By Lemma \ref{Lem:traceunique} (applying to $B=\spa\{\iota_n(\Ped(A)):n\in \IZ\}$), $\tau\in {\rm T}(\fB )$ is $\sigma$-invariant if and only
if it satisfies $\tau \circ \iota_{n+1}=\tau \circ \iota_{n}$ for all $n\in \IZ$.
This fact together with Lemma \ref{Lem:tr} proves the claim. 
\end{proof}

\begin{Cor}\label{Cor:gitrace}
There is an affine isomorphism from the space of
gauge-invariant tracial weights $\tau$ on $\Cn{\rho}$ onto
${\rm T}_\rho(A)$.
The correspondence is given by
sending $\tau$ to $\tau|_A$.
\end{Cor}\begin{proof}
By Theorem \ref{Thm:decomp}, it suffices to show the claim for $\fB\rca{\sigma}\IZ$ in place of $\Cn{\rho}$.
We note that any gauge-invariant tracial weight $\tau$ on $\fB\rca{\sigma}\IZ$ is of the form
$\tau_0 \circ E$, where $\tau_0:=\tau|_{\fB } \in {\rm T}_\sigma(\fB)$ and $E\colon \fB\rca{\sigma}\IZ \rightarrow \fB$ is the canonical conditional expectation,
and conversely, for any $\tau_0\in {\rm T}_\sigma(\fB)$, the composite $\tau_0 \circ E$ forms a gauge invariant tracial weight on $\fB\rca{\sigma}\IZ$.
Thus the statement follows from Lemma \ref{Lem:invtr}.
\end{proof}
When the positive powers of $\rho$ are sufficiently outer in the following sense, Corollary \ref{Cor:gitrace} in fact gives a complete description of ${\rm T}(\Cn{\rho})$.
\begin{Def}
For $\tau \in {\rm T}_\rho(A)$, we say that $\rho$ is \emph{properly $\tau$-outer}, if it satisfies the following condition:
There is no nonzero $x\in \pi_\tau(A)''$ satisfying
$x\pi_\tau(a)=\pi_\tau(\rho(a))x$ for all $a\in A$.

We say that $\rho$ is \emph{tracially properly outer}
if it is properly $\tau$-outer for all $\tau\in {\rm T}_\rho(A)$.
\end{Def}
\begin{Rem}
When $\rho$ is an automorphism,
our tracial proper outerness is equivalent
to the strong outerness defined e.g., in \cite{MS}, Definition 2.7.
\end{Rem}
\begin{Prop}\label{Prop:uniquetr}
Let $\tau_0 \in {\rm T}_\rho(A)$.
Assume that the positive powers of $\rho$ are properly $\tau_0$-outer.
Then there is a unique $\tau \in {\rm T}(\Cn{\rho})$ with $\tau|_A=\tau_0$.
\end{Prop}
\begin{proof}
By Theorem \ref{Thm:decomp}, it suffices to show the statement for $\fB\rca{\sigma} \IZ$ instead of $\Cn{\rho}$.

The existence of $\tau$ follows from Corollary \ref{Cor:gitrace}.
Hence we only need to show the uniqueness of extensions.

Let $\omega\in {\rm T}(\fB \rca{\sigma} \IZ)$ with $\omega|_A=\tau_0$.
Then, by Lemma \ref{Lem:invtr}, we have $\omega|_{\fB }=\tau|_{\fB }$.
To show that $\omega=\tau$, by Lemma \ref{Lem:traceunique} (applied to $B=\spa\{\iota_n(\Ped(A)) u^m: n, m\in \IZ\}$),
it suffices to show the equality $\omega(\iota_n(a)u^m)=0$ for all $a\in {\rm Ped}(A)$, $n\in \IZ$, and $m\geq 1$.
Since $\omega \circ \ad u=\omega$, it suffices to show the equality when $n=0$.
To lead to a contradiction, assume that
$\omega(a u^m)\neq 0$ for some $a\in {\rm Ped}(A)$ and $m\geq 1$.
Define $N$, $M$ to be the strong closure of $\pi_\omega(A)$ and $\pi_\omega(\Cn{\rho})$ respectively.
Let $p$ be the unit of $M$. Note that $p\in N$ and that $M=p\pi_\omega(\fB \rca{\sigma} \IZ)'' p$.
We equip them with the GNS tracial weights $\bar{\omega}$.
Then the inclusion $N\subset M$ preserves $\bar{\omega}$.
Since $\bar{\omega}$ is normal, faithful, and semi-finite on both $N$ and $M$,
one has a trace preserving conditional expectation
$E\colon M \rightarrow N$.
Put
\[x:=\pi_\omega(a),\quad y:=E(p\pi_\omega(u^m)p)\in N.\]
Note that
\[\bar{\omega}(xy)=\bar{\omega}(xp\pi_\omega(u^m)p)= \bar{\omega}(x\pi_\omega(u^m))\neq 0\]
since $\bar{\omega}$ is tracial, $px = xp =x$, and $\bar{\omega}\circ E =\bar{\omega}$.
Hence $y\in N \setminus \{0\}.$
For any $b\in A$, one has
\[\pi_\omega(b)y=E(p\pi_\omega(bu^m)p)=E(p\pi_\omega(u^m)p\pi_\omega (\rho^m(b))p)=y\pi_\omega(\rho^m(b))p,\]
where the last equality follows because $\pi_\omega(\cM(A))p\subset N$.
Since $\omega|_A=\tau_0$, we have a natural $\ast$-isomorphism between $N$ and $\pi_{\tau_0}(A)''$.
This contradicts to the proper $\tau_0$-outerness of $\rho^m$.
\end{proof}
By combining Corollary \ref{Cor:gitrace} and Proposition \ref{Prop:uniquetr}, we obtain the following consequence.
\begin{Thm}\label{Thm:trace}
Assume that
the positive powers of $\rho$ are tracially properly outer.
Then the map $\bullet|_A\colon {\rm T}(\Cn{\rho}) \rightarrow {\rm T}_\rho(A)$
gives an affine isomorphism.
\end{Thm}

We now study KMS weights for generalized quasi-free flows.
Assume that $A$ is equipped with a flow $\alpha \colon \mathbb{R} \acts A$.
Let $\cE$ be an $\mathbb{R}$-\Cs-correspondence over $(A, \alpha)$.
By Remark \ref{Rem:st}, after passing to the stabilization,
we may assume that $\cE={}_\rho A$, where $\rho\colon A \rightarrow \cM(A)$ is $\mathbb{R}$-equivariant and the action $\IR \acts {}_\rho A$ (as a set) is equal to $\alpha$.
We use the following result recorded in the book \cite{Thom} (where an explicit correspondence is given).
\begin{Thm}[\cite{Thom}, Theorem 7.2.14]\label{Thm:KMSvstr}
Let $\alpha \colon \IR \acts A$ be a flow on a \Cs-algebra $A$.
Let $\hat{\alpha} \colon \IR \acts A\rca{\alpha} \IR$ denote the dual flow of $\alpha$.
Then there is an affine isomorphism
from ${\rm KMS}^\beta(A, \alpha)$ onto ${\rm T}^\beta(A\rca{\alpha}\IR, \hat{\alpha})$.
\end{Thm}
The next statement would be well-known for experts. (At least when the flow is almost periodic or the flow admits an invariant full element in ${\rm Ped}(A)$, it is easy to show the claim.)
Here we give a proof in full generality, based on Theorem \ref{Thm:KMSvstr}.
\begin{Prop}\label{Prop:rest}
Let $B\subset (A, \alpha)$ be an inclusion of $\IR$-\Cs-algebras.
Then for any $\beta \in \IR$ and any $\varphi \in {\rm KMS}^\beta(A, \alpha)$,
one has $\varphi|_B \in {\rm KMS}^\beta(B, \alpha)$. That is, the restriction weight is still densely defined on $B$.
\end{Prop}
\begin{proof}
Note that the statement holds true for tracial weights instead of KMS weights.
Then observe that the restriction map coincides with the composite of the following sequence
\[{\rm KMS}^\beta(A, \alpha) \cong {\rm T}^\beta (A\rca{\alpha}\IR, \hat{\alpha}) \rightarrow {\rm T}^\beta (B\rca{\alpha}\IR, \hat{\alpha}) \cong {\rm KMS}^\beta(B, \alpha)\]
where the first and last maps are the isomorphisms given in Theorem \ref{Thm:KMSvstr}, and the middle map is the restriction map.
Hence it is well-defined.
\end{proof}

\begin{Thm}\label{Thm:KMS}
Let $\rho \colon A \rightarrow \cM(A)$ be an $\IR$-equivariant $\ast$-homomorphism.
Let ${\alpha_\upsilon} \colon \IR\acts \Cn{\rho}$ be the generalized quasi-free flow.
Let $\beta \in \IR$.
Then the map $\bullet|_{A}\colon {\rm KMS}^\beta(\Cn{\rho}, {\alpha_\upsilon}) \rightarrow {\rm KMS}^\beta_\rho(A, \alpha)$
is an affine surjection.
Moreover, the map is also injective if
the positive powers of the induced $\ast$-homomorphism
\[\rho \rca{\alpha} \IR\colon A\rca{\alpha}\IR \rightarrow \cM(A\rca{\alpha}\IR)\] are tracially properly outer.
\end{Thm}
\begin{proof}
Consider the following commutative diagram, where the vertical maps are the isomorphisms given in Theorem \ref{Thm:KMSvstr}
and the horizontal maps are the restriction maps (which are well-defined by Proposition \ref{Prop:rest}):
\[
 \begin{CD}
 {\rm KMS}^\beta(\Cn{\rho}, {\alpha_\upsilon}) @>{\bullet|_{A}}>> {\rm KMS}^\beta_\rho(A, \alpha) \\
 @VVV @VVV \\
 {\rm T}^\beta(\Cn{\rho}\rca{{\alpha_\upsilon}}\IR, \hat{{\alpha_\upsilon}}) @>{\bullet}|_{A\rca{\alpha}\IR}>> {\rm T}_{\rho \rca{\alpha} \IR}^\beta(A\rca{\alpha}\IR, \hat{\alpha}).
 \end{CD}
\]
By the Hao--Ng isomorphism \cite{HN}, $\Cn{\rho}\rca{{\alpha_\upsilon}}\IR$ is naturally isomorphic to $\Cn{\rho \rca{\alpha} \IR}$.
Then, by Theorem \ref{Thm:trace}, the bottom map is surjective.
Hence so is the top map.
When the positive powers of $\rho\rca{\alpha} \IR$ are tracially properly outer,
the bottom map is also injective by Theorem \ref{Thm:trace}. Hence so is the top map.
\end{proof}
Since the trace space and the KMS spaces are invariant under strong Morita equivalence,
by applying Theorems \ref{Thm:trace} and \ref{Thm:KMS} to the stabilizations, we conclude the following complete description results in the general form.
Before stating the results, we give a few more definitions.
\begin{Def}
Let $\cE$ be a \Cs-correspondence over $A$.
Let $\theta \colon {\rm T}(A) \rightarrow {\rm T}(\IK(\cE))$ be an affine isomorphism
induced from the strong Morita equivalence.
That is, for $\tau \in {\rm T}(A)$, $\theta(\tau)$ is a tracial weight on $\IK(\cE)$ satisfying $\theta(\tau)(e_{\xi, \eta})=\tau(\ip{\eta, \xi})$ for $\xi, \eta\in \cE \cdot \Ped(A)$.
We set
\[{\rm T}_\cE(A) := \{\tau \in {\rm T}(A): \theta(\tau)^\cM \circ (\bullet \otimes 1_\cE) \leq \tau, \quad \theta(\tau)(a\otimes 1_\cE)=\tau(a){\rm~for~}a\in A_+ {\rm~with~}
a\otimes 1_\cE \in \IK(\cE)\}.\]
We use the analogous notations for KMS weights.

Consider the dense subspace
\[\cE_0:=\left\{\xi\in \cE: \ip{\xi, \eta}\in {\rm Ped}(A){\rm~for~all~}\eta\in \cE \right\} \subset \cE.\]
For $\tau \in {\rm T}_\cE(A)$, let ${\cE}^{\tau}$ be the completion of the quotient space of $\cE_0$ given by the semi-inner product
$\ip{\xi, \eta}_\tau:=\tau(\ip{\xi, \eta})$ for $\xi, \eta\in \cE_0$.
The Hilbert space ${\cE}^{\tau}$ naturally forms a ${\rm W}^\ast$-bimodule over $\pi_\tau(A)''$.
We say that $\cE$ is properly $\tau$-outer, if
there is no nonzero $A$-bilinear operator in $\IB({\rm L}^2(\pi_\tau(A)''), {\cE}^\tau)$.
We say that $\cE$ is tracially properly outer, if it is properly $\tau$-outer for all $\tau \in {\rm T}_\cE(A)$.
\end{Def}

\begin{Exm}\label{Exm:trace}
Let $B$ be a \Cs-algebra with a unique tracial weight $\tau$ up to scalar.
Let $\sigma$ be an automorphism on $B$
whose induced automorphism on $\pi_\tau(B)''$ is outer.
Then, as any tracial weight $\varphi$ on $A\otimes B$ is of the form
$\omega \otimes \tau$ for some $\omega\in {\rm T}(A)$, for any \Cs-correspondence $\cE$,
the exterior tensor product $\cE\otimes ({}_\sigma B)$
is tracially properly outer.
An important choice is the Bernoulli shift $(\cZ^{\otimes \IZ}, \sigma)$ over the Jiang--Su algebra $\cZ$,
because the construction in this special case preserves K-theoretical data (Elliott invariants) for $\cZ$-stable \Cs-algebras,
as well as the (equivariant) Kasparov class and the ideal lattice in general.

\end{Exm}

We now state classification results of traces and KMS weights for general Cuntz--Pimsner algebras.
This immediately follows from Theorems \ref{Thm:trace}, \ref{Thm:KMS}, together with the following obvious observation.
\begin{Lem}
The following statements hold true.
\begin{itemize}
\item A $\ast$-homomorphism $\rho \colon A \rightarrow \cM(A)$ is tracially properly outer
if and only if the \Cs-correspondence ${}_\rho A$ is tracially properly outer.
\item A \Cs-correspondence $\cE$ is tracially properly outer
if and only if the stabilization $\cE \otimes \IK$ is tracially properly outer.
\end{itemize}
\end{Lem}
\begin{proof}
The second statement is obvious, hence we only show the first claim.
Observe that when $\cE={}_\rho A$, one has a right $\pi_\tau(A)''$-linear isomorphism
${\cE}^\tau \cong L^2(\pi_\tau(A)'')$.
Hence we have a natural identification between $\pi_\tau(A)''$
and the set of all right $\pi_\tau(A)''$-linear bounded maps from $\cE^\tau$ to $L^2(\pi_\tau(A)'')$.
With this identification, $x\in \IB(\cE^\tau, L^2(\pi_\tau(A)''))$ is $A$-bilinear
if and only if it is contained in $\pi_\tau(A)''$ and satisfies $x \pi_\tau(\rho(a))=\pi_\tau(a)x$ for all $a\in A$.
This proves the claim.
\end{proof}

\begin{Thm}
Let $\cE$ be a \Cs-correspondence over $A$.
Assume that the positive tensor powers of $\cE$
are tracially properly outer.
Then there is an affine homeomorphism from ${\rm T}(\Cn{\cE})$ onto ${\rm T}_{\cE}(A)$ given by the restriction.
\end{Thm}
\begin{Thm}
Let $(A, \alpha)$ be an $\IR$-\Cs-algebra.
Let $\cE$ is an $\IR$-\Cs-correspondence over $A$.
Let ${\alpha_\upsilon} \colon \IR \acts \Cn{\cE}$ denote the generalized quasi-free flow.
Assume that the positive tensor powers of $\cE\rca{{\alpha_\upsilon}} \IR$
are tracially properly outer.
Then for any $\beta\in \IR$,
the map
$\bullet|_A\colon {\rm KMS}^\beta(\Cn{\cE}, {\alpha_\upsilon})\rightarrow {\rm KMS}^\beta_\cE(A, \alpha)$
is an affine isomorphism.
\end{Thm}
\section{Ideals of Cuntz--Pimsner algebras}
In this section, we study the ideal structure of Cuntz--Pimsner algebras beyond the simple case.
As in the previous section, we first study the case $\cE={}_\rho A$ for some $\rho \colon A \rightarrow \cM(A)$.
Again the general case will be deduced to this case by the stabilization theorem.

We employ the following notion of outerness for a $\ast$-homomorphism.
An extension of this notion to a general \Cs-correspondence will be given later.

For a \Cs-algebra $A$, we set
$\ell^2(A)$ to be the standard \Cs-correspondence $\ell^2 \otimes A$ over $A$.
\begin{Def}
We say that a $\ast$-homomorphism $\rho \colon A \rightarrow \cM(A)$ is centrally free,
if for any $b\in B_{[0, 1]}$, $a\in A$, $n, m\in \IN$ with $n\neq m$, $\ve>0$, there exists $\xi \in (\ell^2(B_{[0, 1]}))_1$ with
\[\ip{\xi, b \xi} \approx_\ve b,\quad \ip{\rho^n(\xi), a \rho^{m}(\xi)}\approx_\ve 0.\]
\end{Def}
\begin{Lem}\label{Lem:central free 1}
Assume that $\rho \colon A \rightarrow \cM(A)$ is centrally free.
Then for any $b\in B_{[0, 1]}$, $a\in A$, $n, m \geq 1$, and $\ve >0$,
there exists $\xi\in (\ell^2(B_{[0, 1]}))_1$ with
\[\sum_{k\in \IN} \xi(k) (b + \iota_n(a)u^{-m}) \xi(k)^\ast \approx_\ve b.\] 
\end{Lem}
\begin{proof}
We identify $B_{[0, 1]}$ with $A+ \rho(A)$ via $\iota_1$.
Then for any $c\in B_{[0, 1]}$, one has
\[c\iota_n(a) u^{-m} c^\ast= \iota_n(\rho^{n-1}(c) a \rho^{n+m-1}(c^\ast)) u^{-m}.\]
Thus the condition in the statement follows from the central freeness of $\rho$.
\end{proof}
\begin{Lem}\label{Lem:central free 2}
Assume that $\rho \colon A \rightarrow \cM(A)$ is centrally free.
Then for any $b\in B_{[0, 1]}$, $x\in \fB$, $m \geq 1$, and $\ve >0$,
there exists $\xi\in (\ell^2(B_{[0, 1]}))_1$ with
\[\sum_{n\in \IN} \xi(n)^\ast (b + xu^{-m}) \xi(n) \approx_\ve b.\] 
\end{Lem}
\begin{proof}
The proof below is similar to the argument in the proof of Lemma 3.1 in \cite{SuzMAAN2}.
For completeness, we include a proof.

By approximation, it suffices to show the claim when $x\in B_{[-N, N]}$ for some $N\in \IN$.
We may choose $N\geq m$.
Let $\delta>0$, which will be determined later.
Choose $e\in A$ with $0\leq e \leq 1$, $e b e \approx_{\delta} b$.
Then one has $e x \sigma^m(e) \in B_{[m, N]}$.
Hence one can find $a_m, \ldots, a_N\in A$
with $exe =\sum_{j=m}^N \iota_j (a_j)$.
By applying Lemma \ref{Lem:central free 1} to $ebe$ and $\iota_m(a_m)$,
one can find $\xi_0\in (\ell^2(B_{[0, 1]}))_1$
with 
\[b_1:= \sum_{n\in \IN} \xi_0(n)^\ast e b e \xi_0(n) \approx_\delta ebe, \quad \sum_{n\in \IN} \xi_0(n)^\ast \iota_m(a_m)u^{-m} \xi_0(n) \approx_\delta 0.\]
Next we apply Lemma \ref{Lem:central free 1} to $b_1$ and $\iota_{m+1}(a'_{m+1})$ where $a'_{m+1}:= \sum_{n\in \IN} \rho(\xi_0(n)^\ast) a_{m+1}\rho^{m+1}(\xi_0(n))$
to get $\xi_1\in (\ell^2(B_{[0, 1]}))_1$ with
\[b_2:= \sum_{n\in \IN} \xi_1(n)^\ast b_1 \xi_1(n) \approx_\delta b_1, \quad \sum_{n_0, n_1\in \IN} \xi_1(n_1)^\ast \xi_0(n_0)^\ast \iota_{m+1}(a_{m+1})u^{-m} \xi_0(n_0) \xi_1(n_1) \approx_\delta 0.\]
By iterating this argument $N-m+1$ times, we obtain a sequence $(\xi_j)_{j=0}^{N-m}$ in $(\ell^2(B_{[0, 1]}))_1$ with the following conditions:
\[b_{j+1}:= \sum_{n\in \IN} \xi_{j}(n)^\ast b_{j} \xi_j(n) \approx_\delta b_j,\]
\[\sum_{n_0, \ldots, n_j\in \IN} \xi_j(n_j)^\ast \xi_{j-1}(n_{j-1})^\ast \cdots \xi_0(n_0)^\ast\iota_{m+j}(a_{m+j})u^{-m} \xi_0(n_0) \cdots \xi_j(n_j) \approx_\delta 0.\]
Define $\xi \in \ell^2(\IN^{N-m+1}, B_{[0, 1]})$ to be
\[\xi(k_0, \ldots, k_{N-m}):=e \xi_0(k_0)\xi_1(k_1)\cdots \xi_{N-m}(k_{N-m}).\]
Then by the choice of $(\xi_j)_{j=0}^{N-m}$, one has
\[\sum_{\fn\in \IN^{N-m+1}} \xi(\fn)^\ast b \xi(\fn) \approx_{(N-m+2)\delta} b, \quad \sum_{\fn\in \IN^{N-m+1}} \xi(\fn)^\ast x u^{-m} \xi(\fn) \approx_{(N-m+2)\delta} 0.\]
Now put $\delta = \frac{\ve}{N-m+2}$. Then for a bijection $f\colon \IN \rightarrow \IN^{N-m+1}$, the composite $\xi \circ f \in \ell^2(B_{[0, 1]})$ gives
the desired element.
\end{proof}
\begin{Thm}\label{Thm:idealclass1}
Assume that $\rho \colon A \rightarrow \cM(A)$ is centrally free.
Put $D:=A+\rho(A)$.
Then the map $\bullet \cap D$ gives a lattice isomorphism
from the ideal lattice of $\Cn{\rho}$ onto the lattice of all ideals of $D$ satisfying
condition $(\ref{equation:inv2})$ in Lemma \ref{Lem:idclass}.
\end{Thm}
\begin{proof}
By Theorem \ref{Thm:decomp}, it suffices to show the statement for $\fB\rca{\sigma} \IZ$ instead of $\Cn{\rho}$.

Let $I\lhd \fB\rca{\sigma} \IZ$.
Since $\IZ$ is amenable (or exact), to show the claim, it suffices to show that $E(I) \subset I$ by Lemma \ref{Lem:idclass}.
Since $\cl(E(I))$ is a $\sigma$-invariant ideal of $\fB$, by Lemma \ref{Lem:idclass}, it is enough to show that
$\cl(E(I)) \cap B_{[0, 1]} \subset I$.

Fix $d \in \overline{E(I)} \cap B_{[0, 1]}$. Let $\ve>0$.
Then one has $x\in I$ with $E(x) \approx_\ve d$.
Choose $y\in \fB \rtimes_{{\rm alg}, \sigma} \IZ$
with $y \approx_\ve x$, $E(y)=d$.
By repeated applications of Lemma \ref{Lem:central free 2}, for some $N\in \IN$, one has $\xi \in (\ell^2(\IN^N, B_{[0, 1]}))_1$
with
\[\sum_{\fn\in \IN^N} \xi(\fn)^\ast y \xi(\fn) \approx_{\ve} d.\]
This proves $d\in_{3\ve} I$. Since $\ve>0$ is arbitrarily small, we conclude $d\in I$.
\end{proof}

Based on Theorem \ref{Thm:idealclass1}, under a suitable assumption, we generalize the classification result of ideals for general Cuntz--Pimsner algebras.
\begin{Def}
We say that a \Cs-correspondence $\cE$ is centrally free, if for any $d\in D:= A\otimes 1_\cE +\IK(\cE)$, $n, m\in \IN$ with $n\neq m$, $x\in\IK(\cE^{\otimes m}, \cE^{\otimes n})$, and any $\ve>0$,
there exists $\xi\in (\ell^2(D))_1$ with
\[\sum_{n \in \IN} \xi(n)^\ast d \xi(n)\approx_\ve d,\quad \sum_{k\in \IN} (\xi(k)^\ast \otimes 1_{\cE^{\otimes n}})(x\otimes 1_\cE) (\xi(k) \otimes 1_{\cE^{\otimes m}}) \approx_\ve 0.\]
By a similar argument to the proof of Lemma \ref{Lem:central free 2},
we only need to check the second condition for rank one operators $x= e_{\eta, \zeta}$; $\eta\in \cS_n, \zeta\in \cS_m$
for a sequence of total subsets $\cS_k \subset \cE^{\otimes k}$; $k\in \IN$.
\end{Def}
\begin{Exm}
Let $B$ be a \Cs-algebra.
Let $\sigma \colon \IZ \acts B$ be a centrally free action in the sense of \cite{SuzCMP}.
Then for any \Cs-correspondence $\cE$,
the exterior tensor product $\cE \otimes {}_\sigma B$
is centrally free.
Again an important choice is $B=\cZ^{\otimes \IZ}$ and $\sigma$ is the Bernoulli shift,
by the same reason as Example \ref{Exm:trace}.
For the fact that $(\cZ^{\otimes \IZ}, \sigma)$ is centrally free, see Example 4.10 in \cite{SuzCMP}.
\end{Exm}

The next obvious lemma shows that the central freeness of $\cE$ implies the central freeness of the associated $\ast$-homomorphism $\rho$ given by Proposition \ref{Prop:st}.

\begin{Lem}
The following statements hold true.
\begin{itemize}
\item A $\ast$-homomorphism $\rho \colon A \rightarrow \cM(A)$ is centrally free
if and only if the \Cs-correspondence ${}_\rho A$ is centrally free.
\item If a \Cs-correspondence $\cE$ is centrally free,
then so is the stabilization $\cE \otimes \IK$.
\end{itemize}
\end{Lem}
\begin{Def}
Let ${\rm T} \colon \cE \rightarrow \IB(\cE, \cE^{\otimes 2})$ denote the criation operator:
\[{\rm T}_\xi \zeta:=\xi \otimes \zeta \quad{\rm~for~}\xi, \zeta \in \cE.\]
We say that $I\lhd D:= A\otimes 1_\cE +\IK(\cE) \subset \IB(\cE)$ is $\cE$-invariant if it satisfies
\[{\rm T}_\cE ^\ast (I\otimes 1_{\cE}) {\rm T}_\cE \subset I,\quad (D \otimes 1_{\cE}) \cap (T_\cE \cdot I\cdot T_\cE^\ast) \subset I \otimes 1_\cE.\]
\end{Def}
\begin{Lem}
When $\cE={}_\rho A$ for some $\rho \colon A \rightarrow \cM(A)$,
$I\lhd D$ is $\cE$-invariant if and only if 
it satisfies condition $(\ref{equation:inv2})$ in Lemma $\ref{Lem:idclass}$.
\end{Lem}
\begin{Thm}
Let $\cE$ be a centrally free \Cs-correspondence over $A$.
Then there is a bijective correspondence between ideals of $\Cn{\cE}$ and $\cE$-invariant ideals of $D:=A\otimes 1_\cE + \IK(\cE)$.
The correspondence is given by sending $I \lhd \Cn{\cE}$ to $I\cap D \lhd D$.
\end{Thm}

\section{Quasi-free actions on $\Cn{n}$}\label{section:quasi-free}
In this section,
we consider quasi-free actions of locally compact groups on the Cuntz algebras.
Although it is also possible to apply the strategy in this section to graph \Cs-algebras, to keep this article in a suitable length, we concentrate on the Cuntz algebras. 
The Hao--Ng isomorphism \cite{HN}, \cite{KR}, \cite{DT} (see also Theorem \ref{Thm:HN}) yields that
the reduced crossed product of a quasi-free action on a Cuntz algebra is isomorphic to the Cuntz--Pimsner algebra of a \Cs-correspondence over the reduced group \Cs-algebra of the acting group.
Using this isomorphism together with our results in Sections \ref{section:simple} and \ref{section:simplepjl},
we characterize the simplicity of their reduced crossed products.
We note that related results have been studied for abelian groups in \cite{MR2016248,Kat,MR623751}. 
For a quasi-free action of a compact group,
the crossed product algebra and the fixed point algebra are identified with a corner of a graph \Cs-algebra in \cite{MR1670363,MR1432596,Izuqp}.
Our approach also conceptually explains these isomorphisms in terms of a \Cs-correspondence.

\begin{no}
Let $G$ be a locally compact group.
Throughout this section, all unitary representations are assumed to be nonzero and strongly continuous.
Let $\sigma\colon G\to \cU(\cH)$ and $\sigma_{i}\colon G\to \cU(\cH_{i})$ ($i=1,2$) be unitary representations of $G$.
\begin{itemize}
 \item We write $1_{\hat{G}}$ for the trivial (one-dimensional) representation of $G$.
 \item Let $\cH^G$ denote the $G$-fixed point space of $\sigma$.
 \item Let $\alpha_\sigma\colon G \acts \Cn{\cH}$ denote the quasi-free action of $\sigma$.
 \item Let $\rg(G)$ denote the reduced group \Cs-algebra of $G$.
 \item Denote by $(\sigma_1, \sigma_2)$ the space of intertwiners from $\sigma_{2}$ to $\sigma_{1}$, that is,
\[(\sigma_{1},\sigma_{2}):=\{T\in \IB(\cH_{2},\cH_{1}): T\sigma_{2}(g)=\sigma_{1}(g)T\;\text{ for all }\;g\in G\}.\]
Note that when $\sigma_1$ is irreducible, $(\sigma_{1},\sigma_{2})$ forms a Hilbert space
with respect to the inner product
\[\ip{T_1, T_2}:= T_1 T_2^\ast \in \IC\cdot 1_{\cH_1} \cong \IC\]
 \item When $\sigma$ is finite-dimensional,
 we denote its conjugate representation by $\overline{\sigma}$.
 \item Denote by $F(\sigma)$ the (full) Fock representation of $\sigma$:
 \[F(\sigma):=\bigoplus_{n=0}^{\infty}\sigma^{\otimes n}\colon G\to \cU\bigg(\bigoplus_{n=0}^{\infty}\cH^{\otimes n}\bigg).\]
 \item 
 When $\sigma_1$ is weakly contained in $\sigma_2$,
 we write $\sigma_{1}\prec\sigma_{2}$.
 If $\sigma_{1}$ and $\sigma_{2}$ are weakly equivalent,
 then we write $\sigma_{1}\sim\sigma_{2}$.
 When $\sigma_1$ is conjugate to a subrepresentation of $\sigma_2$, we write $\sigma_1 \leq \sigma_2$.
 \item Denote by $\hat{G}$ the unitary dual of $G$, that is, the set of unitary equivalence classes of irreducible unitary representations of $G$.
 \item Denote by $\hat{G}_{w}$ the set of weak equivalence classes of irreducible unitary representations of $G$.
 We denote by $\hat{G}_{w,\prec\lambda}$ the subset of $\hat{G}_{w}$ consisting of those weakly contained in the left regular representation $\lambda$.
 \item Whenever necessary, we identify $\hat G$ (resp.\ $\hat G_w$, $\hat G_{w,\prec\lambda}$) with a set of unitary representations by fixing a choice of representatives.
 \item When $\sigma\prec\lambda$, we use the same symbol $\sigma \colon \rg(G)\to \IB(\cH)$
 to denote its induced $*$-representation.
 \item Unless otherwise specified, for a unitary representation $\pi$, we denote by $\cH_\pi$ its underlying Hilbert space.
\end{itemize}
\end{no}

\begin{Def}
Let $\pi\colon G\to \cU(\cH_{\pi})$ be a unitary representation.
 We define the \Cs-correspondence $\cH_{\pi}\rtimes_{\rm r} G$ over $\rg(G)$ as follows:
 \begin{itemize}
 \item As a right Hilbert \Cs-module, it is the exterior tensor product $\cH_{\pi}\otimes\rg(G)$.
 \item The left $\rg(G)$-action on $\cH_{\pi}\rtimes_{\rm r} G$ is defined via the unitary representation
 \[\pi\otimes\lambda\colon G\to \cU(\cH_{\pi}\rtimes_{\rm r} G).\]
 \end{itemize}
 Clearly it satisfies Assumption \ref{ass:setting}.
\end{Def}
\begin{Rem}\label{Rem:quasi free crossed}
When we regard $\cH_\pi$ as a $G$-\Cs-correspondence over $\IC$ via $\pi$,
the \Cs-correspondence $\cH_{\pi}\rtimes_{\rm r} G$ is nothing but its reduced Hao--Ng \Cs-correspondence.
We thus have a natural isomorphism
\[\Cn{\cH_{\pi}}\rca{\alpha_\pi} G \cong \Cn{\cH_{\pi}\rtimes_{\rm r} G}.\]
This observation is fundamental throughout this section.
\end{Rem}

We first characterize the minimality of $\cH_{\pi}\rtimes_{\rm r} G$, in the sense of Definition \ref{Def:IdC*corr}.

\begin{Lem}\label{lem_minimal}
 Let $G$ be a locally compact group. Let $\pi\colon G\to \cU(\cH_{\pi})$ be a unitary representation.
 \begin{itemize}
 \item
 If $\pi$ is finite-dimensional,
 then the following two conditions are equivalent.
 \begin{enumerate}\upshape
 \item $\cH_{\pi}\rtimes_{\rm r} G$ is minimal.
 \item There is no unitary representation $\mu$ of $G$ satisfying $\mu\prec\lambda$, $\mu\nsim\lambda$, and $\pi\otimes\mu\sim\mu$.
 \end{enumerate}
 \item
 If $\pi$ is infinite-dimensional, then the following two conditions are equivalent.
 \begin{enumerate}[label=(\arabic*$^\prime$)]\upshape
 \item $\cH_{\pi}\rtimes_{\rm r} G$ is minimal.
 \item There is no unitary representation $\mu$ of $G$ satisfying $\mu\prec\lambda$, $\mu\nsim\lambda$, and $\pi\otimes\mu\prec\mu$.
 \end{enumerate}
 \end{itemize}
\end{Lem}
\begin{proof}
 We in fact describe all $(\cH_\pi \rtimes_{\rm r} G$)-invariant ideals of $\rg(G)$, in the sense of Definition \ref{Def:IdC*corr}.
 The statement is its immediate consequence.
 
 Let $I\lhd\rg(G)$ be a nontrivial ideal.
 Choose a unitary representation $\mu \prec \lambda$ of $G$ with $I=\ker\mu$.
 Observe that
 $I=\ker \mu$ is ($\cH_{\pi}\rtimes_{\rm r} G$)-invariant if and only if the following two conditions hold true:
 \begin{align}
 \langle \cH_{\pi}\rtimes_{\rm r} G, \ker\mu \cdot (\cH_{\pi}\rtimes_{\rm r} G)\rangle&\subset\ker\mu,\label{eq invariant1}\\
 (\pi\otimes\lambda)^{-1}(\IK(\cH_{\pi})\otimes\ker\mu)&\subset\ker\mu.\label{eq invariant2}
 \end{align}
 We will reformulate these conditions in terms of unitary representations of $G$.
 For $\xi,\eta\in\cH_{\pi}$, let $\omega_{\xi,\eta}\colon \IB(\cH_{\pi})\to\mathbb{C}$ denote the vector functional
 given by $\omega_{\xi, \eta}(x):=\ip{x\xi, \eta}$ for $x\in \IB(\cH_{\pi})$.
 Condition (\ref{eq invariant1}) is equivalent to the next condition: for every $x\in\ker\mu$ and $\xi,\eta\in\cH_{\pi}$, one has
 \begin{equation}\label{eq kernel}
 (\omega_{\xi,\eta}\otimes\id_{\rg(G)})((\pi\otimes\lambda)(x))\in\ker\mu.
 \end{equation}
 Since
 \[ \mu \circ (\omega_{\xi,\eta}\otimes\id_{\rg(G)}) \circ (\pi\otimes\lambda) = (\omega_{\xi,\eta}\otimes\mu)\circ(\pi\otimes\lambda)=(\omega_{\xi,\eta}\bar{\otimes} \id_{\IB(\cH_{\mu})})\circ(\pi\otimes\mu)\]
 on $\rg(G)$ for $\xi, \eta \in \cH_\pi$,
 condition (\ref{eq kernel}) is equivalent to the equality
 \begin{align*}
0=(\omega_{\xi,\eta}\bar{\otimes} \id_{\IB(\cH_{\mu})})((\pi\otimes\mu)(x)).
 \end{align*}
 Note that the maps $\omega_{\xi,\eta}\bar{\otimes} \id_{\IB(\cH_{\mu})}$; $\xi, \eta \in \cH_\pi$ separate points of $\IB(\cH_\pi \otimes \cH_\mu)$.
 Thus we conclude that condition (\ref{eq invariant1}) is equivalent to the condition $\ker\mu\subset\ker (\pi\otimes\mu)$; that is, the relation $\pi \otimes \mu \prec \mu$.
 
 We next study condition (\ref{eq invariant2}) when $\pi$ is finite-dimensional.
 In this case the image of $\pi\otimes\lambda\colon \rg(G)\to \IB(\cH_{\pi}\rtimes_{\rm r} G)$ sits in $\IK(\cH_{\pi}\rtimes_{\rm r} G)=\IB(\cH_\pi) \otimes \rg(G)$.
 (In other words, $\cH_\pi \rtimes_{\rm r} G$ is proper.)
 Thus, one has the decomposition
 \[\pi\otimes\mu=(\id_{\IB(\cH_{\pi})}\otimes\mu)\circ(\pi\otimes\lambda)\colon\rg(G)\to \IB(\cH_{\pi})\otimes \IB(\cH_{\mu}).\]
 Since $\pi\otimes\lambda \sim \lambda$, we obtain
 \[\ker (\pi\otimes\mu)=(\pi\otimes\lambda)^{-1}(\IK(\cH_{\pi})\otimes\ker\mu).\]
 Thus, the inclusion (\ref{eq invariant2}) is equivalent to $\ker (\pi\otimes\mu)\subset\ker\mu$, that is, the relation $\mu \prec \pi \otimes \mu$.
 In summary, we have shown the bijective correspondence between ($\cH_\pi \rtimes_{\rm r} G$)-invariant ideals of $\rg(G)$
 and weak equivalence classes of unitary representations $\mu$ of $G$ satisfying $\mu \prec \lambda$ and $\mu \sim \pi \otimes \mu$.
 This in particular proves the equivalence of conditions (1) and (2) when $\pi$ is finite-dimensional. 

 We next consider the infinite-dimensional case.
We claim that if $\pi$ is infinite-dimensional, then one has the equality
 \begin{equation}\label{eq not proper}
\left(\pi\otimes\lambda\right)(\rg(G))\cap\left(\IK(\cH_{\pi})\otimes\rg(G)\right)=0.
 \end{equation} 
 (In particular, $\cH_{\pi}\rtimes_{\rm r} G$ is non-proper.)
 Equation (\ref{eq not proper}) implies that
 condition (\ref{eq invariant2}) is redundant.
 Let $W\in \IB(\cH_{\pi}\otimes L^{2}(G))$ be the unitary operator defined by $(W\xi)(g):=\pi(g^{-1})(\xi(g))$ for $\xi\in \cH_\pi \otimes L^2(G)$ and $g\in G$. Here we identify $\cH_\pi \otimes L^2(G)$ with $L^2(G, \cH_\pi)$ in the obvious way.
 Then one has $\ad(W)\circ (\pi\otimes\lambda)=1_{\cH_{\pi}}\otimes \lambda$ (the Fell absorption principle).
 Hence, to prove (\ref{eq not proper}), it suffices to show the equality
 \[W(\IK(\cH_{\pi})\otimes\rg(G))W^{*}\cap\left(1_{\cH_{\pi}}\otimes\rg(G)\right)=0.\]
 For each $\xi,\eta\in\cH_{\pi}$ and $g\in G$, observe that the operator
 \[W(e_{\xi,\eta}\otimes\lambda_{g})W^{*}(1\otimes\lambda_{g}^{*})\in \IB(\cH_{\pi}\otimes L^{2}(G))\] is equal to the multiplication operator of the function $f_g \in C_b(G,\IK(\cH_\pi))$ defined by
 \[f_g \colon G\ni h\mapsto e_{\pi(h^{-1})\xi,\,\pi(h^{-1}g)\eta}\in\IK(\cH_{\pi}).\]
 Clearly, the map $G\ni g \mapsto f_g\in C_b(G,\IK(\cH_\pi))$ is continuous in norm.
 For any $a\in C_c(G) \subset \rg(G)$, one has
 \[W(e_{\xi, \eta} \otimes a)W^\ast = \int_G a(g) f_g \lambda_g dg.\]
 By approximating the function $g \mapsto a(g) f_g$ by simple Borel functions $s= \sum_{i=1}^n x_i \chi_{E_i}$; $x_i \in C_b(G,\IK(\cH_\pi))$, in the $L^1$-norm, we obtain approximations
 \[W(e_{\xi, \eta} \otimes a)W^\ast \approx \sum_{i=1}^n x_i \int_{E_i} \lambda_g dg \in \spa \Big(C_{b}(G,\IK(\cH_{\pi}))\cdot(1_{\cH_\pi} \otimes \rg(G))\Big).\] 
 This proves
 \[
 W(\IK(\cH_{\pi})\otimes\rg(G))W^{*}\subset \csp\Big(C_{b}(G,\IK(\cH_{\pi}))\cdot(1_{\cH_\pi} \otimes \rg(G))\Big).\]
Observe that
 \[\csp\Big(C_{b}(G,\IK(\cH_{\pi}))\cdot(1_{\cH_\pi} \otimes \rg(G))\Big)\cap (1_{\cH_\pi} \otimes \rg(G))=0.\]
 Indeed,
 if $x$ is an element of the intersection,
 then for any vector functional $\omega\in \IB(L^{2}(G))_{*}$,
 one has $(\id_{\IB(\cH_{\pi})} \bar{\otimes} \omega)(x)\in \IK(\cH_{\pi})\cap\IC1_{\cH_{\pi}}=0$.
 This implies $x=0$,
 and hence (\ref{eq not proper}) holds.

 In summary, we have shown the bijective correspondence between ($\cH_\pi \rtimes_{\rm r} G$)-invariant ideals of $\rg(G)$
 and weak equivalence classes of unitary representations $\mu$ of $G$ satisfying $\mu \prec \lambda$ and $\pi \otimes\mu \prec \mu$.
 This in particular proves the equivalence of conditions (1') and (2') when $\pi$ is infinite-dimensional. 
\end{proof}
\begin{Cor}\label{cor infinite quasi free}
 Let $\pi$ be a faithful infinite-dimensional unitary representation of a locally compact group $G$. Assume that $F(\pi)$ weakly contains some $\sigma \prec \lambda$.
 Then the following conditions are equivalent.
\begin{enumerate}\upshape
 \item $\Cn{\cH_\pi} \rca{\alpha_\pi} G$ is simple.
 \item There is no unitary representation $\mu$ of $G$ satisfying $\mu\prec\lambda$, $\mu\nsim\lambda$, and $\pi\otimes\mu\prec\mu$.
 \item $\lambda\prec F(\pi)$.
\end{enumerate}
\end{Cor}
\begin{proof}
 As observed in Remark \ref{Rem:quasi free crossed}, the Hao--Ng isomorphism theorem \cite{HN}, \cite{DT} (see also Theorem \ref{Thm:HN}) implies $\Cn{\cH_{\pi}}\rca{\alpha_\pi} G \cong \Cn{\cH_{\pi}\rtimes_{\rm r} G}$.
 As shown in the proof of Lemma \ref{lem_minimal},
 the \Cs-correspondence $\cH_{\pi}\rtimes_{\rm r} G$ is non-proper.
 Thus, by \cite[Theorem 2.3]{Kit},
 $\Cn{\cH_{\pi}\rtimes_{\rm r} G}$ is simple if and only if $\cH_{\pi}\rtimes_{\rm r} G$ is minimal.
 Thus
 the equivalence between $(1)$ and $(2)$ follows from Lemma \ref{lem_minimal}.

 We next show the implication (2) $\Rightarrow$ (3).
 Assume that $\pi$ satisfies condition (2).
 Set
 \[S:=\{\sigma\in\hat{G}: \sigma\prec\lambda,\; \sigma\prec F(\pi)\}.\]
 Note that $S \neq \emptyset$ by the assumption on $\pi$.
Define $\mu:=\bigoplus_{\sigma\in S}\sigma$.
Then clearly we have 
\[\mu\prec\lambda,\quad \mu\prec F(\pi),\quad \mu\otimes\pi\prec\mu.\]
By these relations and condition (2),
one has $\lambda\sim\mu\prec F(\pi)$.
(We note that this implication is the only point where we use the assumption on $\pi$.)

We next show the implication (3) $\Rightarrow$ (2), which completes the proof.
Assume that $\lambda\prec F(\pi)$.
Then, for every unitary representation $\mu\prec\lambda$ with $\mu\otimes\pi\prec\mu$, one has
\[\lambda\sim \lambda\otimes\mu\prec F(\pi)\otimes\mu\prec \mu\prec\lambda.\]
This implies $\mu \sim \lambda$, hence $\pi$ satisfies condition (2).
\end{proof}
\begin{Rem}
 When $G$ is abelian,
 Corollary \ref{cor infinite quasi free} is equivalent to the characterization given in \cite[Theorem 5.1]{MR623751}.
 Indeed,
 if $\pi$ is a unitary representation of a locally compact abelian group $G$,
 then the spectrum ${\rm Sp}\,F(\pi)\subset\hat{G}$ of the Fock representation is equal to the closed subsemigroup of $\hat{G}$ generated by ${\rm Sp}\,\pi$.
 Hence condition (3) in Corollary \ref{cor infinite quasi free} is equivalent to the condition in \cite[Theorem 5.1]{MR623751}.
\end{Rem}
\begin{Exm}\label{Exm:SL2}
As an example,
we study a quasi-free action of $\SL_{2}(\IR)$.
Let $\Pi\colon \SL_{2}(\IR)\to \cU(\cH_{\Pi})$ be a faithful unitary representation.
Recall that $\SL_{2}(\IR)$ admits no non-trivial finite-dimensional unitary representations.
The tensor product structure and the Fell topology of the unitary dual $\widehat{\SL_{2}(\IR)}$ are well understood \cite{MR308795,MR509073} (see also \cite[Section 7.6]{MR1397028}).
Using these well-known structures,
we characterize the simplicity of $\Cn{\cH_{\Pi}} \rca{\alpha_\Pi} \SL_{2}(\IR)$.

Following the notation in \cite{MR509073},
let $\{\pi_{is,\ve}^{u}\}_{s\in\IR^{+},\,\ve=\pm 1}$ be the principal series of unitary representations,
$\{\pi_{s,1}^{u}\}_{s\in(-1,0)}$ be
the complementary series of unitary representations,
$\{T_{n}\}_{n\in\IZ\setminus\{\pm 1,\,0\}}$ be the discrete series representations,
$T_{\pm 1}$ be the mock discrete series representations.
We show that $\Cn{\cH_{\Pi}}\rca{\alpha_\Pi} \SL_{2}(\IR)$ is simple if and only if one of the following conditions holds true:
\begin{enumerate}
 \item $\pi_{is,-1}^{u}\prec\Pi$ for some $s\in\IR^{+}$.
 \item $\pi_{t,1}^{u}, T_{n}\prec\Pi$ for some $t\in i\IR^{+}\cup (-1,0)$ and an odd integer $n\in\IZ\setminus\{0\}$. 
 \item $T_{-m}, T_{n}\prec\Pi$ for some integers $n,m \geq 1$ such that $n$ or $m$ is odd.
\end{enumerate}
This characterization follows from well-known results of the representation theory of $\SL_2(\IR)$.
For completeness, we provide a brief explanation.

It is well-known that $\SL_{2}(\IR)$ is
of Type I, and the unitary dual $\widehat{\SL_{2}(\IR)}$ (as a set) is equal to the disjoint union of the principal series, the discrete series, the mock discrete series, the complementary series representation, and the trivial representation.
Moreover, one has
\begin{align}\label{eq SL2}
 \{\sigma\in\widehat{\SL_{2}(\IR)}:\sigma\prec\lambda\}
 =\{\pi^{u}_{is,\ve}: s\in\IR^{+},\,\ve=\pm1\}\cup\{T_{n}: n\in\IZ\setminus\{0\}\}.
 \end{align}
The mock discrete series representations $T_{\pm 1}$ are contained in the closure of the set of the principal series representations \cite[Section 7.6]{MR1397028}.
Hence $\Pi$ satisfies condition (3) in Corollary \ref{cor infinite quasi free} if and only if one has $\pi^{u}_{is,\ve},T_{n}\prec F(\Pi)$ for every $s\in\IR^{+}$, $\ve=\pm1$ and $n\in\IZ\setminus\{0,\pm1\}$.
We also note that for any non-trivial irreducible unitary representation of $\SL_{2}(\IR)$,
its tensor square weakly contains a representation $\sigma \prec \lambda$ \cite{MR509073}.
Hence any faithful $\Pi$ satisfies the assumption of Corollary \ref{cor infinite quasi free}.

Using \cite[Theorems 4.6, 5.9, 6.4, 7.1, 7.3, and 8.1]{MR509073} together with (\ref{eq SL2}),
one can verify that $\lambda\prec F(\Pi)$ if either the above (1), (2), or (3) holds.
Conversely,
if $\Pi$ fails all three of these conditions,
then the set $S:=\{\sigma\in\widehat{\SL_{2}(\IR)}:\sigma\prec\Pi\}$ is contained in one of the following subsets:
\begin{itemize}
 \item $\{\pi^{u}_{is,1}\}_{s\in\IR^{+}}\cup\{\pi^{u}_{t,1}\}_{t\in(-1,0)}\cup\{T_{n}\}_{n\in 2\IZ\setminus\{0\}}$.
 \item $\{T_{n}\}_{n\geq1}$.
 \item $\{T_{n}\}_{n\leq-1}$.
\end{itemize}
In all cases,
$S$ is disjoint from $\{\pi^{u}_{is,-1}\}_{s\in\IR^{+}}$. 
By \cite[Theorems 4.6, 5.9, 6.4, 7.1, 7.3, and 8.1]{MR509073},
it follows that $\pi^{u}_{is,-1}\nprec\Pi^{\otimes n}$ for every $n\in\IN$ and $s\in\IR^{+}$.
Furthermore, the set $\{\pi^{u}_{is,-1}\}_{s\in \IR^{+}}$ is open in $\widehat{\SL_{2}(\IR)}$ with respect to the Fell topology \cite[Section 7]{MR1397028}.
Therefore, we conclude $\lambda\nprec F(\Pi)$. 
Consequently $\Cn{\cH_\Pi}\rca{\alpha_\Pi}\SL_{2}(\IR)$ is not simple by Corollary \ref{cor infinite quasi free}.
\end{Exm}

In the second half of this section, we concentrate on the finite-dimensional case.
In this case, the Hao--Ng \Cs-correspondences are proper, hence we need our new results in Sections \ref{section:simple} and \ref{section:simplepjl}.

We first show the following lemma.
\begin{Lem}\label{lem inner}
 Let $A$ be a stable \Cs-algebra and let $\rho\colon A\to A$ be a faithful nondegenerate $*$-endomorphism with a conjugate $\bar{\rho}\colon A\to A$ in the sense of \cite[Definition 4.3]{MR2085108}.
 If $p\in A$ is a projection such that $pAp\cong p\cdot(_{\rho} A)\cdot p$,
 then $pAp\cong p\cdot(_{\bar{\rho}} A)\cdot p$.
\end{Lem}
\begin{proof}
 Let $u\in \rho(p)Ap$ be a partial isometry element satisfying $uxu^{*}=\rho(x)$ for every $x\in pAp$.
 Let $R\in(\bar{\rho}\circ\rho,\,\id)$ and $\bar{R}\in(\rho\circ\bar{\rho},\,\id)$ be intertwiners satisfying $\rho(R^{*})\bar{R}=\bar{\rho}(\bar{R}^{*})R=1$.
 It is known that $d:=R^{*}R$ and $\bar{d}:=\bar{R}^{*}\bar{R}$ are invertible elements in ${\rm Z}(\cM(A))$.
 We claim that the following map $\bar{u}$ gives the desired isomorphism:
 \[\bar{u}\colon pAp\ni a\mapsto (\sqrt{d})^{-1}\bar{\rho}(u^{*})Ra\in p\cdot({}_{\bar{\rho}}A)\cdot p.\]
 It is straightforward to verify that $\bar{u}$ is adjointable and $(pAp)$-bilinear.
 Since $d^{-1}R^{*}\bar{\rho}(uu^{*})R=p$,
 $\bar{u}$ is an isometry operator.
 To show the surjectivity, we consider the conditional expectation
 \[E_{\bar{\rho}}\colon A\ni a\mapsto \bar{\rho}\left(\bar{d}^{-1}\bar{R}^{*}\rho(a)\bar{R}\right)\in\bar{\rho}(A).\]
 Observe that $E_{\bar{\rho}}$ has a quasi-basis $(R^{*}\bar{\rho}(\bar{d})^{\frac{1}{2}},\bar{\rho}(\bar{d})^{\frac{1}{2}}R)$.
 Indeed, for every $a\in A$, direct calculations show that
 \begin{align*}
 E_{\bar{\rho}}\left(aR^{*}\bar{\rho}(\bar{d})^{\frac{1}{2}}\right)\bar{\rho}(\bar{d})^{\frac{1}{2}}R
 =&E_{\bar{\rho}}\left(aR^{*}\right)\bar{\rho}(\bar{d})R\\
 =&\bar{\rho}\left(\bar{d}^{-1}\bar{R}^{*}\rho(aR^{*})\bar{R}\right)\bar{\rho}(\bar{d})R\\
 =&\bar{\rho}\left(\bar{R}^{*}\rho(aR^{*})\bar{R}\right)R
 =a.
 \end{align*}
 Then observe that
 \[\bar{u}(pAp)=\bar{\rho}(u^{*})RpAp=\bar{\rho}(pAu^{*})R=\bar{\rho}(pA\rho(p))R=\bar{\rho}(p)\left(\bar{\rho}(A)R\right)p=p \cdot ({}_{\bar{\rho}}A) \cdot p.\]
 Thus $\bar{u}$ is indeed a unitary operator.
\end{proof}
For finite-dimensional unitary representations $\sigma$ and $\pi$ of a locally compact group $G$,
one has a natural isomorphism
\[(\cH_{\sigma}\rtimes_{\rm r} G)\otimes_{\rg(G)}(\cH_{\pi}\rtimes_{\rm r} G)\cong(\cH_{\sigma}\otimes\cH_{\pi})\rtimes_{\rm r} G\] of \Cs-correspondences.
It follows that $\pi\mapsto\cH_{\pi}\rtimes_{\rm r} G$ is a tensor functor from the category ${\rm Rep_{f}}(G)$ of finite-dimensional unitary representations of $G$ to the category ${\rm Corr}(\rg(G))$ of \Cs-correspondences over $\rg(G)$.
Consequently,
for each finite-dimensional unitary representation $\pi$,
$\cH_{\bar{\pi}}\rtimes_{\rm r} G$ is a conjugate object of $\cH_{\pi}\rtimes_{\rm r} G$ in the sense of \cite[Definition 4.3]{MR2085108}.

\begin{Prop}\label{prop outer}
 Let $\pi\colon G\to \IB(\cH_{\pi})$ be a faithful finite-dimensional unitary representation of a locally compact group.
 Assume that $G$ satisfies
 \begin{equation}\label{equation:vanish}
 \inf\{\|\sigma(a)\|:\sigma\in \hat{G}_{w,\prec\lambda}\}=0 \quad {\rm~for~all~}a\in \rg(G).
 \end{equation}
 Then the following conditions are equivalent:
 \begin{enumerate}
 \item $\cH_{\pi}\rtimes_{\rm r} G$ is minimal.
 \item There is no unitary representation $\mu$ of $G$ satisfying $\mu\otimes\pi\sim\mu$,
 $\mu\prec\lambda$, and $\mu\nsim\lambda$.
 \item $\Cn{\cH_{\pi}}\rca{\alpha_\pi} G$ is simple.
 \end{enumerate}
\end{Prop}
\begin{proof}
 The equivalence between (1) and (2) follows from Lemma \ref{lem_minimal}.
 
 To prove (2)$\Rightarrow$(3),
 by Theorem \ref{Thm:simpleproless},
 it suffices to show that $\cH_{\pi}\rtimes_{\rm r} G$ satisfies condition (2) therein.
 Let $\rho_{\pi}$ (resp. $\rho_{\bar{\pi}}$) be an endomorphism on $\rg(G)\otimes\IK\otimes\Cn{\infty}$ associated with the stabilized \Cs-correspondence $(\cH_{\pi}\rtimes_{\rm r} G)\otimes\IK\otimes\Cn{\infty}$ (resp. $(\cH_{\bar{\pi}}\rtimes_{\rm r} G)\otimes\IK\otimes\Cn{\infty}$) by Proposition \ref{Prop:st}.
 For each unitary representation $\sigma\prec\lambda$,
 we put 
 \[\sigma^{\mathrm{stab}}:=\sigma\otimes\id_{\IK\otimes\Cn{\infty}}\colon \rg(G)\otimes\IK\otimes\Cn{\infty}\to \IB(\cH_{\sigma})\otimes \IK\otimes\Cn{\infty}.\]
 Note that by the assumption (\ref{equation:vanish}), one has
 \begin{equation}\label{equation:vanish2}
 \inf\{\|\sigma^{\mathrm{stab}}(x)\|: \sigma \in \hat{G}_{w, \prec \lambda}\} =0 \quad {\rm~for~all~}x\in \rg(G)\otimes\IK\otimes\Cn{\infty}.
 \end{equation}
 By construction of $\rho_{\pi}$,
 it is straightforward to see that
$\sigma^{\mathrm{stab}}\circ\rho_{\pi}^{i}\circ\rho_{\bar{\pi}}^{j}$ is unitarily equivalent to ${(\sigma\otimes\pi^{\otimes i}\otimes\bar{\pi}^{\otimes j})}^{\mathrm{stab}}$ for every $i,j \in \IN$.
 If a nonzero projection $p\in \rg(G)\otimes\IK\otimes\Cn{\infty}$ and $n\geq 1$ fail condition (2) in Theorem \ref{Thm:simple},
 then Lemma \ref{lem inner} implies the Murray--von Neumann equivalences
 \begin{equation}\label{eq MV}
 \rho_{\pi^{\otimes n}}(p)\sim p\sim \rho_{\bar{\pi}^{\otimes n}}(p).
 \end{equation}
 For each $0 \leq i, j \leq n-1$,
 put
 \[S_{i,j}:=\left\{\sigma\in \hat{G}_{w,\prec\lambda}:(\sigma\otimes\pi^{\otimes i}\otimes\bar{\pi}^{\otimes j})^{\mathrm{stab}}(p)\neq 0\right\}.\]
 Define $F:=\bigcap_{0 \leq i,j \leq n-1}\left(\hat{G}_{w,\prec\lambda}\setminus S_{i,j}\right)$.
 We show $F=\emptyset$.
 If $F\neq \emptyset$,
 then $\mu:=\bigoplus_{\sigma\in F}\sigma$ satisfies the relations
 \[\mu\prec\lambda, \quad (\mu\otimes\pi^{\otimes i}\otimes\bar{\pi}^{\otimes j})^{\mathrm{stab}}(p)=0 \quad {\rm~for ~all~} 0\leq i, j \leq n-1.\]
 By construction,
 $\mu$ is, with respect to weak containment, the maximal representation with these properties. 
 By using the unitary equivalence among $(\sigma\otimes\pi^{\otimes i}\otimes\bar{\pi}^{\otimes j})^{\mathrm{stab}}$, $\sigma^{\mathrm{stab}}\circ\rho_{\pi}^{i}\circ\rho_{\bar{\pi}}^{j}$,
 and $\sigma^{\mathrm{stab}}\circ\rho_{\bar{\pi}}^{j}\circ\rho_{\pi}^{i}$ together with
 condition (\ref{eq MV}), we obtain
 \begin{align*}
 ((\mu\otimes\pi)\otimes\pi^{\otimes i}\otimes\bar{\pi}^{\otimes j})^{\mathrm{stab}}(p)&=0,\\
 ((\mu\otimes\bar{\pi})\otimes\pi^{\otimes i}\otimes\bar{\pi}^{\otimes j})^{\mathrm{stab}}(p)&=0
 \end{align*}
 for every $0\leq i, j \leq n-1$.
 By the maximality of $\mu$,
 we have $\mu\otimes\pi,\;\mu\otimes\bar{\pi}\prec\mu$.
 As $\mu\leq\mu\otimes\bar{\pi}\otimes\pi\prec\mu\otimes\pi$,
 we get $\mu\sim\mu\otimes\pi$.
 By condition (2),
 these relations yield $\lambda\sim\mu$.
 This contradicts to the equality $\mu^{\mathrm{stab}}(p)=0$.
 Hence $F=\emptyset$, and thus $\hat{G}_{w,\prec\lambda}=\bigcup_{i,j=0}^{n-1}S_{i,j}$.
 Put $a:=\sum_{i,j=0}^{n-1}\rho_{\pi}^{i}\circ\rho_{\bar{\pi}}^{j}(p)$.
For each $\sigma\in S_{i,j}$,
we have 
\[\|\sigma^{\mathrm{stab}}(a)\|\geq\|\sigma^{\mathrm{stab}}\circ\rho_{\pi}^{i}\circ\rho_{\bar{\pi}}^{j}(p)\|=\|(\sigma\otimes\pi^{\otimes i}\otimes\bar{\pi}^{\otimes j})^{\mathrm{stab}}(p)\|=1.\]
This contradicts to (\ref{equation:vanish2}).
This proves (2)$\Rightarrow$(3).

The converse follows from Theorems \ref{Thm:simple} and \ref{Thm:simpleproless}.
\end{proof}

\begin{Exm}
 If $G$ is non-discrete and the center ${\rm Z}(\rg(G))$ is non-degenerate in $\rg(G)$,
 then $G$ satisfies the assumption (\ref{equation:vanish}) of Proposition \ref{prop outer}.
 Indeed,
 for each $\sigma\in\hat{G}_{w,\prec\lambda}$,
the restriction $\sigma|_{{\rm Z}(\rg(G))}$ is a character of ${\rm Z}(\rg(G))$ and the set $\{\sigma|_{{\rm Z}(\rg(G))}: \sigma\in \hat{G}_{w,\prec\lambda}\}$ is equal to the Gelfand spectrum of ${\rm Z}(\rg(G))$ (by the Hahn--Banach theorem and the GNS-construction).
Since ${\rm Z}(\rg(G))$ is non-unital,
we obtain 
\[\inf\{\|\sigma(a)\|:\sigma\in\hat{G}_{w,\prec\lambda}\}=0\]
for every $a\in {\rm Z}(\rg(G))$.
As ${\rm Z}(\rg(G))$ is non-degenerate in $\rg(G)$,
$G$ satisfies the assumption in Proposition~\ref{prop outer}.

For instance, if $G\rtimes K$ is a non-discrete semidirect product of an abelian group $G$ by a compact group $K$,
then the center ${\rm Z}(\rg(G\rtimes K))$ contains $C_{0}(\hat{G}/K)\otimes {\rm Z}(\rg(K))$, which is non-degenerate in $\rg(G\rtimes K)$.
Hence
$G\rtimes K$ satisfies the assumption (\ref{equation:vanish}) of Proposition \ref{prop outer}.
\end{Exm}

Since the inclusion $\Cn{n}^{{\rm U}(n)}\subset \Cn{n}$ is irreducible (see \cite[Theorem 3.2]{MR818932} and \cite[Corollary 3.3]{MR901232}),
every faithful quasi-free action is outer.
Hence, for discrete groups,
the reduced crossed product of a faithful quasi-free action is simple thanks to Kishimoto's theorem \cite{Kis}.

If $G$ is a connected locally compact group with a faithful finite-dimensional unitary representation,
then $G$ must be isomorphic to a direct product of a compact connected semisimple Lie group and $\IR^{n}\times\IT^{m}$ for some $n, m\in \IN$.
Indeed,
if $\pi\colon G\to {\rm U}(n)$ is a faithful finite-dimensional representation,
then for any open neighborhood $1\in V \subsetneq {\rm U}(n)$,
$\pi^{-1}(V)$ does not contain a nontrivial subgroup.
Hence,
the Gleason--Yamabe theorem shows that $G$ is a connected Lie group.
By the Levi--Malcev theorem, $G$ is isomorphic to the semidirect product $H\rtimes K$ of a semisimple Lie group $K$ and a solvable Lie group $H$.
By the structure theory of Lie groups together with the existence of a faithful finite-dimensional unitary representation,
it follows that $K$ is compact and $H\cong \IR^{n}\times\IT^{m}$ for some $m,n\in\IN$.
Consider the irreducible decomposition $\pi|_{H}=\bigoplus_{j=1}^{k}\chi_{j}^{\oplus m_{j}}$, where $\chi_1, \ldots, \chi_k\in \hat{H}$ are pairwise distinct.
Then, as $K$ normalizes $H$, $K$ acts on the finite set $\{\chi_{1},\dots,\chi_{k}\}$ by permutations.
Since $K$ is connected,
this action must be trivial.
As $\pi$ is faithful, $\chi_1, \ldots, \chi_k$ separate points of $H$.
Hence the action $K \acts H$ is also trivial.
This proves $G\cong \IR^{n}\times\IT^{m}\times K$.
We note that, as $H$ is abelian, any irreducible unitary representation $\pi$ of $H\times K$ is of the form
$\pi(h, k)=\chi(h)\sigma(k)$ for some $\chi \in \hat{H}$ and $\sigma \in \hat{K}$.
We denote $\pi$ as $(\chi, \sigma)$.

Consider the direct product $G\times K$ of a compact group $K$ and a locally compact abelian group $G$.
Let
\[
\Pi=\bigoplus_{i=1}^{n}(t_{i},\pi_{i})
\]
be a faithful finite-dimensional unitary representation of $G\times K$,
where $t_{1},\dots,t_{n}\in\hat{G}$ are pairwise distinct and $\pi_{1},\dots,\pi_{n}$ are finite-dimensional representations of $K$.
Associated to $\Pi$, we define a directed graph $\cG_{\Pi}=(V_{\Pi},E_{\Pi})$ as follows:
\begin{itemize}
 \item The vertex set is defined to be $V_{\Pi}:=\hat{G}\times\hat{K}$.
 \item For $(s_{1},\sigma_{1}), (s_{2},\sigma_{2})\in V_{\Pi}$,
 the number of edges from $(s_{1},\sigma_{1})$ to $(s_{2},\sigma_{2})$ is equal to $\dim ((s_1, \sigma_1), \Pi \otimes (s_2, \sigma_2))=\sum_{i=1}^{n}\dim\;(\sigma_{1},\;\pi_{i}\otimes\sigma_{2})\delta_{s_{1},\,t_{i}s_{2}}$.
 \end{itemize}
 We point out that it is occasionally helpful to realize the set of edges from $(s_1, \sigma_1)$ to $(s_2, \sigma_2)$ as a linear basis of the intertwiner space $((s_1, \sigma_1), \Pi\otimes (s_2, \sigma_2))$.
 
We next introduce a few related notations.
For $w=(w_{1},\dots,w_{m})\in\bigsqcup_{l=0}^{\infty}\{1,\dots,n\}^{l}$,
we put
\[t_{w}:=t_{w_{1}}t_{w_{2}}\cdots t_{w_{m}}, \quad \pi_{w}:=\bigotimes_{j=1}^{m}\pi_{w_{j}}.\]
Here, when $w=\emptyset \in\{1,\dots,n\}^{0}$ (the empty word), we set $t_{\emptyset}:=1_{\hat{G}}$ and $\pi_{\emptyset}:=1_{\hat{K}}$.
For $\sigma_1, \sigma_2\in \hat{K}$, we define
\begin{equation}\label{eq:words}
 W(\sigma_{1}\to\sigma_{2}):=\{w\in\bigsqcup_{m=0}^{\infty}\{1,\dots,n\}^{m}: \sigma_{1}\leq \pi_{w}\otimes\sigma_{2}\}.
 \end{equation}
\begin{Prop}\label{prop topological graph}
Under the notation introduced above,
the following conditions are equivalent.
 \begin{enumerate}
 \item $\Cn{\cH_{\Pi}}\rca{\alpha_\Pi}(G\times K)$ is simple.
 \item For any infinite directed path $\beta=\{(s_{j},\sigma_{j})\}_{j=1}^{\infty}$ of $\cG_{\Pi}$ and any $\sigma\in\hat{K}$, one has the equality
 \begin{equation}\label{eq:cofinal}
 \cl\bigg(\bigcup_{j=1}^{\infty}\{t_{w}s_{j}\in\hat{G}: w\in W(\sigma\to\sigma_{j})\}\bigg)=\hat{G}.\end{equation}
 \end{enumerate}
\end{Prop}
\begin{proof}
 Suppose that (1) holds.
 For a given infinite directed path $\beta=\{(s_{j},\sigma_{j})\}_{j=1}^{\infty}$ of $\cG_{\Pi}$,
 we set
 \[\Sigma_{\beta}:=\bigoplus_{\sigma\in \hat{K}}\bigoplus_{j=1}^{\infty}\bigoplus_{w\in W(\sigma\to\sigma_{j})}(t_{w}s_{j},\,\sigma).\]
 By construction,
 it is straightforward to verify that $\Sigma_{\beta}\otimes\Pi\sim\Sigma_{\beta}$.
 Thus $\Sigma_{\beta} \sim \lambda$ by Proposition \ref{prop outer}.
 To lead to a contradiction, assume that condition (2) fails for some $\sigma\in \hat{K}$.
 Then one can take a positive norm one element $f\in C_{0}(\hat{G})\cong\rg(G)$ with $f(t_{w}s_{j})=0$ for all $j \geq 1$ and $w\in W(\sigma\to\sigma_{j})$.
Take the central support $\chi_{\sigma}\in {\rm Z}(\rg(K))$ of $\sigma$.
Then $f\otimes\chi_{\sigma}\in\rg(G)\otimes \rg(K)= \rg(G\times K)$ is a norm one element in $\ker \Sigma_{\beta}$.
This contradicts to the condition $\Sigma_{\beta} \sim \lambda$.
Hence $\Pi$ satisfies condition (2).

Conversely, 
if $\Cn{\cH_\Pi}\rca{\alpha_\Pi}(G\times K)$ is not simple,
then by Proposition \ref{prop outer}, one has a unitary representation $\Sigma$ of $G\times K$ satisfying $\Sigma\sim\Sigma\otimes\Pi$ and $0\neq\ker\Sigma\lhd \rg(G\times K)$.
Under the isomorphism $\rg(G\times K)\cong\bigoplus_{\sigma\in\hat{K}}C_{0}(\hat{G})\otimes \IB(\cH_{\sigma})$,
we have $\ker\Sigma\cong\bigoplus_{\sigma\in\hat{K}}C_{0}(\hat{G}\setminus F_{\sigma})\otimes \IB(\cH_{\sigma})$ for some closed sets $F_{\sigma}\subset \hat{G}$; $\sigma \in \hat{G}$.
By the choice of $(F_{\sigma})_{\sigma\in\hat{K}}$, one has
\begin{align*}
 \Sigma&\sim\bigoplus_{\sigma\in\hat{K}}\bigoplus_{s\in F_{\sigma}}(s,\sigma).
\end{align*}
Since $\Sigma\sim\Sigma\otimes\Pi$,
for every $\sigma\in\hat{K}$,
one has
\[F_{\sigma}=\bigcup_{i=1}^{n}\bigcup_{\sigma\leq\pi_{i}\otimes\sigma^{\prime}}t_{i}F_{\sigma^{\prime}}.\]
Therefore one has an infinite directed path $\beta=\{(s_{j},\sigma_{j})\}_{j=1}^{\infty}$ of $\cG_{\Pi}$ with $s_{j}\in F_{\sigma_{j}}$ for all $j\geq 1$.
Choose $\sigma\in\hat{K}$ with $F_{\sigma}\neq\hat{G}$.
We then obtain
 \[
 \cl\bigg(\bigcup_{j=1}^{\infty}\{t_{w}s_{j}\in\hat{G}: w\in W(\sigma\to\sigma_{j})\}\bigg)\subset F_{\sigma}\neq\hat{G}.\]
Thus, condition (2) fails.
\end{proof}
As an explicit example,
we give detailed calculations in the case $G=\IR$ and $K=\SU(2)$.

The following basic fact will be used in our analysis.
For completeness, we include a proof.
\begin{Lem}\label{Lem:subsem}
Let $S\subset\IR$ be a closed subsemigroup. If both $S \cap \IR_{>0}$ and $S \cap\IR_{<0}$ are non-empty, then $S$ is a closed subgroup of $\IR$.
\end{Lem}
\begin{proof}
Put
\[
d_{+}:=\inf(S\cap\IR_{>0}),\qquad d_{-}:=\sup(S\cap\IR_{<0}).
\]
Then, since $S$ is closed, one has $d_{+},d_{-}\in S$, and hence $d_{+}+d_{-}\in S$.
By the definitions of $d_{+}$ and $d_{-}$, we must have $d_{+}=-d_{-}$.
If $d_+=0$, then $0$ is an accumulation point of both $S\cap\IR_{>0}$ and $S \cap \IR_{<0}$, hence $S=\IR$.
If $d_+>0$, then $d_+$ (resp.~ $d_-$) is the smallest positive element (resp. the largest negative element) of $S$, and therefore $S=d_+ \IZ$.
\end{proof}
\begin{Prop}\label{prop:SU(2)}
 Let $\Pi$ be a finite dimensional faithful unitary representation of $G:=\SU(2)\times\IR$. 
 Then the following conditions are equivalent:
 \begin{enumerate}
 \item $\Cn{\cH_{\Pi}}\rca{\alpha_\Pi}G$ is simple.
 \item $F(\Pi) \sim \lambda$.
 \item $\Cn{\cH_{\Pi}}\rca{\alpha_\Pi} G$ is purely infinite simple.
 \end{enumerate}
\end{Prop}
\begin{proof}
For each $n\in \IN$, denote by $\pi_{n} (= \bar{\pi}_n)$
the unique irreducible ($n+1$)-dimensional unitary representation of $\SU(2)$.
Recall that 
\begin{equation}\label{eq_SU2}\pi_{n}\otimes \pi_{m}=\pi_{|n-m|}\oplus \pi_{|n-m|+2}\oplus\cdots\oplus \pi_{n+m}\end{equation} holds for all $n,m \in \IN$.
Pick $r_{1},\dots, r_{M},t_{1},\dots,t_{N}\in\hat{\IR}$ and $n_1, \ldots, n_N \geq 1$ with
 \[\Pi=\bigoplus_{i=1}^{N}(t_{i},\pi_{n_{i}})\oplus\bigoplus_{j=1}^{M}(r_{j},1_{\widehat{\SU(2)}}).\]
 Note that $N\geq 1$ as $\Pi$ is faithful. 
We regard $r_{1},\dots, r_{M},t_{1},\dots,t_{N}$ as elements of the additive group $\IR$ via the canonical isomorphism $\hat{\IR} \cong \IR$.

We show (1)$\Rightarrow$(2).
Let $S_{\Pi}$ be the closed subsemigroup of $\IR$ generated by the sequence $r_{1},\dots, r_{M},t_{1},\dots,t_{N}$.
We claim that $S_{\Pi}=\IR$.
To prove the claim, we first show that the sequence contains both positive and negative elements.
By reindexing the sequences if necessary,
we may assume that $\frac{|t_{1}|}{n_{1}}$ is minimal among $\frac{|t_{i}|}{n_{i}}; 1\leq i\leq N$.
We consider the infinite directed path $\beta:=\{(s_{k},\sigma_{k})\}_{k=1}^{\infty}$ given by $s_{k}:=-kt_{1}$ and $\sigma_{k}:=\pi_{kn_{1}}$ for $k \geq 1$.
Then we have
\begin{align}\label{eq_half}
&\bigcup_{k=1}^{\infty}\{t_{w}+s_{k}\in\IR: w\in W(1_{\widehat{\SU(2)}} \to\sigma_{k})\}\\
\subset&\left\{-kt_{1}+\sum_{i=1}^{N}a_{i}t_{i}+\sum_{j=1}^{M}b_{j}r_{j}\in\IR:\,k,a_{i},b_{j}\in \IN,\;\sum_{i=1}^{N}a_{i}n_{i}\geq kn_{1}\right\}.\nonumber
\end{align}
Indeed, for each $w\in W(1_{\widehat{\SU(2)}} \to\pi_{kn_{1}})$,
there exist $a_{1},\dots,a_{N}\in \IN$ satisfying $\pi_{w}=\bigotimes_{i=1}^{N}\pi_{n_{i}}^{\otimes a_{i}}$.
Since we have $1_{\widehat{\SU(2)}} \leq \pi_{w}\otimes \pi_{kn_{1}}$, equivalently, $(\pi_{w},\,\pi_{kn_{1}})\neq0$,
the formula (\ref{eq_SU2}) implies $\sum_{i=1}^{N}a_{i}n_{i}\geq kn_{1}$.
We next show that at least one of the numbers $r_{1},\dots, r_{M},t_{1},\dots,t_{N}$ is negative.
Indeed, if not,
then we have 
\begin{align*}
-kt_{1}+\sum_{i=1}^{N}a_{i}t_{i}+\sum_{j=1}^{M}b_{j}r_{j}&\geq -kt_{1}+\sum_{i=1}^{N}a_{i}n_{i}\frac{t_{i}}{n_{i}}\\
&\geq -kt_{1}+\left(\sum_{i=1}^{N}a_{i}n_{i}\right)\frac{t_{1}}{n_{1}}\geq -kt_{1}+kt_{1}=0.
\end{align*}
Hence, the right-hand side of \eqref{eq_half} is contained in $\IR_{\geq0}$.
By the simplicity of $\Cn{\cH_{\Pi}}\rca{\alpha_\Pi}G$ together with Proposition \ref{prop topological graph} (2), 
the right hand side of \eqref{eq_half} is dense in $\IR$. 
This is a contradiction.
Thus, 
we get $S_{\Pi}\cap\IR_{<0}\neq\emptyset$.
Similarly, we have $S_{\Pi}\cap\IR_{>0}\neq\emptyset$.
Thus, by Lemma \ref{Lem:subsem}, $S_{\Pi}$ is a subgroup of $\IR$.
This implies that $S_\Pi$ contains the left hand side of \eqref{eq_half}, which is dense in $\IR$.
Thus $S_{\Pi}=\IR$.
Since we have $1_{\widehat{\SU(2)}} \leq\pi_{n}^{\otimes 2}$ for any $n\in \IN$, we have
\begin{align*}
 2S_{\Pi}=\cl \left\{\sum_{i=1}^{N}2a_{i}t_{i}+\sum_{j=1}^{M}2b_{j}r_{j} : a_{i}, b_{j}\in \IN\right\}
 \subset \cl\left\{t_{w}\in\IR:w\in W(1_{\widehat{\SU(2)}}\to1_{\widehat{\SU(2)}})\right\}.
\end{align*}
Hence, the semigroup ${\alpha_\upsilon}:=\left\{t_{w}\in\IR:w\in W(1_{\widehat{\SU(2)}}\to1_{\widehat{\SU(2)}})\right\}$ is dense in $\IR$.
To obtain $F(\Pi)\sim\lambda$,
we follow the strategy of the proof of (2)$\Rightarrow$ (1) in Proposition \ref{prop topological graph}.
For $n\in \IN$, we define a closed subset $F_{n}\subset\IR$ to be $F_{n}:=\{s\in \IR:(s,\pi_{n})\prec F(\Pi)\}$.
It suffices to show that $F_{n}=\IR$ for all $n\in \IN$.
We first note that as $\pi_{m}$ factors through the quotient group ${\rm SO}(3)$ for every $m\in 2\IN$,
the faithfulness of $\Pi$ implies that there exists $1\leq i \leq n$ such that $n_i$ is odd.
Then, for any $n\in \IN$, by \eqref{eq_SU2}, for some $k\in \IN$, one has
 $(kt_{i},\pi_{n})\leq(t_{i},\pi_{n_{i}})^{\otimes k}\leq F(\Pi)$.
Hence $F_n \neq \emptyset$ for $n\in \IN$.
By definition,
for any $s\in {\alpha_\upsilon}$,
there exists $m\in \IN$ satisfying
$(s,1_{\widehat{\SU(2)}})\leq \Pi^{\otimes m}$.
Hence,
for any $n\in \IN$, $s^{\prime}\in F_{n}$, and $s\in{\alpha_\upsilon}$,
we have
\[(s^{\prime}+s,\pi_{n})=(s^{\prime},\pi_{n})\otimes(s,1_{\widehat{\SU(2)}})\prec F(\Pi).\]
That is, $F_{n}+{\alpha_\upsilon}\subset F_{n}.$
Since $F_{n}$ is a non-empty closed subset of $\IR$ and ${\alpha_\upsilon}\subset\IR$ is dense,
we conclude $F_{n}=\IR$.

The implication (2)$\Rightarrow$ (3) follows from a more general result Corollary \ref{cor:purely_infinite}.

The implication (3)$\Rightarrow$ (1) is trivial.
\end{proof}

\subsection*{Compact group quasi-free actions and graph \Cs-algebras}
In \cite[Section 7]{Izuqp},
the following condition (*) is introduced for a unitary representation $\pi$ of a compact group $K$,
which gives a sufficient condition for the quasi-free action $\alpha_\pi$ to be isometrically shift-absorbing (in the sense of \cite{GS2}):\\
(*)\quad For every $\sigma \in \hat{K}$, there exists $n\in \IN$ with $\sigma\leq \pi^{\otimes n}$; equivalently, every $\sigma\in \hat{K}$ is contained in $F(\pi)$.\\
In the infinite-dimensional case, Izumi characterized
isometric shift-absorption of a quasi-free action by condition (*) \cite[Proposition 7.2]{Izuqp}.
Here we prove a similar statement for finite-dimensional representations based on a fundamental fact about graph \Cs-algebras.
This settles a question mentioned in \cite{Izuqp}.
It is classically known that the fixed point algebra of a quasi-free action of a compact group on $\Cn{n}$ is isomorphic to a corner of a graph \Cs-algebra. 
By using the Hao--Ng isomorphism theorem,
we directly construct a Morita equivalence between the crossed product \Cs-algebra (which contains the fixed point algebra as a corner)
 and the corresponding graph \Cs-algebra.

We recall the definition of graph \Cs-algebras.
\begin{Def}[Section 1 of \cite{MR1777234}]\label{defn:graphalgebra}
 Let $\cG$ be a directed graph with a vertex set $V$ and an edge set $E$.
 The source map (resp. range map) of $\cG$ is denoted by $s\colon E\to V$ (resp. $r\colon E\to V$).
 The graph \Cs-algebra \Cs$(\cG)$ is the universal \Cs-algebra generated by pairwise orthogonal projections $(p_{v})_{v\in V}$ and partial isometry elements $(S_{e})_{e\in E}$ with the following relations.
 \begin{itemize}
 \item $S_{e}^{*}S_{e}=p_{r(e)}$ for all $e\in E$.
 \item $\sum_{s(e)=v}S_{e}S_{e}^{*}=p_{v}$ when $s^{-1}(v)$ is finite and non-empty..
 \end{itemize}
\end{Def}
\begin{Rem}\label{rmk graph algebra}
For a directed graph $\cG=(V,E)$, let $\{q_{v}\}_{v\in V}$ be the set of minimal projections of $c_{0}(V)$.
Define a \Cs-correspondence $\cE_{\cG}$ over $c_0(V)$ by specifying the multiplicities
\[\dim(q_{v}\cE_{\cG}q_{w})=|\{e\in E: r(e)=w ,\;s(e)=v\}|\]
for every $v,w\in V$.
(Equivalently, one can define $\cE_{\cG}$ as a completion of $c_c(E)$ equipped with a suitable pre-\Cs-correspondence structure.)
It is straightforward to show that the graph \Cs-algebra ${\rm C}^{*}(\cG)$ is isomorphic to the Cuntz--Pimsner algebra $\Cn{\cE_{\cG}}$.
\end{Rem}

Let $K$ be a compact second-countable group and $\pi\colon K\to \cU(\cH_{\pi})$ be a finite-dimensional unitary representation.
Using the definition before Proposition \ref{prop topological graph} with $G=\{1\}$,
we obtain the directed graph $\cG_{\pi}=(V_{\pi},E_{\pi})$ associated with $\pi$.
We prove the following key lemma based on the Hao--Ng isomorphism.
\begin{Lem}\label{lem_morita}
The graph \Cs-algebra ${\rm C}^{*}(\cG_{\pi})$ is stably isomorphic to $\Cn{\cH_{\pi}}\rca{\alpha_\pi} K$.
\end{Lem}
\begin{proof}
 Let $X$ be the \Cs-($\rg(K)$, $c_{0}(V_{\pi})$)-correspondence given by
 \[X:=\bigoplus_{\sigma\in\hat{K}}\cH_{\sigma}\otimes q_{\sigma}\] as a right Hilbert \Cs-$c_{0}(V_{\pi})$-module,
 equipped with the left action of $\rg(K)$ via the unitary representation
 \[\bigoplus_{\sigma\in\hat{K}}\sigma\colon K\to \prod_{\sigma\in\hat{K}} \cU(\cH_{\sigma})\cong\cU(X).\]
 Here
 $\{q_{\sigma}\}_{\sigma\in\hat{K}}$ denotes the set of minimal projections of $c_{0}(V_{\pi})$.
 Then $X$ gives Morita equivalence between $\rg(K)$ and $c_{0}(V_{\pi})$.
 We use the notation in Remark \ref{rmk graph algebra},
 and put $\cE_{\pi}:=\cE_{\cG_{\pi}}$ for short.
 Then one has the isomorphism
 \begin{equation}\label{eq:corr}(\cH_{\pi}\rtimes_{\rm r} K)\otimes_{\rg(K)}X\cong X\otimes_{c_{0}(V_{\pi})}\cE_{\pi}\end{equation}
 of \Cs-($\rg(K)$, $c_{0}(V_{\pi})$)-correspondences.
 Indeed,
 for each $\sigma\in\hat{K}$,
 one has the following unitary isomorphisms of $K$-Hilbert spaces:
 \begin{align*}
 (X\otimes_{c_{0}(V_{\pi})}\cE_{\pi})\cdot q_{\sigma}
 &\cong\bigoplus_{\sigma^{\prime}\in\hat{K}}\,\left(\cH_{\sigma^{\prime}}\otimes q_{\sigma^{\prime}}\cE_{\pi}q_{\sigma}\right)\\
 &\cong\bigoplus_{\sigma^{\prime}\in\hat{K}}\left(\cH_{\sigma^{\prime}}\otimes (\sigma^{\prime},\;\pi\otimes\sigma)\right)\\
 &\cong\cH_{\pi}\otimes\cH_{\sigma} \cong\left((\cH_{\pi}\rtimes_{\rm r} K)\otimes_{\rg(K)}X\right)\cdot q_{\sigma}.
 \end{align*}
 Here,
 in the first and second isomorphisms,
 the left $K$-actions on $\cH_{\sigma^{\prime}}\otimes q_{\sigma^{\prime}}\cE_{\pi}q_{\sigma}$ and $\cH_{\sigma^{\prime}}\otimes (\sigma^{\prime},\;\pi\otimes\sigma)$ are given by $\sigma^{\prime}\otimes 1$.
 In the third isomorphism, 
 the $K$-action on $\cH_{\pi}\otimes\cH_{\sigma} $ is given by $\pi\otimes\sigma$.
 By \cite[Theorem 3.5]{MR1757049},
 we obtain the Morita equivalence between $\Cn{\cH_{\pi}\rtimes_{\rm r} K}\cong\Cn{\cH_{\pi}}\rca{\alpha_\pi} K$ and $\Cn{\cE_{\pi}}\cong{\rm C}^{*}(\cG_{\pi})$.
 This shows the claim.
\end{proof}
\begin{Rem}
 In \cite[Proposition 5.1]{MR1777234}, it is proved that the graph \Cs-algebra ${\rm C}^{*}(\cG_{\pi})$ is simple if and only if $\cG_{\pi}$ is cofinal (see \cite[Section 5]{MR1777234} for the definition).
It follows that the cofinality of $\cG_{\pi}$ is equivalent to condition (2) of Proposition \ref{prop topological graph} for the case $G=\{1\}$.
\end{Rem}

\begin{Rem}\label{rmk:graph_algebra}
Again let $K$ be a compact group and let $G$ be a locally compact abelian group,
and consider a unitary representation $\Pi\colon K\times G\to \cU(\cH_{\Pi})$.
Put $\pi:=\Pi|_{K}\colon K\to \cU(\cH_{\Pi})$ and $\tau:=\Pi|_{G}\colon G\to \cU(\cH_{\Pi})$.
We use the notation $\cG_{\pi}$ and $\cE_{\pi}$ as in Lemma \ref{lem_morita}.
The isomorphism \eqref{eq:corr} of \Cs-correspondences gives rise to a unitary representation
\[
\tilde{\tau}\colon G\to \cU(\cE_{\pi})\cap c_{0}(V_{\pi})^{\prime}
\]
associated with the representation
\[
\tau\otimes 1_{\rg(K)}\colon G\to \cU(\cH_{\Pi}\rtimes_{\rm r}K)\cap(\pi\otimes\lambda_{K})(K)^{\prime}.
\]
Indeed, let
\[\Pi=\bigoplus_{i=1}^{n}(\pi_{i},t_{i})\]
 be the irreducible decomposition of $\Pi$. 
Then one has
\[
\pi=\bigoplus_{i=1}^{n}\pi_{i},
\qquad
\tau=\bigoplus_{i=1}^{n}t_{i}^{\oplus \dim \pi_{i}}.
\]
For each $\sigma,\sigma^{\prime}\in V_{\pi}$,
the representation $\tilde{\tau}$ of $G$ on
$
q_{\sigma^{\prime}}\cE_{\pi}q_{\sigma}\cong\bigoplus_{i=1}^n(\sigma^{\prime},\pi_{i}\otimes\sigma)
$
is given by
$
\bigoplus_{i=1}^n 1_{(\sigma^{\prime},\pi_{i}\otimes\sigma)}\otimes t_{i}.
$
We denote the quasi-free action $G\acts\Cn{\cE_{\pi}}$ of $\tilde{\tau}$ by $\alpha_{\tilde{\tau}}$.

Using the $\rg(K)$-$c_{0}(V_{\pi})$ correspondence $X$ in \eqref{eq:corr},
we obtain an isomorphism
\[
(X\otimes\rg(G))\otimes_{c_{0}(V_{\pi})\otimes \rg(G)}(\cE_{\pi}\rtimes_{\rm r} G)
\cong
\big(\cH_{\Pi}\rtimes_{\rm r}(K\times G)\big)\otimes_{\rg(K\times G)}(X\otimes\rg(G))
\]
of \Cs-($\rg(K\times G)$, $c_{0}(V_{\pi})\otimes\rg(G)$)-correspondences.
Here $X\otimes\rg(G)$ denotes the exterior tensor product, while the other tensor products are interior.
Using the Hao--Ng isomorphism theorem and \cite[Theorem 3.5]{MR1757049},
it follows that the two \Cs-algebras
\[
\Cn{\cE_{\pi}}\rca{\alpha_{\tilde{\tau}}} G \cong \Cn{\cE_{\pi}\rtimes_{\rm r} G}
\quad\text{and}\quad
\Cn{\cH_{\Pi}}\rca{\alpha_\Pi} (K\times G)\cong \Cn{\cH_{\Pi}\rtimes_{\rm r}(K\times G)}
\]
are stably isomorphic.
\end{Rem}

To complete the proof of Proposition \ref{prop:SU(2)}, we include the following lemma. The pure infiniteness of crossed product \Cs-algebras arising from finite-dimensional 
unitary representations of abelian groups was characterized in \cite{MR1943097}. The same technique applies to quasi-free actions on graph \Cs-algebras associated with unitary representations of compact groups.
As the proof of the following lemma follows that of \cite{MR1943097} closely, we describe only the necessary modifications for the context of graph algebras.
This can also be obtained by applying Theorem A of \cite{Kat08} in an appropriate setting.
For the remaining details, we refer the reader to \cite{MR1943097} or \cite{Kat08}.
\begin{Lem}\label{lem:scaling_element}
Under the setting in Remark \ref{rmk:graph_algebra},
assume further that the abelian group $G$ is non-trivial and that
\[\dim(\pi^{\otimes m},1_{\hat{K}})\geq 2 \quad {\rm for~some~}m \in \IN.
\]
Set
\[
{\alpha_\upsilon}:=\{\chi\in\hat{G}: (1_{\hat{K}},\chi)\leq F(\Pi)\}.
\]
If ${\alpha_\upsilon}$ is dense in $\hat{G}$,
then there exists a scaling element $x$ in $\Cn{\cE_{\pi}}\rca{\alpha_{\tilde{\tau}}} G$ \cite{BC}; that is,
\[
(x^{*}x)(xx^{*}) = xx^{*}
\qquad\text{and}\qquad
x^{*}x\neq xx^{*}.
\]

Moreover, if we assume in addition that the graph $\cG_{\pi}$ is cofinal,
then $\Cn{\cE_{\pi}}\rca{\alpha_{\tilde{\tau}}} G$,
or equivalently $\Cn{\cH_{\Pi}}\rca{\alpha_\Pi}(K\times G)$,
is purely infinite simple.
\end{Lem}

\begin{proof}
Put $q:=q_{1_{\hat{K}}}\in c_{0}(V_{\pi})$.
For any $m\in\IN$,
we have
\[
q\cE_{\pi}^{\otimes m}q\cong (1_{\hat{K}},\pi^{\otimes m})
\cong (\cH_{\Pi}^{\otimes m})^{K}.
\]
For every unit vector $\xi\in q\cE_{\pi}^{\otimes m}q$,
let $S_{\xi}\in\Cn{\cE_{\pi}}$ be the associated partial isometry element.
Since the action of $G$ on $q\cE_{\pi}^{\otimes m}q$ is given by the restriction of $\tau^{\otimes m}$ to $(\cH_{\Pi}^{\otimes m})^{K}$,
we see that $(1_{\hat{K}},\chi)\leq \Pi^{\otimes m}$ if and only if there exists a unit vector $\xi\in q\cE_{\pi}^{\otimes m}q$ such that
\[
(\alpha_{\tilde{\tau}})_{g}(S_{\xi})=\chi(g)S_{\xi}
\]
for all $g\in G$.

By assumption,
we have $\dim q\cE_{\pi}^{\otimes m}q\geq 2$ for some $m$.
Thus,
\[
\dim q\cE_{\pi}^{\otimes km}q
\geq
\dim \big(q\cE_{\pi}^{\otimes m}q\big)^{\otimes k}
= \big(\dim q\cE_{\pi}^{\otimes m}q \big)^{k}
\to\infty
\qquad (k\to\infty).
\]
Hence, as in \cite[Lemma 4.3]{MR1943097},
for any neighborhood $U$ of $1\in \hat{G}$ and any $M\in\IN$,
one can choose unit vectors
\[
\xi_{1}^{\prime},\dots,\xi_{M}^{\prime}\in \bigcup_{m=1}^{\infty} q\cE_{\pi}^{\otimes m}q
\]
and characters $\chi_{1},\dots,\chi_{M}\in U\cap{\alpha_\upsilon}$ satisfying
\[
(\alpha_{\tilde{\tau}})_{g}(S_{\xi_{i}^{\prime}})=\chi_{i}(g)S_{\xi_{i}^{\prime}} \quad{\rm~for~all~} g\in G
,
\qquad
S_{\xi_{i}^{\prime}}^{*}S_{\xi_{j}^{\prime}}=\delta_{i,j}q
\]
for every $i,j=1,\dots,M$.

Take a compact neighborhood $X\subsetneq\hat{G}$ of $1\in \hat{G}$.
Following an argument similar to that of \cite[Lemma 4.4]{MR1943097} and applying the above claim,
one can choose sequences
\[
\xi_{1},\dots,\xi_{N}\in \bigcup_{m=1}^{\infty} q\cE_{\pi}^{\otimes m}q,
\qquad
\chi_{1},\dots,\chi_{N}\in{\alpha_\upsilon},
\qquad
f_{1},\dots,f_{N}\in C_{c}(\hat{G})_{+},
\]
satisfying the following conditions:
\begin{itemize}
\item $(\alpha_{\tilde{\tau}})_{g}(S_{\xi_{i}})=\chi_{i}(g)S_{\xi_{i}}$ for all $g\in G$ and $1 \leq i \leq N$,
\item $S_{\xi_{i}}^{*}S_{\xi_{j}}=\delta_{i,j}q$ for all $1\leq i, j\leq N$,
\item $\sum_{i=1}^{N}f_{i} \equiv 1$ on $X$,
\item $\chi_{i} \cdot \supp(f_{i})\subset X$ (the multiplication in the group $\hat{G}$)
\item $\sum_{i=1}^{N}f_{i}\in C_{0}(\hat{G})$ is not a projection (in the pointwise product).
\end{itemize}
Let $x_i:= \hat{f}_{i}\in \rg(G)$ be the Fourier transform of $f_{i}$.
Put
\[x:=\sum_{i=1}^{N}S_{\xi_{i}}x_{i}\in\Cn{\cE_{\pi}}\rca{\alpha_{\tilde{\tau}}} G.\]
Then, just as in \cite[Proposition 4.5]{MR1943097},
we conclude that $x$ is a scaling element.

If the graph $\cG_{\pi}$ is cofinal,
then the left-hand side of \eqref{eq:cofinal} is non-empty for any directed path and any $\sigma\in\hat{K}$.
Since the left-hand side of \eqref{eq:cofinal} is ${\alpha_\upsilon}$-invariant,
$\Pi$ satisfies condition (2) in Proposition \ref{prop topological graph}.
Hence,
$\Cn{\cH_{\Pi}}\rca{\alpha_\Pi} (K\times G)$ is simple.
By Remark \ref{rmk:graph_algebra},
$\Cn{\cE_{\pi}}\rca{\alpha_{\tilde{\tau}}} G$ and the corner
$
q(\Cn{\cE_{\pi}}\rca{\alpha_{\tilde{\tau}}} G)q\cong \left(q\Cn{\cE_{\pi}}q\right)\rca{\alpha_{\tilde{\tau}}} G
$
are also simple.
The corner $q(\Cn{\cE_{\pi}}\rca{\alpha_{\tilde{\tau}}} G)q$ contains 
an infinite projection by \cite[Proposition 4.2]{MR1943097}, as it contains a scaling element.

We now extend \cite[Lemma 4.6]{MR1943097} to our setting.
Let
\[
{\alpha_\upsilon}\colon\IT\acts\Cn{\cE_{\pi}}\rca{\alpha_{\tilde{\tau}}} G
\]
be the action induced by the gauge action on $\Cn{\cE_{\pi}}$. Let $E$ be the averaging conditional expectation of ${\alpha_\upsilon}$.
Take $M\in \IN$ and sequences
\[
\xi_{1},\dots,\xi_{L},\eta_{1},\dots,\eta_{L}\in\bigcup_{m=1}^{M}q\cE_{\pi}^{\otimes m},\quad\quad f_{1},\dots,f_{L}\in\rg(G),
\]
such that
$
y:=\sum_{l=1}^{L}S_{\xi_{l}}f_{l}S_{\eta_{l}}^{*} \in q(\Cn{\cE_{\pi}}\rca{\alpha_{\tilde{\tau}}} G)q
$
is positive and nonzero.
We next construct an element
$
a\in q\Cn{\cE_{\pi}}q
$
satisfying
\[
\|a\|\leq1,
\qquad
aya^{*}\in\rg(G)q,\qquad \|aya^{*}\|=\|E(y)\|.
\]
To adapt the proof of \cite[Lemma 4.6]{MR1943097} to the graph algebra setting,
we first fix orthonormal bases which respect the path structure of $\cG_{\pi}$ and the $G$-action.
First, for each $\sigma,\sigma^{\prime}\in V_{\pi}$,
we fix an orthonormal basis
$
\ONB_{G}(q_{\sigma}\cE_{\pi}q_{\sigma^{\prime}})
$
of $q_{\sigma}\cE_{\pi}q_{\sigma^{\prime}}$ each of whose elements $\zeta$ is an eigenvector for $\tilde{\tau}$; that is,
\[
\tilde{\tau}(g)\zeta=\chi(g)\zeta
\]
for some $\chi\in \hat{G}$.
Then, for $m \geq 2$, we define
\[
\ONB_{G}(q_{\sigma}\cE_{\pi}^{\otimes m}q_{\sigma^{\prime}}):=\{\zeta=\zeta_{1}\otimes\cdots\otimes\zeta_{m}\in q_{\sigma}\cE_{\pi}^{\otimes m}q_{\sigma^{\prime}}\;:\;
\zeta\neq0,\;\zeta_{i}\in\ONB_{G}(q_{\sigma_{i-1}}\cE_{\pi}q_{\sigma_{i}})\}.
\]
Note that if we identify 
$\ONB_{G}(q_{\sigma}\cE_{\pi}q_{\sigma^{\prime}})$
with the set of all edges of $\cG_\pi$ from $\sigma$ to $\sigma'$ via a fixed bijection,
then each $0 \neq \zeta=\zeta_{1}\otimes\cdots\otimes\zeta_{m}\in\ONB_{G}(q_{\sigma}\cE_{\pi}^{\otimes m}q_{\sigma^{\prime}})$ corresponds to the directed path $\beta_{\zeta}\colon\sigma_{0}\overset{\zeta_{1}}{\to}\sigma_{1}\cdots\overset{\zeta_{m}}{\to}\sigma_{m}$ of length $m$ in $\cG_{\pi}$.
In what follows, we will freely use these bijective correspondences.

Choose the smallest $m\in \IN$ with $\dim(1_{\hat{K}},\pi^{\otimes m})\geq 2$.
Pick two distinct directed loops
\[\beta_{\zeta}\colon1_{\hat{K}}\overset{\zeta_{1}}{\to}\sigma_{1}\cdots\overset{\zeta_{m}}{\to}1_{\hat{K}}\quad\text{and}\quad 
\beta_{\zeta^{\prime}}\colon1_{\hat{K}}\overset{\zeta_{1}^{\prime}}{\to}\sigma_{1}^{\prime}\cdots\overset{\zeta_{m}^{\prime}}{\to}1_{\hat{K}}\]
of $\cG_\pi$.
Let $i:=\min\{j:\zeta_{j}\neq\zeta_{j}^{\prime}\}$.
By the minimality of $m$, the common subpath
\[1_{\hat{K}}\overset{\zeta_{1}}{\to}\sigma_{1}\cdots\overset{\zeta_{i-1}}{\to}\sigma_{i-1}\]
of $\beta_\zeta$ and $\beta_{\zeta^{\prime}}$ does not contain a subloop.
By using $\beta_{\zeta}$ and $\beta_{\zeta^{\prime}}$, we next construct directed loops
\[\beta_{\zeta^{(1)}}\colon1_{\hat{K}}\overset{\zeta_{1}^{(1)}}{\to}\sigma^{(1)}_{1}\cdots\overset{\zeta_{m_{1}}^{(1)}}{\to}1_{\hat{K}} 
\quad\text{and}\quad 
\beta_{\zeta^{(2)}}\colon1_{\hat{K}}\overset{\zeta_{1}^{(2)}}{\to}\sigma_{1}^{(2)}\cdots\overset{\zeta_{m_{2}}^{(2)}}{\to}1_{\hat{K}}\]
such that $\zeta_{i}$ appears exactly once in $\beta_{\zeta^{(1)}}$ and $\zeta_{j}^{(1)}\neq\zeta_{j}^{(2)}$ for some $j\leq \min\{m_{1},m_{2}\}$.
(These conditions correspond to the aperiodicity of the word in Cuntz's original proof \cite{Cu77}.)
Define $i_0:= \min\{j: \zeta_{j}=\zeta_i\}, i_1:=\max\{j: \zeta_j=\zeta_i\}$. Note that $i_0 \leq i \leq i_1$.
We define $\beta_{\zeta^{(1)}}$ by removing the edges from the $i_0$-th to the $(i_1-1)$-th positions of $\beta_{\zeta}$. (If $i_0 =i_1$, the original path remains unchanged.)
The resulting directed path $\beta_{\zeta^{(1)}}$ remains a loop based at $1_{\hat{K}}$
in which $\zeta_i$ appears exactly once. Note that $\zeta^{(1)}_{i_0}=\zeta_i$.
We next construct $\zeta^{(2)}$. If $i_0=i$, we simply set $\zeta^{(2)}:=\zeta^{\prime}$.
Otherwise we have $i_0<i$. By the choice of $i$ and $i_0$, one has
\[\sigma' _{i_0-1}= \sigma_{i_0-1}= s(\zeta_i)=\sigma_{i-1}=r(\zeta_{i-1})=\sigma'_{i-1}.\]
Thus, by removing the edges from the $i_0$-th to the $(i-1)$-th positions of $\beta_{\zeta^{\prime}}$,
we obtain a new directed loop, say $\beta_{\zeta^{(2)}}$.
Again by the choice of $i$ and $i_0$, we obtain $\zeta^{(1)}_{i_0}=\zeta_i \neq \zeta^{\prime}_i = \zeta^{(2)}_{i_0}$.
Thus $\beta_{\zeta^{(1)}}$ and $\beta_{\zeta^{(2)}}$ have the desired properties.

Now, by the choice of $\beta_{\zeta^{(1)}}$ and $\beta_{\zeta^{(2)}}$, for $k\in \IN$ with $km_{1}\geq M$,
we obtain
\[(S_{\zeta^{(1)}}^{k}S_{\zeta^{(2)}})^{*}S_{\xi}S_{\zeta^{(1)}}^{k}S_{\zeta^{(2)}}=0\]
for all $\xi\in \bigcup_{m=1}^{M}q\cE_{\pi}^{\otimes m}q$.
For each $\sigma \in V_\pi$,
we denote by $\zeta^{(i)}_{\sigma}\in q_{\sigma}\cE_{\pi}^{\otimes m_{i}}q_{\sigma}$
the image of $\zeta^{(i)}$ by the canonical $G$-equivariant inclusion
\[q\cE_{\pi}^{\otimes m_{i}}q\cong(1_{\hat{K}},\,\pi^{\otimes m_{i}})\to(\sigma,\,\pi^{\otimes m_{i}}\otimes\sigma)\cong q_{\sigma}\cE_{\pi}^{\otimes m_{i}}q_{\sigma}.\]
Note that $\zeta^{(i)}_{\sigma}\in\ONB_{G}(q_{\sigma}\cE_{\pi}^{\otimes m_{i}}q_{\sigma})$.
Pick $\chi_{i}\in\hat{G}$ satisfying
\[(\alpha_{\tilde{\tau}})_{g}(S_{\zeta^{(i)}})=\chi_{i}(g)S_{\zeta^{(i)}}\quad {\rm~for~all~}g\in G,\]
and let 
\[F:=\{\sigma\in V_\pi :q\cE_{\pi}^{\otimes M}q_{\sigma}\neq0\}.\]
Define
\[u:=\sum_{\sigma\in F}\sum_{\zeta\in\ONB_{G}(q\cE_{\pi}^{\otimes M}q_{\sigma})}S_{\zeta}S_{\zeta^{(1)}_{\sigma}}^{k}S_{\zeta^{(2)}_{\sigma}}S_{\zeta}^{*}\in q\Cn{\cE_{\pi}}q.\]
Note that $u$ is a partial isometry element.
In the same way as in the proof of \cite[Lemma 4.6]{MR1943097},
one has
\[u^{*}yu=\hat{\alpha}_{\chi_{1}^{k}\chi_{2}}(E(y)).\]
Note that $u^{*}yu=\hat{\alpha}_{\chi_{1}^{k}\chi_{2}}(E(y))$ is a nonzero positive element in
\[
\overline{\rm span}\{S_{\xi}f S_{\xi^{\prime}}^{*}: f\in\rg(G),\;\xi,\xi^{\prime}\in q\cE_{\pi}^{\otimes M}\}\cong C_{0}(\hat{G})\otimes \IK(q\cE_{\pi}^{\otimes M})
\cong \bigoplus_{\sigma\in V_\pi} (C_{0}(\hat{G})\otimes \IK(q\cE_{\pi}^{\otimes M}q_\sigma)).
\]
Hence,
one can take $\sigma\in F$, $\xi\in q\cE_{\pi}^{\otimes M}q_{\sigma}$ and $\zeta\in\ONB_{G}(q\cE_{\pi}^{\otimes M}q_{\sigma})$ satisfying
\[S_{\zeta}S_{\xi}^{*}u^{*}yuS_{\xi}S_{\zeta}^{*}\in S_{\zeta}S_{\zeta}^{*}\otimes \rg(G),\qquad \|S_{\zeta}S_{\xi}^{*}u^{*}yuS_{\xi}S_{\zeta}^{*}\|=\|E(y)\|.\]
By assumption,
there exists a directed loop based at $1_{\hat{K}}$ in $\cG_{\pi}$.
Hence,
by the cofinality of $\cG_\pi$, there exists a directed path from $\sigma$ to $1_{\hat{K}}$.
Thus, one has a unit vector $\zeta^{\prime}\in\bigcup_{m=1}^{\infty}q_{\sigma}\cE_{\pi}^{\otimes m}q$.
Set
\[a:=S_{\zeta^{\prime}}^{*}S_{\xi}^{*}u^{*}=S_{\zeta^{\prime}}^{*}S_{\zeta}^{*}S_{\zeta}S_{\xi}^{*}u^{*}.\]
Then it satisfies the desired conditions
\[aya^{*}\in q\rg(G), \quad \|aya^{*}\|=\|E(y)\|.\]

Since the above claims have been verified, pure infiniteness of $q(\Cn{\cE_{\pi}}\rca{\alpha_{\tilde{\tau}}} G)q$ follows in the same way as in \cite[Theorem 4.7]{MR1943097}.
\end{proof}
\begin{Cor}\label{cor:purely_infinite}
 Let $G$ be a locally compact abelian group and let $K$ be a compact group. Let $\Pi\colon K\times G \rightarrow \cU(\cH_\Pi)$ be a finite-dimensional unitary representation.
 If $F(\Pi) \sim \lambda$,
 then $\Cn{\cH_{\Pi}}\rca{\alpha_\Pi}(K\times G)$ is purely infinite simple.
\end{Cor}
\begin{proof}
When $K$ is trivial,
the statement follows from \cite[Section 7]{Izuqp}.
Let $\pi:=\Pi|_{K}\colon K\to \IB(\cH_{\Pi})$ be the restriction of $\Pi$.
Then $F(\Pi)\sim \lambda_{K\times G}$ implies that $F(\pi)$ contains every irreducible representation $\sigma\in\hat{K}$.
This implies that the associated graph $\cG_{\pi}$ is cofinal.

Observe that, if $F(\Pi)\sim\lambda_{K\times G}$,
then, as $\Pi$ is finite-dimensional, the set
\[
{\alpha_\upsilon}=\left\{\chi\in\hat{G}:(1_{\hat{K}},\chi)\leq F(\Pi)\right\}
\]
is dense in $\hat{G}$.
Consequently,
the statement follows from Lemma \ref{lem:scaling_element}.
\end{proof}

Finally, we give a characterization of isometrically shift-absorbing quasi-free actions for compact groups.
The new result of this article is the implications (2), (3) $\Rightarrow$ (1),
which characterizes the property in terms of $\pi$.

\begin{Thm}\label{thm_isa}
 Let $K$ be a compact group and let $\pi\colon K\to \cU(\cH_{\pi})$ be a faithful finite-dimensional unitary representation.
 Then the following conditions are equivalent.
\begin{enumerate}\upshape
 \item $F(\pi)$ contains every irreducible representation of $K$.
 \item $\alpha_{\pi}\colon K\acts\Cn{\cH_{\pi}}$ is isometrically shift-absorbing.
 \item $\Cn{\cH_{\pi}}^{K}$ is purely infinite simple.
 \end{enumerate}
\end{Thm}

\begin{proof}
 Since quasi-free actions are minimal,
 the equivalence between (2) and (3) follows from \cite[Theorem 1.2]{Izuqp}.
 The implication (1)$\Rightarrow$(2) follows from \cite[Section 7]{Izuqp}. 
 
 It is enough to show (3)$\Rightarrow$(1).
 Suppose that the fixed point algebra $\Cn{\cH_{\pi}}^{K}$ is purely infinite simple.
 Since $\Cn{\cH_{\pi}}\rca{\alpha_\pi} K$ is a simple \Cs-algebra containing $\Cn{\cH_{\pi}}^{K}$ as a full corner, the \Cs-algebras
 \[\Cn{\cH_{\pi}}^{K}, \quad \Cn{\cH_{\pi}}\rca{\alpha_\pi} K,\quad {\rm C}^{*}(\cG_{\pi})\]
 are mutually stably isomorphic.
 As in \cite[Theorem 2.4]{MR1626528}, 
 if $\cG_{\pi}$ has no directed loops,
 then ${\rm C}^{*}(\cG_{\pi})$ is an AF-algebra.
 Hence, $\cG_\pi$ must contain a directed loop.
 Put 
 \[S:=\{\sigma\in \hat{K}: \sigma\;\text{lies on a directed loop of } \cG_\pi\} (\neq \emptyset).\]
 Set $\mu_S:= \bigoplus_{\sigma \in S} \sigma$.
 By definition, one has $\mu_{S}\otimes\pi\sim\mu_{S}$.
 By Proposition \ref{prop outer},
 one has $1_{\hat{K}}\leq\mu_{S}$,
 which implies $1_{\hat{K}}\in S$.
 Thus one can choose $n\geq1$ with $1_{\hat{K}}\leq\pi^{\otimes n}$.
 This implies that
 \[\overline{\pi}\leq\pi^{\otimes n-1}\leq F(\pi),\]
 whence
 \[\bigoplus_{k,l=1}^{\infty}\pi^{\otimes k}\otimes\overline{\pi}^{\otimes l}\prec F(\pi).\]
 As $\pi$ is faithful, the left-hand side contains every irreducible representation of $K$ (see Remark \ref{rmk faithful} for details).
 Hence $\pi$ satisfies condition (1).
 \end{proof}
\begin{Rem}\label{rmk faithful}
 Let $\pi$ be a unitary representation of a compact group $K$,
 and put $\Pi:=\bigoplus_{k,l=1}^{\infty}\pi^{\otimes k}\otimes\overline{\pi}^{\otimes l}$.
 Then it is well-known that $\pi$ is faithful if and only if $\sigma\leq\Pi$ for every $\sigma\in\hat{K}$.
 Indeed,
 the \Cs-subalgebra of the commutative \Cs-algebra $C(K)$ generated by $\{\langle \pi(\cdot)\xi,\eta\rangle\in C(K):\xi,\eta\in\cH_{\pi}\}$ is equal to
 \[A_{\Pi}:=\csp\left\{\langle \Pi(\cdot)\xi,\eta\rangle\in C(K):\xi,\eta\in\bigoplus_{k,l=1}^{\infty}\cH_{\pi}^{\otimes k}\otimes\cH_{\overline{\pi}}^{\otimes l}\right\}.\]
 Therefore,
 $\pi$ is faithful if and only if
 $A_{\Pi}\subset C(K)$ separates points of $K$.
 The claim follows by the Peter--Weyl theorem and the Stone--Waierstrass theorem.
 \end{Rem}

 \begin{Cor}\label{cor_semisimple}
 If $K$ is a compact semisimple Lie group,
 then every faithful finite-dimensional unitary representation $\pi$ of $K$ satisfies the equivalent conditions in Theorem \ref{thm_isa}.
\end{Cor}

 \begin{proof}
 In general,
 $\pi^{\otimes{\dim}\pi}$ contains a one-dimensional representation ${\bigwedge^{{\dim}\pi}}\pi = \det \circ \pi$, which is trivial in the present case
 due to the semisimplicity of $K$.
 Consequently,
 we obtain $\overline{\pi}\leq\pi^{\otimes({\dim}\pi-1)}$.
 As in the proof of Theorem \ref{thm_isa},
 it follows that $\pi$ satisfies condition (1) in Theorem \ref{thm_isa}.
 \end{proof}
 \begin{Rem}
 Semisimplicity is used only to ensure that $\det \circ \pi =1$.
 \end{Rem}

\subsection*{Acknowledgements}
The authors would like to thank Takeshi Katsura for helpful discussions on graph algebras regarding Theorem \ref{thm_isa}, and Kan Kitamura for valuable comments on outer actions in Section \ref{section:prop outer} and semisimple Lie groups in Corollary \ref{cor_semisimple}.
They also thank Elias Katsoulis and Adam Dor-On for some remarks on the Hao--Ng isomorphism problem.

\subsection*{Fundings}
The first author was supported by JSPS KAKENHI (Grant-in-Aid for Early-Career Scientists) 26K16999.
The second author was supported by JSPS KAKENHI (Grant-in-Aid for Early-Career Scientists) JP22K13924, JSPS KAKENHI Grant-in-Aid for Scientific Research (C) 26K06820,
and the Distinguished Researcher Grant of Hokkaido University.
\subsection*{AI statement}
AI tools were utilized solely for language editing and literature search. All mathematical content, including proofs and their verification, was carried out entirely without the use of AI.


\begin{thebibliography}{99}
\bibitem{MR1777234}
T.~Bates, D.~Pask, I.~Raeburn, and W.~Szyma\'{n}ski,
{\it The \Cs-algebras of row-finite graphs.}
New York J. Math., {\bf 6}:307--324, 2000.
\bibitem{BKQR}E.~ B{\'e}dos, S.~ Kaliszewski, J.~ Quigg, D.~ Robertson,
{\it A new look at crossed
product correspondences and associated \Cs-algebras.} J.~ Math.~ Anal.~App.~ {\bf 426} (2015), no. 2, 1080--1098.

\bibitem{BC}B.~Blackadar, J,~Cuntz, {\it The structure of stable algebraically simple \Cs-algebras.} Amer.~ J.~ Math.~
{\bf 104} (1982), no. 4, 813--822.
\bibitem{MR818932}

O.~Bratteli and D.~E. Evans,
{\it Derivations tangential to compact groups: the nonabelian case.}
Proc. London Math. Soc. (3), {\bf 52}(2):369--384, 1986.
\bibitem{Br} L.~ G.~ Brown, {\it Stable isomorphism of hereditary subalgebras of \Cs-algebras.} Pacific J.~ Math.~ {\bf 71}(1977),
335--348.
\bibitem{BO} N.~ P.~ Brown, N.~Ozawa, {\it \Cs-algebras and finite-dimensional approximations.} Graduate Studies in Mathematics {\bf 88}. American Mathematical Society, Providence, RI, 2008.
\bibitem{CGS} J.~ Carri{\'o}n, J.~Gabe, C.~ Schafhauser, A.~ Tikuisis, S.~ White.
{\it Classifying $\ast$-homomorphisms I: unital simple nuclear \Cs-algebras.}
To appear in Acta Math., arXiv:2307.06480v4.
\bibitem{Cu77} J.~ Cuntz, {\it Simple \Cs -algebras generated by isometries.}
Comm.~ Math.~ Phys.~ {\bf 57}(2)(1977), 173--185.
\bibitem{MR901232}
S.~Doplicher and J.~E. Roberts,
{\it Duals of compact Lie groups realized in the Cuntz algebras and their actions on \Cs-algebras.}
J. Funct. Anal., {\bf 74}(1)(1987), 96--120.

\bibitem{DT} A.~ Dor-On, I.~ Thompson, {\it The Hao--Ng isomorphism theorem for reduced crossed products.} Preprint, arXiv:2505.00587v3.
\bibitem{MR1397028}
G.~B. Folland,
{\em A Course in Abstract Harmonic Analysis.}
Studies in Advanced Mathematics, CRC Press, Boca Raton, FL, 1995.
\bibitem{MR1670363}
N.~J. Fowler, M.~Laca, and I.~Raeburn,
{\it The \Cs-algebras of infinite graphs.}
{\em Proc. Amer. Math. Soc.}, {\bf 128}(8):2319--2327, 2000.
\bibitem{GS2} J.~ Gabe, G.~ Szab\'o, {\it The dynamical Kirchberg--Phillips theorem.} Acta Math.~ {\bf 232} (2024) 1--77.

\bibitem{HN} G.~ Hao, C.-K. Ng, {\it Crossed products of \Cs -correspondences by amenable group
actions.} J.~ Math.~ Anal.~ App.~ {\bf 345} (2008), no. 2, 702--707.
\bibitem{Izu}M.~ Izumi, {\it Inclusions of simple \Cs-algebras.} J.~ Reine Angew.~ Math.~ {\bf 547} (2002), 97--138.
\bibitem{Izuqp}M.~ Izumi, {\it Minimal compact group actions on \Cs-algebras with simple fixed point algebras.}
Rev. Math. Phys., published online, doi:10.1142/S0129055X24610026.


\bibitem{MR2085108}
T.~Kajiwara, C.~Pinzari, and Y.~Watatani,
Jones index theory for Hilbert {$C^*$}-bimodules and its equivalence with conjugation theory,
{\em J. Funct. Anal.}, {\bf 215}(1), 1--49, 2004.
\bibitem{Kas}G,~ Kasparov, {\it Hilbert \Cs-modules: theorems of Stinespring and Voiculescu.} J.~
Operator Theory {\bf 4}(1980), 133--150.
\bibitem{Kat17} E.~ Katsoulis, {\it \Cs-envelopes and the Hao--Ng isomorphism for discrete groups.} Int. Math. Res. Not., IMRN {\bf 2017} (2017), no. 18, 5751--5768
\bibitem{KR19} E.~ Katsoulis, C.~ Ramsey, {\it Crossed products of operator algebras.} Mem.~ Amer.~ Math.~ Soc.~ {\bf 258} (2019), no. 1240, xviii + 85 pp. 
\bibitem{KR} E.~ Katsoulis, C.~ Ramsey, {\it The non-selfadjoint approach to the Hao--Ng isomorphism.} Int. Math. Res. Not., IMRN {\bf 2021} (2021), no. 2, 1160--1197.
\bibitem{MR1943097}
T.~Katsura,
{\it AF-embeddability of crossed products of Cuntz algebras.}
J. Funct. Anal., {\bf 196}(2):427--442, 2002.
\bibitem{Kat}T.~Katsura, {\it On crossed products of the Cuntz algebra ${\mathcal O}_\infty$ by quasi-free actions of abelian groups.} Operator algebras and mathematical physics, 209--233, Theta, Bucharest, (2003).
\bibitem{MR2016248}T.~Katsura, {\it The ideal structures of crossed products of Cuntz algebras by quasi-free actions of abelian groups.}
Canad. J. Math., {\bf 55}(6) (2003), 1302--1338.
\bibitem{Kat04}T.~ Katsura, {\it On \Cs-algebras associated with \Cs-correspondences.} J.~ Funct.~
Anal.~ {\bf 217} (2004), no. 2, 366--401.
\bibitem{KatId}
T.~Katsura, {\it A class of \Cs-algebras generalizing both graph algebras and homeomorphism \Cs-algebras III, ideal structures.} Ergodic Theory Dynam. Systems, {\bf 26} (2006), no. 6, 1805--1854.
\bibitem{Kat08}T.~Katsura, {\it A class of \Cs-algebras generalizing both graph algebras and homeomorphism \Cs-algebras IV, pure infiniteness.} 
J.~ Funct.~ Anal.~ {\bf 254} (2008), no. 5, 1161--1187.
\bibitem{MR623751}
A.~Kishimoto,
{\it Simple crossed products of \Cs-algebras by locally compact abelian groups.}
{Yokohama Math. J.}, {\bf 28}(1-2):69--85, 1980.
\bibitem{Kis}A.~Kishimoto, {\it Outer automorphisms and reduced crossed products of simple \Cs-algebras.} Commun.~ Math.~ Phys.~ {\bf 81} (1981), no. 3, 429--435.

\bibitem{KK} A.~ Kishimoto, A.~ Kumjian, {\it Simple stably projectionless \Cs-algebras arising as crossed products.} Canad.~ J.~ Math.~ {\bf 48} (1996), no. 5, 980--996.
\bibitem{Kit}K.~Kitamura, {\it Actions of tensor categories on Kirchberg algebras.}
Ann. Sci. {\'E}c. Norm. Sup{\'e}r. {\bf 59} (2026), no.2, 273--330.
\bibitem{MR1626528}
A.~Kumjian, D.~Pask, and I.~Raeburn,
Cuntz--Krieger algebras of directed graphs,
{\em Pacific J. Math.}, {\bf 184}(1)(1998), 161--174.
\bibitem{MR1432596}
A.~Kumjian, D.~Pask, I.~Raeburn, and J.~Renault,
{\it Graphs, groupoids, and Cuntz--Krieger algebras.}
{\em J. Funct. Anal.}, {\bf 144}(2)(1997), 505--541.
\bibitem{Kum}A.~ Kumjian, {\it On certain Cuntz--Pimsner algebras.} Pacific J.~ Math.~ {\bf 217} (2004), no. 2, 275--289.

\bibitem{MR3756105}
 B.~K.~Kwa\'{s}niewski, and R.~Meyer,
 {\it Aperiodicity, topological freeness and pure outerness: from group actions to Fell bundles.}
 Studia~ Math.~ {\bf 241} (2018), no. 3, 257--303.

\bibitem{LN}M.~Laca, S.~Neshveyev, {\it KMS states of quasi-free dynamics on Pimsner algebras.} J.~ Funct.~ Anal.~ {\bf 211} (2), 457--482 (2004).
 \bibitem{MS} H.~Matui, Y.~Sato, {\it $\cZ$-stability of crossed products by strongly outer actions.}
Comm.~ Math.~ Phys.~ {\bf 314} (2012), 193--228.

\bibitem{Mey}R.~ Meyer, {\it On the classification of group actions on \Cs-algebras up to equivariant KK-equivalence.} Ann. K-Theory {\bf 6} (2021), 157--238.
\bibitem{MR308795}
D.~Mili\v{c}i\'{c},
{\it Topological representation of the group \Cs-algebra of {${\rm SL}(2,\IR)$}.}
{Glasnik Mat. Ser. III}, {\bf 6}(26):231--246, 1971.
\bibitem{MP} A.~Mingo, J.~Phillips, {\it Equivariant triviality theorems for {H}ilbert \Cs-modules.} 
Proc. Amer. Math. Soc. {\bf 91} (1984), 225--230. 
\bibitem{MR1757049}
P.~S. Muhly and B.~Solel,
{\it On the Morita equivalence of tensor algebras.}
{Proc. London Math. Soc. (3)}, {\bf 81}(1):113--168, 2000.
\bibitem{OP}D.~ Olesen, G.~ K.~ Pedersen, {\it Applications of the Connes spectrum to \Cs-dynamical systems
II.} J.~ Funct.~ Anal., {\bf 36} (1980), no.1,18--32.
\bibitem{OS}N.~ Ozawa, Y.~ Suzuki, {\it On characterizations of amenable \Cs-dynamical systems and new
examples.} Selecta Math.~ (N.S.), {\bf 27} (2021), Article number: 92 (29 pages).

\bibitem{Pedbook}G.~Pedersen, {\it \Cs-algebras and their automorphism groups.} 
Pure and Applied Mathematics, Academic Press, London, 2018. Second edition.

\bibitem{Pim}M.~ V.~ Pimsner, {\it A class of \Cs-algebras generalizing both Cuntz--Krieger algebras and crossed
products by $\IZ$.} Free probability theory, 189--212, Fields Inst. Commun., {\bf 12} (1997), Amer. Math. Soc.,
Providence, RI.

\bibitem{MR509073}
J.~Repka,
{\it Tensor products of unitary representations of {${\rm SL}_2({\IR})$}.}
{Amer. J. Math.}, {\bf 100}(4):747--774, 1978.

\bibitem{Ric}N.~ W.~ Rickert, {\it Amenable groups and groups with the fixed point property.} Trans.~ Amer.~ Math.~ Soc., {\bf 127} (1967), pp. 221--232.

\bibitem{Sch}J.~ Schweizer, {\it Dilations of \Cs-correspondences and the simplicity of Cuntz--Pimsner algebras.} J.~
Funct.~ Anal., {\bf 180}(2) (2001), 404--425.
\bibitem{Suzeq}Y.~Suzuki, {\it Simple equivariant \Cs-algebras whose full and reduced crossed products coincide.}
J.~ Noncommut.~ Geom. {\bf 13} (2019), 1577--1585.

\bibitem{SuzCMP} Y.~ Suzuki, {\it Complete descriptions of intermediate operator algebras by intermediate extensions of dynamical systems.} Commun.~ Math.~ Phys.~ {\bf 375} (2020), 1273--1297.

\bibitem{Suzsf}Y.~Suzuki, {\it Every countable group admits amenable actions on stably finite simple \Cs-algebras.}
Amer. J. Math. {\bf 148} no.1 (2026), 69--77.
\bibitem{SuzP}Y.~Suzuki, {\it Amenable actions on finite simple \Cs-algebras arising from flows on Pimsner algebras.}
To appear in M{\"u}nster J. Math., special issue in honour of Eberhard Kirchberg (invited), arXiv:2305.13056.
\bibitem{SuzMAAN2}Y.~Suzuki, {\it Crossed product splitting of intermediate operator algebras via 2-cocycles.}
Math.~ Ann.~ {\bf 394} (2026), Article: 38, 37 pages.


\bibitem{Thom}K.~Thomsen, {\it An introduction to KMS weights I, II, III.}
Preprint, arXiv:2204.01125v5.

\bibitem{Ur} D.~Ursu, {\it Characterizing traces on crossed products of noncommutative \Cs-algebras.}	Adv.~ in Math.~ {\bf 391} (2021), 107955.
\end{thebibliography}
\end{document}